\documentclass[reqno, 11pt]{amsart}
\usepackage[bbgreekl]{mathbbol}
\usepackage{amsmath,amssymb,amsthm,amsfonts, amscd, url}
\usepackage{appendix, enumerate}
\usepackage[numeric]{amsrefs}
\usepackage{hyperref}
\input xy
\xyoption{all}
\usepackage{array} 

\usepackage{chngcntr}
\counterwithin{figure}{subsection}
\counterwithin{table}{subsection}

\makeatletter
\def\@tocline#1#2#3#4#5#6#7{\relax
  \ifnum #1>\c@tocdepth 
  \else
    \par \addpenalty\@secpenalty\addvspace{#2}%
    \begingroup \hyphenpenalty\@M
    \@ifempty{#4}{%
      \@tempdima\csname r@tocindent\number#1\endcsname\relax
    }{%
      \@tempdima#4\relax
    }%
    \parindent\z@ \leftskip#3\relax \advance\leftskip\@tempdima\relax
    \rightskip\@pnumwidth plus4em \parfillskip-\@pnumwidth
    #5\leavevmode\hskip-\@tempdima
      \ifcase #1
       \or\or \hskip 1em \or \hskip 2em \else \hskip 3em \fi%
      #6\nobreak\relax
    \dotfill\hbox to\@pnumwidth{\@tocpagenum{#7}}\par
    \nobreak
    \endgroup
  \fi}
\makeatother

\usepackage[dvipsnames]{xcolor}
\usepackage{graphics}
\usepackage{graphicx}

\newcommand{\red}[1]{{\color{red}#1}}

\usepackage{mathptmx} 
\usepackage[scaled=0.90]{helvet} 
\usepackage{courier} 
\normalfont
\usepackage[T1]{fontenc}
\newcommand{\comment}[1]{{\iffalse{\red{#1}}\fi}}

\newcommand{\Is}{I}
\newcommand{\Js}{J}
\newcommand{\As}{\mathsf{A}}
\newcommand{\Qs}{\mathsf{Q}}
\newcommand{\Bs}{\mathsf{B}}

\newcommand{\Pv}{\Pi^{\vee}}
\newcommand{\wc}{\widetilde{\circ}}
\newcommand{\dv}{d^{\vee}}

\newcommand{\tn}{\widetilde{n}}
\newcommand{\tN}{\widetilde{N}}
\newcommand{\mD}{\widetilde{\Delta}}

\newcommand{\mIw}{\widetilde{\Iw}}
\newcommand{\mI}{\widetilde{\I}}

\newcommand{\mm}{\widetilde{\mf{m}}}

\newcommand{\Tm}{\widetilde{\mathbf{T}}}
\newcommand{\Tmw}{\widetilde{\mathbf{T}}^{\flat}}
\newcommand{\Pm}{\widetilde{\P}}
\newcommand{\Pmw}{\widetilde{\P}^{\flat}}

\newcommand{\ti}{\mc{X}}

\newcommand{\thqn}{\widetilde{\mc{Q}}_v}
\newcommand{\tqn}{\widetilde{\mc{Q}}_v^{\fin}}

\newcommand{\n}{n}

\newcommand{\tR}{\widetilde{R}}
\newcommand{\tcc}{\widetilde{\cc}}

\newcommand{\tQ}{\widetilde{\mc{Q}}}

\newcommand{\tu}{\tilde{u}}
\newcommand{\ttv}{\tilde{v}}

\newcommand{\tw}{\widetilde{w}}

\newcommand{\ty}{\widetilde{y}}

\newcommand{\tx}{\widetilde{x}}

\newcommand{\tAw}{{\widetilde{A}^{\flat}}}
\newcommand{\tC}{{\widetilde{C}}}
\newcommand{\tCw}{{\widetilde{C}^{\flat}}}

\newcommand{\bmu}{\bbmu}
\DeclareMathOperator{\wf}{W}
\DeclareMathOperator{\Span}{Span}
\DeclareMathOperator{\rk}{rank}
\DeclareMathOperator{\ad}{ad}

\newcommand{\z}{\mathbf{z}}

\DeclareMathOperator{\Av}{Av}

\DeclareMathOperator{\fin}{fin}
\DeclareMathOperator{\Hom}{Hom}
\DeclareMathOperator{\Aut}{Aut}
\DeclareMathOperator{\Supp}{Supp}
\DeclareMathOperator{\Iw}{Iw}
\DeclareMathOperator{\iw}{Iw}

\DeclareMathOperator{\wid}{width}
\DeclareMathOperator{\len}{length}

\DeclareMathOperator{\Ad}{Ad}
\DeclareMathOperator{\chare}{char}
\newcommand{\ve}{\mathbf{v}}

\newcommand{\bv}{b^{\vee}}

\newcommand{\bG}{\mathbf{G}}

\newcommand{\h}{\wt{h}}
\renewcommand{\a}{\gamma}
\renewcommand{\b}{\beta}

\newcommand{\gamv}{\gamma^{\vee}}

\renewcommand{\a}{a}
\newcommand{\s}{\mf{s}}

\newtheorem*{de}{Definition}
\newtheorem*{nthm}{Theorem}
\newtheorem*{nprop}{Proposition}
\newtheorem*{nlem}{Lemma}
\newtheorem*{ncor}{Corollary}
\newtheorem*{nclaim}{Claim}
\newtheorem*{nconj}{Conjecture}

\newtheorem*{nrem}{Remark} 

\theoremstyle{remark}

\newcommand{\be}[1]{\begin{eqnarray} \label{#1}}

\newcommand{\ee}{\end{eqnarray}}

\newcommand{\tpoint}[1]{\subsubsection{#1}}

\newcommand{\spoint}{\subsubsection{}}

\numberwithin{equation}{section}

\DeclareMathOperator{\diag}{diag}

\newcommand{\xv}{\xi^{\vee}}

\newcommand{\gv}{\gamma^{\vee}}

\newcommand{\av}{a^{\vee}}

\newcommand{\unp}{\Ie_{\leq,K}}
\newcommand{\whit}{\mathtt{Wh}}
\DeclareMathOperator{\whi}{Wh}

\newcommand{\bH}{\mathbf{H}}

\newcommand{\tB}{\widetilde{B}}

\newcommand{\aw}{\mathbb{W}}
\newcommand{\dd}{\mathbf{d}}
\newcommand{\cc}{\mathbf{c}}

\newcommand{\I}{I}

\newcommand{\cpro}{\mf{c}}
\newcommand{\cprow}{\mf{c}^{\flat}}

\newcommand{\valu}{\mathsf{val}}

\newcommand{\la}{\langle}
\newcommand{\ra}{\rangle}
\newcommand{\lv}{\Lambda^{\vee}}
\newcommand{\mv}{\mu^{\vee}}

\newcommand{\mf}[1]{\mathfrak{#1}}

\newcommand{\rr}{\rightarrow}
\newcommand{\mc}[1]{\mathcal{#1}}
\newcommand{\wt}[1]{\widetilde{#1}}

\newcommand{\Lv}{\Lambda^{\vee}}
\renewcommand{\lv}{\lambda^{\vee}}

\newcommand{\w}{\wt{w}}

\newcommand{\dw}{\dot{w}}

\newcommand{\zee} {\mathbb{Z}}
\newcommand{\C} {\mathbb{C}}

\renewcommand{\O}{\mathcal{O}}
\newcommand{\qn}{{\mathcal Q}^{\fin}_{v}}
\newcommand{\qnf}{{\mathcal Q}^{\fin}_{v}}

\renewcommand{\c}{\mathbf{c}}
\renewcommand{\b}{\mathbf{b}}
\newcommand{\B}{\mathsf{B}}
\newcommand{\Aw}{A^{\flat}}

\newcommand{\W}{\mathcal{W}}

\newcommand{\resi}{\mathsf{res}}

\newcommand{\ZZ}{\mathbb{Z}}

\newcommand{\gam}{\gamma}

\newcommand{\lam}{\lambda}

\newcommand{\kk}{\kappa}

\newcommand{\dynkradius}{0.065*\dynkstep}
\newcommand{\dynkstep}{0.7cm}
\newcommand{\dynklinewidth}{0.03cm}
\newcommand{\quadscale}{0.9}
\newcommand{\scallg}{0.85}
\newcommand{\scalsm}{0.15}
\newcommand{\scallgd}{0.85}
\newcommand{\scalsmd}{0.15}
\newcommand{\scallgt}{0.83}
\newcommand{\scalsmt}{0.17}
\newcommand{\scallgq}{0.817}
\newcommand{\scalsmq}{0.183}

\newcommand{\scalsmard}{0.9}
\newcommand{\scallgard}{0.1}
\newcommand{\scalsmart}{0.9}
\newcommand{\scallgart}{0.1}
\newcommand{\scalsmarq}{0.9}
\newcommand{\scallgarq}{0.1}

\newcommand{\dynkdot}[2]{
\draw[line width = \dynklinewidth, fill=white] (\dynkstep *#1,\dynkstep *#2) circle (\dynkradius);}

\newcommand{\dynkline}[4]{
\draw[line width = \dynklinewidth] (\dynkstep*\scallg*#1+\dynkstep*\scalsm*#3,\dynkstep*\scallg*#2+\dynkstep*\scalsm*#4) -- (\dynkstep*\scallg*#3+\dynkstep*\scalsm*#1,\dynkstep*\scallg*#4+\dynkstep*\scalsm*#2);}

\newcommand{\dynkcdots}[4]{
\draw[line width = \dynklinewidth] 
(\dynkstep*\scallg*#1+\dynkstep*\scalsm*#3,\dynkstep*\scallg*#2+\dynkstep*\scalsm*#4) -- (\dynkstep*0.7*#1+\dynkstep*0.3*#3,\dynkstep*0.7*#2+\dynkstep*0.3*#4);
\draw[line width = \dynklinewidth] 
(\dynkstep*\scallg*#3+\dynkstep*\scalsm*#1,\dynkstep*\scallg*#4+\dynkstep*\scalsm*#2) -- (\dynkstep*0.7*#3+\dynkstep*0.3*#1,\dynkstep*0.7*#4+\dynkstep*0.3*#2);
\draw[dotted, line width = \dynklinewidth] 
(\dynkstep*0.7*#1+\dynkstep*0.3*#3,\dynkstep*0.7*#2+\dynkstep*0.3*#4) -- (\dynkstep*0.7*#3+\dynkstep*0.3*#1,\dynkstep*0.7*#4+\dynkstep*0.3*#2);
}

\newcommand{\dynkdoubleline}[4]{
\draw[line width = 3*\dynklinewidth,] 
(\dynkstep*\scallgd*#1+\dynkstep*\scalsmd*#3,\dynkstep*\scallgd*#2+\dynkstep*\scalsmd*#4) -- (\dynkstep*\scallgd*#3+\dynkstep*\scalsmd*#1,\dynkstep*\scallgd*#4+\dynkstep*\scalsmd*#2);
\draw[white, line width = 1*\dynklinewidth,] 
(\dynkstep*\scallgd*#1+\dynkstep*\scalsmd*#3,\dynkstep*\scallgd*#2+\dynkstep*\scalsmd*#4) -- (\dynkstep*\scallgd*#3+\dynkstep*\scalsmd*#1,\dynkstep*\scallgd*#4+\dynkstep*\scalsmd*#2);
}

\newcommand{\dynktripleline}[4]{
\draw[line width = 4*\dynklinewidth] 
(\dynkstep*\scallgt*#1+\dynkstep*\scalsmt*#3,\dynkstep*\scallgt*#2+\dynkstep*\scalsmt*#4) -- (\dynkstep*\scallgt*#3+\dynkstep*\scalsmt*#1,\dynkstep*\scallgt*#4+\dynkstep*\scalsmt*#2);
\draw[white, line width = 2*\dynklinewidth] 
(\dynkstep*\scallgt*#1+\dynkstep*\scalsmt*#3,\dynkstep*\scallgt*#2+\dynkstep*\scalsmt*#4) -- (\dynkstep*\scallgt*#3+\dynkstep*\scalsmt*#1,\dynkstep*\scallgt*#4+\dynkstep*\scalsmt*#2);
\draw[line width = \dynklinewidth] 
(\dynkstep*\scallgt*#1+\dynkstep*\scalsmt*#3,\dynkstep*\scallgt*#2+\dynkstep*\scalsmt*#4) -- (\dynkstep*\scallgt*#3+\dynkstep*\scalsmt*#1,\dynkstep*\scallgt*#4+\dynkstep*\scalsmt*#2);}

\newcommand{\dynkquadrupleline}[4]{
\draw[line width = \quadscale*5.5*\dynklinewidth] 
(\dynkstep*\scallgq*#1+\dynkstep*\scalsmq*#3,\dynkstep*\scallgq*#2+\dynkstep*\scalsmq*#4) -- (\dynkstep*\scallgq*#3+\dynkstep*\scalsmq*#1,\dynkstep*\scallgq*#4+\dynkstep*\scalsmq*#2);
\draw[white, line width = \quadscale*3.5*\dynklinewidth] 
(\dynkstep*\scallgq*#1+\dynkstep*\scalsmq*#3,\dynkstep*\scallgq*#2+\dynkstep*\scalsmq*#4) -- (\dynkstep*\scallgq*#3+\dynkstep*\scalsmq*#1,\dynkstep*\scallgq*#4+\dynkstep*\scalsmq*#2);
\draw[line width = \quadscale*2.5*\dynklinewidth] 
(\dynkstep*\scallgq*#1+\dynkstep*\scalsmq*#3,\dynkstep*\scallgq*#2+\dynkstep*\scalsmq*#4) -- (\dynkstep*\scallgq*#3+\dynkstep*\scalsmq*#1,\dynkstep*\scallgq*#4+\dynkstep*\scalsmq*#2);
\draw[white, line width = \quadscale*0.5*\dynklinewidth] 
(\dynkstep*\scallgq*#1+\dynkstep*\scalsmq*#3,\dynkstep*\scallgq*#2+\dynkstep*\scalsmq*#4) -- (\dynkstep*\scallgq*#3+\dynkstep*\scalsmq*#1,\dynkstep*\scallgq*#4+\dynkstep*\scalsmq*#2);}

\newcommand{\dynkarrowheadto}[4]{
\draw[white, line width = 0.5*\dynklinewidth,  decoration={markings,mark=at position 1 with {\arrow[line width = 1*\dynklinewidth, scale = 2, color = black]{>}}},
    postaction={decorate}]
(\dynkstep*0.5*#1+\dynkstep*0.5*#3,\dynkstep*0.5*#2+\dynkstep*0.5*#4) -- (\dynkstep*\scallgard*#3+\dynkstep*\scalsmard*#1,\dynkstep*\scallgard*#4+\dynkstep*\scalsmard*#2);
}

\newcommand{\dynkarrowheadfrom}[4]{
\draw[white, line width = 0.5*\dynklinewidth, decoration={markings,mark=at position 1 with {\arrow[line width = 1*\dynklinewidth, scale = 2, color = black]{>}}},
    postaction={decorate}]
(\dynkstep*0.5*#3+\dynkstep*0.5*#1,\dynkstep*0.5*#4+\dynkstep*0.5*#2) -- (\dynkstep*\scallgard*#1+\dynkstep*\scalsmard*#3,\dynkstep*\scallgard*#2+\dynkstep*\scalsmard*#4);
}

\newcommand{\dynktrarrowheadto}[4]{
\draw[line width = 0.5*\dynklinewidth,  decoration={markings,mark=at position 1 with {\arrow[line width = 1*\dynklinewidth, scale = 2, color = black]{>}}},
    postaction={decorate}]
(\dynkstep*0.5*#1+\dynkstep*0.5*#3,\dynkstep*0.5*#2+\dynkstep*0.5*#4) -- (\dynkstep*\scallgart*#3+\dynkstep*\scalsmart*#1,\dynkstep*\scallgart*#4+\dynkstep*\scalsmart*#2);
}

\newcommand{\dynktrarrowheadfrom}[4]{
\draw[line width = 0.5*\dynklinewidth, decoration={markings,mark=at position 1 with {\arrow[line width = 1*\dynklinewidth, scale = 2, color = black]{>}}},
    postaction={decorate}]
(\dynkstep*0.5*#3+\dynkstep*0.5*#1,\dynkstep*0.5*#4+\dynkstep*0.5*#2) -- (\dynkstep*\scallgart*#1+\dynkstep*\scalsmart*#3,\dynkstep*\scallgart*#2+\dynkstep*\scalsmart*#4);
}

\newcommand{\dynkqarrowheadto}[4]{
\draw[white, line width = 0.1*\dynklinewidth,  decoration={markings,mark=at position 1 with {\arrow[line width = 1.2*\dynklinewidth, scale = 2, color = black]{>}}},
    postaction={decorate}]
(\dynkstep*0.5*#1+\dynkstep*0.5*#3,\dynkstep*0.5*#2+\dynkstep*0.5*#4) -- (\dynkstep*\scallgarq*#3+\dynkstep*\scalsmarq*#1,\dynkstep*\scallgarq*#4+\dynkstep*\scalsmarq*#2);
}

\newcommand{\dynkdoublelineto}[4]{
\dynkdoubleline{#1}{#2}{#3}{#4}
\dynkarrowheadto{#1}{#2}{#3}{#4}
}

\newcommand{\dynkdoublelinefrom}[4]{
\dynkdoubleline{#1}{#2}{#3}{#4}
\dynkarrowheadfrom{#1}{#2}{#3}{#4}
}

\newcommand{\dynkdoublelinetofr}[4]{
\dynkdoubleline{#1}{#2}{#3}{#4}
\dynkarrowheadfrom{#1}{#2}{#3}{#4}
\dynkarrowheadto{#1}{#2}{#3}{#4}
}

\newcommand{\dynktriplelineto}[4]{
\dynktripleline{#1}{#2}{#3}{#4}
\dynktrarrowheadto{#1}{#2}{#3}{#4}
}

\newcommand{\dynktriplelinefrom}[4]{
\dynktripleline{#1}{#2}{#3}{#4}
\dynktrarrowheadfrom{#1}{#2}{#3}{#4}
}

\newcommand{\dynkquadruplelineto}[4]{
\dynkquadrupleline{#1}{#2}{#3}{#4}
\dynkqarrowheadto{#1}{#2}{#3}{#4}
}

\usepackage{graphics}
\usepackage{tikz}
\usetikzlibrary{decorations.pathreplacing}
\usetikzlibrary{shapes.geometric}
\usetikzlibrary{arrows}
\usetikzlibrary{decorations.markings}

\usepackage{subfig}
\usepackage{mathdots}

\begin{document}

%
\address{University of Alberta \\ Department of Mathematical and Statistical Sciences, CAB 632 \\ Edmonton, Alberta T6G 2G1}
\email{patnaik@ualberta.ca, puskas@ualberta.ca}

\newcommand{\tA}{\widetilde{A}}
\newcommand{\Pshi}{\widetilde{\Phi}}
\newcommand{\tG}{\widetilde{G}}
\newcommand{\tK}{\widetilde{K}}
\newcommand{\tU}{\widetilde{U}}
\newcommand{\an}{\mathbf{a}}
\newcommand{\nn}{\mathbf{n}^-}
\newcommand{\tT}{\widetilde{\mathbf{T}}}
\newcommand{\tGg}{\tG^{gen}}

\newcommand{\un}{\mf{u}}
\newcommand{\mq}{\mathbf{\mu}_{q-1}}
\newcommand{\Q}{\mathsf{Q}}
\newcommand{\tH}{\widetilde{H}}

\newcommand{\g}{\mathbf{g}}

\newcommand{\T}{\mathbf{T}}
\newcommand{\mT}{\widetilde{\mathbf{T}}}

\newcommand{\spw}{\widetilde{\mathbf{1}}_{\psi}}
\newcommand{\hs}{\widetilde{\mathbf{1}}}

\newcommand{\tve}{\widetilde{\ve}}

\newcommand{\sx}{\mathsf{x}}
\newcommand{\sh}{\mathsf{h}}
\newcommand{\sw}{\mathsf{w}}
\newcommand{\bh}{\mathbf{h}}
\newcommand{\bx}{\mathbf{x}}
\newcommand{\bw}{\mathbf{w}}
\newcommand{\bE}{\mathcal{E}}
\newcommand{\bA}{\mathbf{A}}
\renewcommand{\sc}{\mathsf{c}}

\newcommand{\tI}{\widetilde{I}}
\newcommand{\tIm}{\widetilde{I}^-}

\newcommand{\Qf}{\mathsf{Q}}
\newcommand{\gf}{\mathbb{g}}

\newcommand{\tL}{\widetilde{\Lambda}}
\newcommand{\tav}{\widetilde{a}^{\vee}}
\newcommand{\ta}{\widetilde{a}}
\newcommand{\tLv}{\widetilde{\Lambda}^{\vee}}
\newcommand{\mult}{S}
\newcommand{\ring}{\C^{\fin}_{v, \gf, \mult}[\Lv]}
\newcommand{\nice}{\C^{\fin}_{v, \gf, \mult}[\tLv]}
\newcommand{\nicer}{\C^{\fin}_{v, \gf}(\widetilde{\Lambda}^{\vee})}
\newcommand{\nuv}{{\nu^{\vee}}}
\newcommand{\nicetwist}{\nicer [\la s, t \ra_\star]}


\newcommand{\rh}[1]{\mathbf{H {#1}}}

\renewcommand{\I}{\mc{I}}
\renewcommand{\Iw}{\mc{I}^{\flat}}
\newcommand{\hQ}{\widehat{\mc{Q}}}
\newcommand{\hqn}{\mc{Q}_v}

\newcommand{\sfQ}{\mathsf{Q}}
\newcommand{\Rv}{R^{\vee}}
\newcommand{\Qv}{Q^{\vee}}

\newcommand{\sfq}{\mathsf{Q}}
\newcommand{\sfB}{\mathsf{B}}
\newcommand{\sfb}{\mathsf{B}}

\newcommand{\cw}{\c^{\flat}}
\newcommand{\qnd}{\qn[W]^{\vee}}
\newcommand{\Tw}{\T^{\flat}}
\newcommand{\mul}{\mathsf{m}}
\newcommand{\ct}{\mathrm{ct}}

\renewcommand{\P}{\mc{P}}
\newcommand{\Pw}{\mc{P}^{\flat}}

\newcommand{\p}{\vec{p}}

\newcommand{\thv}{\theta^{\vee}}

\title{Metaplectic Covers of Kac-Moody groups and Whittaker Functions}
\author{Manish M. Patnaik and Anna Pusk\'{a}s}

\maketitle

\begin{abstract} Starting from some linear algebraic data (a Weyl-group invariant bilinear form) and some arithmetic data (a bilinear Steinberg symbol), we construct a cover of a Kac-Moody group generalizing the work of Matsumoto. Specializing our construction over non-archimedean local fields, for each positive integer $n$ we obtain the notion of $n$-fold metaplectic covers of Kac-Moody groups. In this setting, we prove a Casselman-Shalika type formula for Whittaker functions. \end{abstract}

\tableofcontents

\section{Introduction} 

Let $\As$ be a generalized Cartan matrix and $F$ some field. Using some auxiliary data, one may construct a corresponding Kac-Moody group $G$ over $F.$ Our goals in this paper are two-fold. 

\begin{enumerate}[(i)] 
\item Starting with some linear algebraic data (a Weyl group invariant quadratic form on a maximal torus of $G$) and some arithmetic data on $F$ (a Steinberg symbol), we construct a central extension of $G$\footnote{We work here in the ``simply connected'' case, a notion which will be explained below in \S \ref{s:rd}} following the classical construction of Matsumoto \cite{mat}.  In case $F$ is a non-archimedean local field and if one chooses the $n$-th order Hilbert symbol ($n \geq 1$ an integer), this leads to the notion of a $n$-fold metaplectic cover of $G$ . 

\item We explain how to generalize the Casselman-Shalika formula \cite{ca:sh} for unramified Whittaker functions to our metaplectic covers. The strategy here is as in our previous work \cite{pat:whit}, \cite{pat:pus}, but some new combinatorial work is necessary. Our results are sharpest for affine Kac-Moody groups in a manner that will be explained later in this introduction. 
 
\end{enumerate} 

This work was partly motivated by a desire to understand the Fourier coefficients of Eisenstein series on these new metaplectic Kac-Moody groups and their conjectured link to the theory of multiple Dirichlet series. The results in (ii) above are expected to play the main local tool in such a study. We present our constructions in more detail in \S \ref{s:intro-covers} - \S \ref{s:intro-whit}, make some comments on connection with the existing literature in \S \ref{s:intro-lit}, and then fix some basic notation in \S \ref{s:intro-not}. 


\subsection{Summary of our results}

\tpoint{Construction of covering groups (after Matsumoto \cite{mat})} \label{s:intro-covers} Let $G_o$ be split, simple, and simply connected group with Weyl group $W_o$ defined over some field $F$. Let $\Qs_o$ be a $W_o$-invariant integral quadratic form on the cocharacter lattice of $G_o$ and $(\cdot, \cdot): F^* \times F^* \rr A$ be a Steinberg symbol (cf. \S \ref{stein-symb}) with target some abelian group $A.$  Following the work of Matsumoto \cite{mat} (cf. also the works of Steinberg \cite{stein}, Moore \cite{moore}, and Kubota \cite{kub}), one can attach a covering group to $G_o$ using $\Qs_o$ and $(\cdot, \cdot)$. In the restrictive setting in which we are working, the choice of $\Qs_o$ amounts to just a choice of an integer\footnote{Matsumoto's original construction does not actually have a general $\Qs_o$. He uses what amounts to the unique $\Qs_o$ which takes value $1$ on the short coroots.}, but we still choose to phrase things with this more involved notation following the example of Deligne and Brylinski \cite{del:bry} who construct covers in more general contexts (i.e. non-split, reductive groups) starting from similar data. In any event, we find the use of $\Qs_o$ a useful notational device in the general Kac-Moody context.  Matsumoto's construction proceeds in three steps.

\begin{enumerate}[M1.] 
\item From $\Qs_o$ and the symbol $(\cdot, \cdot),$ one constructs a central extension of the torus $H_o$ by the group $A.$ Furthermore, one defines a family of automorphisms of this cover $\{ \mf{s}_i \}$ indexed by the simple roots of $G_o$ and satisfying the braid relations.

\item Let $N_o$ be the normalizer of the torus $H_o$ and $N_{o, \zee}$ its integral version (sometimes also called the \emph{extended Weyl group}). Using an explicit presentation of $N_{o, \zee}$ (obtained earlier by Tits \cite{tits:norm}) together with the automorphisms $\{ \mf{s}_i \}$ introduced in the previous step, one constructs a central extension $\tN_o$ of $N_o$ that restricts to the previous extension of $H_o.$

\item Using the Bruhat decomposition, one can define the fiber product (of sets) $S_o:= G_o \times_N \tN_o.$ Matsumoto constructs a group of operators $E_o$ acting on $S_o$ whose action is seen to be simply transitive. This fact along with the explicit form of the operators comprising $E_o$ allows one to verify that $E_o$ is a central extension of $G_o$ with kernel $A$ and satisfying a number of desirable properties.
\end{enumerate}

In generalizing the above to a Kac-Moody context, we could perhaps have chosen any of the existing constructions (not all of which are equivalent) for Kac-Moody groups. We choose here the functorial approach of Tits \cite{tits:km} \footnote{At some points though (cf. \eqref{Ua:dec}), we find it convenient to refer to results from Carbone-Garland \cite{car:gar} who construct a group, based on representation theory, which is shown to be a ``homomorphic image'' of the Tits construction.}, whose input is a ``root datum'' consisting of a quadruple $(Y, \{ y_i \}, X, \{ x_i \})$ where $Y, X$ are to the play the role of the cocharacter and character lattice respectively of the group, and $y_i$ and $x_i$ the coroots and roots respectively.  In this setup the ``simply connected" groups will \emph{roughly} be the ones for which $Y$ is spanned by the $\{ y_i \}$, though some further care is required when $\As$ is not of full rank. See \S \ref{s:rd}.

To carry over Matsumoto's strategy, we first need to fix an integral Weyl group invariant form $\Qs$ on $Y.$ In the ``simply-connected'' case just described, these are again easily classified (cf. Proposition \ref{Q-fill}). Starting from the choice of such $\Qs$ and a positive integer $n$, we can also form a ``metaplectic root datum'' as in \cite{macn:ps} or \cite{weis:sp}. We use this new root datum to describe the Casselman-Shalika formula and the metaplectic Demazure-Lusztig operators, but we do not actually need the construction of the metaplectic dual group. In fact, our restrictions on the cardinality of the residue field mask almost all of the subtleties of this dual group which Weissman has constructed in\cite{weis:sp}. In the affine setting, we tabulate the possibilities for metaplectic dual root systems in Table \ref{table:metaplecticrootsystem}.  Fixing $\Qs$ and choosing some Steinberg symbol, Matsumoto's strategy carries over with little change to the Kac-Moody context. 

\begin{enumerate}[KM1.]

\item The construction of the cover of the torus and its family of automorphisms follows as in \cite{mat}. The possible complications involving non-degeneracy of the Cartan matrix (e.g. the loop rotation in the affine case) are conveniently taken care of using the ``$\Qs$- formalism.''  

\item The next step is to obtain a presentation for a group $N$ that plays a role analogous to the normalizer of the torus in the finite-dimensional context, and then refine this to a presentation of the ``integral'' version of this normalizer. We follow the classical arguments of Tits here \cite{tits:cts} adapted with almost no change to the Kac-Moody context. Here we use the simply-connected assumption to obtain a simple presentation for $N$, though it should be possible to remove this assumption.

\item Finally, the construction of a cover $E$ as a group of operators on a set $S$ (again constructed as a fiber product using the cover of $N$ and the Bruhat decompostion) follows as in Matsumoto. Recall also that Matsumoto's argument involves a rank two check, and at first glance it might seem that some new rank 2 Kac-Moody root systems could intervene. However, this is actually not the case (see \S \ref{rank2}). 
\end{enumerate}

After constructing the cover, we also need to verify certain splitting properties over both general fields and over local fields. Using the concrete realization of the cover $E$ as a group of operators on $S,$ one immediately verifies a set of axioms which Tits \cite[\S 5.2]{tits:km} has described \footnote{These generalize the notion of $(B, N)$-pairs, and in fact must incorporate the existence of two (in general non-equivalent) $BN$-pairs which Kac-Moody groups possess.} and which imply the various decompositions one would like the group to have (e.g. Bruhat \emph{and} Birkhoff factorizations). By imposing here some assumptions relating the cardinality of the residue field of our local field and the degree $n$ of the cover, we construct explicitly a splitting of the ``maximal compact'' subgroup of $G$. We do not address any uniqueness questions concerning the splittings we use. 

\tpoint{Whittaker functions and the Casselman-Shalika formula} \label{s:intro-whit} Whereas the construction of the covers of a general Kac-Moody group $G$ does not pose any real technical difficulty, some new work (i.e. not contained in \cite{pat:pus}, \cite{pat:whit}) of a Kac-Moody, \emph{not} metaplectic, nature is required to obtain the Casselman-Shalika formula.  Before describing this, we first note that to define Whittaker functions in the finite-dimensional case, one uses certain Whittaker functionals defined as integrals over unipotent subgroups. Taking advantage of the algebraic nature of these integrals (whose integrands have large groups of invariance), we can rewrite them as sums and it is the latter description which carries over to the infinite-dimensional context, where the notion of integration with respect to a (e.g. Haar) measure is problematic. This was the approach taken in \cite{bk:sph}, and it was used again in \cite{pat:whit} to study unramified Whittaker functions  for (untwisted) affine Kac-Moody groups over a local field. To make sense of this definition, one must still prove ``finiteness'' results to ensure that the sums involved are well-defined. In \emph{op. cit}, we used the main finiteness result of \cite{bgkp} which worked in the context of (untwisted) affine Kac-Moody groups. Recently, H\'ebert \cite{heb} (building on earlier work of Gaussent and Rousseau \cite{gau:rou:sph}) has proven similar finiteness results for general Kac-Moody groups, thereby allowing us to extend the definition of Whittaker function to these groups. The extension of this definition to metaplectic covers is straightforward once some basic structure theory of the group over a local field is established.

As for computing the Whittaker function, the same strategy as in \cite{pat:whit}, \cite{pat:pus} can be used here (\emph{N.B.} a version of this approach for spherical functions, which motivated the Whittaker story, appeared earlier in \cite{bkp}): first one breaks up the Whittaker function into ``Iwahori-Whittaker'' pieces, then one shows that each of these pieces can be expressed via certain Demazure-Lusztig-type operators, and finally one reassembles this sum using some combinatorial identities. For finite dimensional groups, the operators in question were first described by Brubaker-Bump-Licata \cite{bbl} in the non-metaplectic context and their metaplectic analogue, built now from the Weyl group action of Chinta-Gunnells \cite{cg}, appeared in \cite{cgp}. Essentially the same definitions of these operators works in the general Kac-Moody context, and in fact Lee and Zhang \cite{lee:zh} have already considered the extension of the Chinta-Gunnells action to Kac-Moody root systems. Our approach is similar to theirs, though we adopt the ``metaplectic root datum'' framework which perhaps clarifies somewhat the role of imaginary roots (they are just the imaginary roots in the new metaplectic Kac-Moody root system). Next, we note that the same decomposition argument as in \cite{pat:whit}, \cite{pat:pus} expresses the Kac-Moody Whittaker function as a sum of Iwahori-Whittaker pieces. A recursion argument using intertwiners (\cite[Corollary 5.4]{pat:pus}) shows that each of these Iwahori-Whittaker pieces is expressed via the application of a metaplectic Demazure-Lusztig operator. However, in reassembling this sum of Demazure-Lusztig operators a technical complexity arises, as the sum is now over an infinite set (the Kac-Moody Weyl group) and one must contend with certain ``convergence'' issues. Similar issues arose in the affine (non-metaplectic) case of \cite{pat:whit}, but the arguments resolving these issues in \emph{op. cit} seemed restricted to the affine case.  

To explain the issues more concretely, suppose that $g=\pi^{\lv}$ is an element in the (metaplectic) torus where $\pi$ is the uniformizer in our local field, and $\lv$ is some dominant coweight. As in the finite dimensional case, the computation of the Whittaker function $\W$ on $G$ reduces to considerations of its values on these torus elements, and we write $\W(\pi^{\lv})$ and  $\W_w(\pi^{\lv})$ for the values of the Whittaker and Iwahori-Whittaker functions, the latter indexed by an element $w$ in the Weyl group of the Kac-Moody root system. Now, the decomposition mentioned in the previous paragraph is $\W (\pi^{\lv}) = \sum_w \W_w(\pi^{\lv})$, where both sides are \emph{numbers}, they calculate some $p$-adic sums. On the other hand, each $\W_w(\pi^{\lv})$ is given in terms of some Demazure-Lusztig operator. More precisely, there is some operator $T_w$ acting on the coweight lattice and depending on some formal parameters. Then our claim is that if one considers the formal expression $T_w(\lv)$ and specializes the parameters suitably, one recovers the value of the corresponding $p$-adic sum $\W_w(\pi^{\lv})$. However, even though we can specialize each of the $T_w(\lv)$, it remains to be seen that the infinite sum \be{inf:sym-intro} \sum_{w \in W} T_w(\lv) \ee can still be specialized. In the affine case, Cherednik \cite[Lemma 2.19]{cher:ma} proved a polynomiality result which implies this is possible, and we need a Kac-Moody version of this result. We do this in Theorem \ref{form:sym}, and note that our argument (though borrowing certain key ideas from the Cherednik's) is actually different from his even in the affine setting. 

We now briefly explain our proof of the polynomiality of the symmetrizer (\ref{inf:sym-intro}) as this will allow us to describe a second technical complexity. Before showing the symmetrizer is ``polynomial'', we first show (cf. Theorem \ref{form:sym}, (1)) that the symmetrizer exists in some formal completion. Just this fact and some simple algebraic manipulations (cf. \cite{bkp}, \cite{pat:whit}, and \cite{cgp} for the metaplectic variant) show that (\ref{inf:sym-intro}) is proportional to a different sum of operators over the Weyl group\footnote{The right hand side of (\ref{prop-intro}) without the $\mf{m}$ factor is the expression which has already been considered by Lee and Zhang \cite{lee:zh}}, which we shall denote by $\sum_w I_w.$ In other words, we have \be{prop-intro} \sum_w T_w(\lv) = \mf{m} \sum_w I_w(\lv), \ee where $\mf{m}$ is some Weyl group invariant coefficient of proportionality.  The Casselman-Shalika formula for the Whittaker function is just $\mf{m} \, \sum_w I_w(\lv).$ In the finite-dimensional case, the factor is $\mf{m}=1$. In the affine case, one has a precise formula for $\mf{m}$ by virtue of Macdonald's constant term conjecture (resolved by Cherednik), but for a general Kac-Moody context we know essentially nothing about it, except for a certain ``polynomiality'' result of Viswanath \cite[\S 7.1]{vis}. It is in this sense that our results are sharpest for affine root systems-- in general, they are just expressed in terms of this unknown factor $\mf{m}$. Returning to the polynomiality of (\ref{inf:sym-intro})-- it follows from three facts: well-definedness in some completion, the proportionality result \eqref{prop-intro}, and Viswanath's polynomiality result for $\mf{m}.$ 

The second technical complexity can now be described. In fact, the link between $\mf{m}$ and the constant term conjecture of Macdonald (and also Viswanath's result) is not direct. As we indicated before, the same strategy outlined above also works to calculate spherical functions, but one must use slightly different Demazure-Lusztig operators, say $T'_w\, (w \in W)$. One must then consider a sum of the form (\ref{inf:sym-intro}) with each $T_w$ replaced by a $T_w'$ and again show a proportionality result as in (\ref{prop-intro}), where the coefficient of proportionality is now written as $\mf{m}'.$ The $I_w$ must also be modified to some other $I'_w.$ Now, it is $\mf{m}'$ which has a direct link with the work of Macdonald and Viswanath. In the affine case, using an indirect argument \cite{pat:whit}-- comparing asymptotics of the spherical and Whittaker functions at level of the $p$-adic group-- we showed that $\mf{m}=\mf{m}',$ thus allowing us to use the formula of Macdonald-Cherednik. What we observe here is that one can show algebraically, without any recourse to $p$-adic spherical functions, that $\mf{m}= \mf{m}'.$

\tpoint{Relation to Existing Literature} \label{s:intro-lit} In \cite{zhu:weil} Y. Zhu has introduced an analogue of the Weil representation for symplectic loop groups and used it to construct a two-fold cover of such groups. He has further calculated the symbol for his groups (cf. \S 3 \emph{op. cit}), and they are described in terms of the $t$-adic (where $t$ is the \emph{loop} direction) valuation. On the other hand, the symbols for the groups we construct are different, being sensitive to the arithmetic of the field over which the loop group is being defined.  It would be very interesting to understand the link (if any) between these two constructions, especially given the rich applications of Zhu's representation (cf. \cite{gar:zhu1}, \cite{gar:zhu2} ).

In another direction, A. Diaconu and V. Pasol have proposed in a general context \cite{di:pa}, and I. Whitehead has studied in detail for affine root systems \cite{white}, the notion of multiple Dirichlet series having infinite groups of symmetries. These series can be axiomatically characterized and generalize the known constructions from the finite-dimensional theory \cite{wmds}.  Moreover, it is expected that the local part of these series should be connected to the Whittaker functions we consider here, generalizing a link expected (and in many cases known) for finite-dimensional groups.  However, A. Diaconu and I. Whitehead have informed us that there is an interesting discrepancy between their proposal and ours. In fact, they have also computed this difference explicitly in certain cases where they have shown that it is entirely encapsulated by a normalizing factor expressed as an infinite product over imaginary roots.  It would be interesting to investigate these issues further as both of our proposals deviate from the ``naive'' one by different, non-trivial  modifications along the imaginary root directions.

%

\tpoint{Acknowledgements} M.P. and A.P. were supported from the Subbarao Professorship in Number Theory and an NSERC discovery grant. We would like to thank Howard Garland and Yongchang Zhu for some helpful discussions about covers of loop groups, Adrian Diaconu and Ian Whitehead for informative discussions and correspondence about their work on multiple Dirichlet series, and Dinakar Muthiah for various helpful discussions throughout the writing process.

\subsection{Basic Notation} \label{s:intro-not}

\spoint \label{functor} Bold faced objects denote functors (of groups usually) and the corresponding roman letters will denote the field-valued points. For example $\bG$ will be a group-valued functor and $G$ will denote $\bG(F)$ where $F$ is a field, assumed to be specified implicitly in our discussion.  

\spoint \label{p-adic} In this paper $F$ will denote an arbitrary field, and $\mc{K}$ will denote a non-archimedean local field. In the latter case, let $\O \subset \mc{K}$ be the integral subring, and denote the valuation map $\valu: \mc{K}^* \rr \zee.$ Let $\pi \in \O$ be a uniformizing element, and $\kk:= \O / \pi \O$ be the residue field, whose size we denote by $q.$ Denote also by $\varpi: \O \rr \kappa$ the natural quotient map.

\spoint \label{stein-symb} Let $A$ be an abelian group. Then a \emph{bilinear Steinberg symbol} is a map $(\cdot, \cdot): F^* \times F^*  \rr A$ such that 
\begin{enumerate}[(1)]
\item \label{bim} $(\cdot, \cdot)$ is bimultiplicative: $(x, yz)  = (x, y) (x, z) $ and $(xy, z) = (x, z) (y, z)$ 
\item \label{K} $(x, 1-x) = 1$ if $x \neq 1.$
\end{enumerate} 

Let us record here some simple consequences of the above conditions (see \cite{mat}):

\begin{nlem} \label{stein} Let $x, y \in F^*$ and let $(\cdot, \cdot)$ be a bilinear Steinberg symbol. The following identities hold
\begin{enumerate}[(i)]
\item \label{1:x} $(1, x) = (x, 1) =1$

\item \label{pow} For any integer $n$ we have $(x, y)^n=(x, y^n) = (x^n, y)$

\item \label{x:-x}$(x, -x)=1$

\item \label{skew:sym} (Skew-symmetry) $(x, y)^{-1} =  (y, x)$

\item \label{x:x} $(x, x)^2=1$ and so $(x, x)= \{ \pm 1\}$ for any $x \in \mc{K}^*$

\item \label{x:-1}$(x^{-1}, x)=(x^{-1}, -1)$ and also $(x^{-1}, -1)^2=1$

\end{enumerate}
\end{nlem} 

\tpoint{Hilbert symbols} \label{s:HS} For $n > 0$ a positive integer, and denote by $\bmu_n \subset \mc{K}$ the set of $n$-th roots of unity. Assume $(n, \chare \,  \mc{K})=1$ and that $| \bmu_n |=n$. The $n$-th order Hilbert symbol (see e.g. \cite[\S 9.2, 9.3]{ser:lf}) is a bilinear map $( \cdot, \cdot)_n : \mc{K}^* \times \mc{K}^* \rr \bmu_n.$ In the tame case, i.e. $(q, n)=1,$ then  $q \equiv 1 \mod n$ and there is an explicit formula for the Hilbert symbol, \be{tame} (x, y)_n = \varpi((-1)^{ab} y^a/x^b )^{\frac{q-1}{n}} \ee where $a=\valu(x)$ and $b=\valu(y)$. As $n$ is fixed throughout our paper, we often drop it from our notation. Note that $(\cdot, \cdot)$ is a bilinear Steinberg symbol (cf. \cite[Chapter V, Proposition 3.2]{neu}), and in the tame case the above formula \eqref{tame} shows that $(\cdot, \cdot)$ is unramified, i.e. $(x, y)=1$ if $x, y \in \O^*.$ 

To avoid certain sign issues, we sometimes make the stronger assumption that $q \equiv 1 \mod{2n}.$ Under this assumption, we have $(-1, -1)=(-1, x)=1$ for $x \in \mc{K}^*$, and also  \be{q:2} (\pi, \pi)=1 \text{ and } (\pi, u) = \varpi(u)^{\frac{q-1}{n}} \text{ for } u \in \mc{O}^*. \ee Though simplifying our formulas, this assumption that $q \equiv 1 \mod 2n$ has a rather drastic effect on the metaplectic $L$-groups of Weissman (cf. \cite[\S 4.4.3]{weis:sp}).

\tpoint{Gauss sums} \label{s:GS} Let $\psi: \mc{K} \rr \C^*$ be an additive character with conductor $\O,$ i.e., $\psi$ is trivial on $\O$ and non-trivial on $\pi^{-1}\O,$  and let $\sigma: \O^* \rr \C^*$ be a multiplicative character. Then define \be{gauss:gen} \g(\sigma, \tau) = \int_{\O^*} \sigma(u') \tau(u') du' \ee to be a \emph{Gauss sum}, where $du'$ is the Haar measure on $\mc{K}$ giving $\O^*$ volume $q-1.$ In case  \be{sig:tau} \begin{array}{lcr} \sigma(u) = (u, \pi)_n^{-k} & \text{ and } & \tau(u) = \psi(- \pi^{-1} u), \end{array} \text{ for } u \in \mc{K}^* \ee we denote the corresponding Gauss sum by $\g_k$, and note that we have \cite{neu} \be{g:sum} \g_k = \g_l \text{ if } n | k-l, \, \g_0 = -1, \text{ and if } k \neq 0 \mod n, \text{ then } \g_k \g_{-k} = q, \ee where for the last equality we must again assume that $q \equiv 1 \mod 2n$.

\tpoint{$p$-adic specialization} \label{p-spe} We shall introduce formal parameters $v, \gf_k \, ( k\in \zee)$ when discussing the Chinta-Gunnells action in \S \ref{cg-not}. To make the link to Whittaker integrals (sums), we use the  specialization, \be{p-spe:1} v=q^{-1},\ \gf_i = \g_i \text{ and with } q \equiv 1 \mod 2n. \ee
We refer to this as the ``$p$-adic specialization''  from now on.

\section{Kac-Moody Root systems and Metaplectic Structures} \label{section:notation}

\subsection{Cartan Data, Weyl groups, Root data, Lie algebras} 

\tpoint{Generalized Cartan matrices} \label{gcms} 
Fix a finite set $\Is$ whose cardinality is denoted $r=|\Is|$. A square matrix $\As= (a_{ij})_{i, j \in \Is}$ with integral entries is called a \emph{Generalized Cartan Matrix} (gcm) if $a_{ii}=2$, for $i \in \Is$; $a_{ij} \leq 0$ for $i \neq j$; and $a_{ij} =0$ implies that $a_{ji} =0.$ Such a matrix is said to be \emph{symmetrizable} if it admits a decomposition $\As= D \cdot B$ where $D$ is a diagonal matrix with positive, rational entries and $B$ is a symmetric matrix. In this case, $B$ is sometimes called a (rational) symmetrization of $\As$, and we write $D= \diag(\epsilon_1, \ldots, \epsilon_{|\Is|})$ where $\epsilon_i \in \mathbb{Q}\, (i \in \Is)$.  The classification of indecomposable (i.e the associated Dynkin diagram is connected) gcms begins with the following,

\begin{nprop}\label{indec-class} \cite[Theorem 4.3, Corollary 4.3, Lemma 4.5]{kac} Let $\As$ be an indecomposable gcm. Then exactly one of the following three conditions hold, \begin{enumerate}[(1)]
\item All principal minors\footnote{The associated principal minors of $\As$ are the determinants of the matrices $\As_J=(a_{ij})_{i, j \in J}$, where $J \subset \Is$ is any subset.}  of $\As$ are positive. In this case $\As$ is said to be of finite type 
\item There is a vector $\delta\in \mathbb{Q}_{> 0}^{|I|}$ such that $\As \delta=0$. In fact $\delta$ is unique up to a constant factor. In this case $\As$ is said to be of affine type 
\item In all other cases, $\As$ is said to be of indefinite type.
\end{enumerate} \end{nprop}

In the case when $\As$ is of affine type, there is a further classification into \emph{twisted} and \emph{untwisted} types. The corresponding Dynkin diagrams are displayed in Figures \ref{figure:untwistedaffineDynkins} and \ref{figure:twistedaffineDynkins} (the labelling convention for the nodes will be described in \S \ref{sec:affine} below).

\tpoint{Cartan Data} \label{s:cd} A Cartan datum is a pair $(\Is, \cdot)$ where $\Is$ is a finite set and $\cdot$ is a symmetric $\zee$-valued bilinear form on the free abelian group $\zee[\Is]$ satisfying,
 \begin{enumerate}[(1)] 
\item $i \cdot i \in \{ 2, 4, 6, \ldots \}$ for $i \in \Is$
\item $2 \frac{ i \cdot j}{i \cdot i} \in \{ 0, -1, -2, \ldots \}$ for $i, j \in \Is$, $i \neq j$
\end{enumerate}Given a Cartan datum $(\Is, \cdot)$, the matrix $\As= (a_{ij})_{i,j\in \Is}$ where \be{car:gcm} a_{ij} = 2 \frac{i \cdot j}{i \cdot i} \ee is a gcm. We say $(\Is, \cdot)$ is of finite, affine, or indefinite type to mean ``the associated gcm to $(\Is, \cdot)$ is of finite, affine, or indefinite type.'' Note that the type of $(\Is, \cdot)$ can also be checked by applying the criterion of Proposition \ref{indec-class} to the matrix of the bilinear form defining the Cartan datum $(\Is, \cdot)$, i.e. $L=(\ell_{ij})$ where $\ell_{ij} = i \cdot j.$ Indeed, each row of $L$ is a rescaling of the corresponding row of $\As$ by a positive, rational number.

\tpoint{Braid and Weyl groups}\label{tpoint:BraidGroups} Given a Cartan datum $(\Is, \cdot)$ with associated gcm $\As$ we define (possibly infinite) integers $h_{i j}$  for $i, j \in \Is$ according the following rules. \begin{enumerate}[(1)]
\item If $(i \cdot i) (j \cdot j) - (i \cdot j)^2 > 0$ (equivalently, $a_{ij}a_{ji}<4$) then $h_{i j}$ is defined by the equation \be{hij:1} \cos^2 \frac{\pi}{h_{ij}} = \frac{ i \cdot j }{i \cdot i} \, \frac{ j \cdot i }{j \cdot j} = \frac{a_{ij} \, a_{ji}} {4}. \ee We observe that $h_{ij} = h_{ji},$ and tabulate the possibilities (for finite $h_{ij}$) in Table \ref{h:tab}. 

\begin{table}[h!]
\centering
\caption{Relations for Coxeter groups}
\label{h:tab}
\begin{tabular}{l|llll}
$h_{ij}$         & 2 & 3 & 4 & 6 \\ \hline
$a_{ij} a_{ji}$ & 0 & 1 & 2 & 3 
\end{tabular}
\end{table}

\item If $(i \cdot i) (j \cdot j) - (i \cdot j)^2 \leq 0$, we set $h_{i j} = h_{j i} = \infty$.
\end{enumerate} The \emph{braid group} associated to $(\Is, \cdot)$, which we denote by $B(\Is, \cdot),$ is the free group generated by symbols $s_i \, (i \in I)$ equipped with relations \be{bd:rel} \underbrace{s_i s_j s_i \cdots }_{h_{ij}} = \underbrace{s_j s_i s_j \cdots }_{h_{ij}} \text{ for } i \neq j  \ee where both sides have $h_{ij} < \infty$ terms. If we further impose the relation $s_i^2=1$ for all $i \in \Is$ we obtain the \emph{Weyl group} $W(\Is, \cdot)$. Note that both $B(I, \cdot)$ and $W(I, \cdot)$ only depend on the associated gcm $\As$ so we often just write these as $B(\As)$ or $W(\As).$

\tpoint{Coxeter groups}\label{tpoint:Coxeter} The pair $(W(\As), S)$ where $S=\{ s_i,\ i \in \Is \}$ described in the previous paragraph forms a Coxeter system (see \cite[Ch. IV, \S 1.3, D\'efinition 3]{bour}). Note that every element $s \in S$ satisfies $s^2=1$ (i.e. $S^{-1}=S$) so words in $S$ are just products of elements from $S.$ We refer to \cite[Ch. IV]{bour} for the definitions of reduced expressions, the length function $\ell: W \rr \zee$, etc. The following is an easy consequence of \cite[Ch. IV, \S1.5, Proposition 4 and Lemma 4]{bour}.

\begin{nlem}\label{Weyl-moves} Any word in $S$ can be transferred to its reduced expression in $W$ by a sequence of the following : \begin{enumerate} \item[$(E_1)$] Delete a consecutive subword of the form $s \, s$ where $s \in S$
\item[$(E_2)$] Replace a consecutive subword as in the left hand side of \eqref{bd:rel} with the right hand side of \eqref{bd:rel}. \end{enumerate} \end{nlem}

\tpoint{Kac-Moody algebras} \label{s:la} Let $\As$ be a \emph{symmetrizable}\footnote{In the general case, one has to quotient out the algebra we construct below by a certain ideal, which is zero in the symmetrizable case}  gcm and let $r= |I |$. Then one can attach a Lie algebra $\mf{g}(\As)$ to this data. Let us review a few points about the construction and establish some basic notation. We refer to \cite{gar:lep} for further details. Note that what we call $\mf{g}(\As)$ here is what is referred to in \emph{op. cit} as $\mf{g}^e(\As)$ (it is sometimes called the extended Kac-Moody algebra).

\begin{enumerate}[(1)]

\item Define $\mf{g}'(\As)$ as the Lie algebra over $\C$ with $3\, r$ generators $e_i, f_i, h_i \, (i \in \Is)$ subject to: \be{serre-pres} \begin{array}{lcccr} [h_i, h_j]=0, & [e_i, f_j] = \delta_{ij} h_i, & [ h_i, e_j] = a_{ij} e_j, & \ [h_i, f_j] = - a_{ij} f_j & \text{ for } i, j \in \Is  \\ 
(\ad\, e_i)^{a_{ij}+1}  e_j = 0, &(\ad\, f_i)^{a_{ij}+1}  f_j = 0 & & \text{ for } i, j \in \Is, \, i \neq j.    \end{array} \ee For each $r$-tuple $(n_1, \ldots, n_r)$ of non-negative (resp. non-positive integers), define $\mf{g}'(n_1, \ldots, n_r)$ to be the space spanned by the elements \be{e-comm} [e_{i_1}, [ e_{i_2}, \ldots, [e_{i_{r-1}}, e_{i_r} ] \ldots ] \, ] \ \ (resp., \,  [f_{i_1}, [ f_{i_2}, \ldots, [f_{i_{r-1}}, f_{i_r} ] \ldots ] \, ] \ee where each $e_j$ (resp. $f_j$) occurs $| n_j |$-times in the above expressions. Set $\mf{g}'(0, \ldots, 0):= \mf{h}'$ the linear span of $h_i \, (i \in \Is)$ which is seen to be an abelian Lie algebra of dimension $r.$ One then has \be{g: dec} \mf{g}'(\As) = \bigoplus_{(n_1, \ldots, n_r) \in \zee^{r}} \mf{g}(n_1, \ldots, n_r) \ee where we make the convention that $\mf{g}'(n_1, \ldots, n_r)=0$ for all tuples which contain both positive and negative integers. Then we define derivations $\dd_i \, (i \in \Is)$ of $\mf{g}'(A)$ by requiring $\dd_i$ act as the scalar $n_i$ on $\mf{g}'(n_1, \ldots, n_r).$ Let $\mf{d}_0$ be the vector space spanned by the commuting derivations $\dd_i \, (i \in \Is)$ of $\mf{g}'(\As)$. For $\mf{d} \subset \mf{d}_0$ be some subspace, we define the Lie algebra \be{g-ext} \mf{g}(\As) = \mf{g}'(\As) \rtimes \mf{d}. \ee The choice  $\mf{d}$ to choose will (in cases in which it matter) be specified. For the moment, we only impose the following condition. To state it let $\mf{h}:= \mf{h}' \oplus \mf{d}$, and define $a_i \in \mf{h}^{\ast} \, (i \in \Is)$ by specifying $a_i(h) = [h, e_i]$ for all $h \in \mf{h}$. Then the condition we impose on $\mf{d}$ is that the $\{ a_i \}_{i \in \Is}$ are linearly independent. If $\As$ is non-degenerate, we can choose $\mf{d}=0$, but our requirement forces it to be of dimension at least $|\Is| - \rk(\As).$ 

\item  \label{gen:sym} We have assumed $\As$ is symmetrizable. Fix a symmetrization $\As=DB$ as in \S \ref{gcms} with 
$D= \diag(\epsilon_1, \ldots, \epsilon_{|\Is|}).$ Then define elements $\av_i:= h_i/\epsilon_i\, (i \in \Is)$ and define a bilinear form $(\cdot, \cdot)$ on $\mf{h}:$ \be{sym:frm} (\av_i, \av_j ) = b_{ij}  \epsilon_i \epsilon_j  \text{ for } i, j \in \Is \text{ and }  (v,w) = 0 \text{ for } v, w \in \mf{d}. \ee Then $(\cdot, \cdot)$ is non-degenerate (cf. \cite[Lemma 2.1b]{kac}), and so induces an isomorphism $\nu: \mf{h} \rr \mf{h}^{\ast};$ denote by $(\cdot, \cdot)$ the induced bilinear form on $\mf{h}^{\ast}.$ Under this isomorphism $\nu(\av_i) = \epsilon_i a_i$ and \be{in:aij} (a_i, a_j) = a_{ij}/\epsilon_i. \ee

\item For each $\varphi \in \mf{h}^*$ let $\mf{g}(\As)^{\varphi}:= \{ x \in \mf{g}(\As) \mid [ h, x] = \varphi(h) x \text{ for all } h \in \mf{h} \}.$ Let $R$ denote the set of all $\varphi$ such that $\mf{g}^{\varphi} \neq 0$, which we call the roots of $\mf{g}(\As)$. Note that $\Pi:= \{ a_1, \ldots, a_r \} \subset R$, and $\Pi$ is called the set of simple roots, and in fact every element in $R$ is written as a linear combination of elements from $\Pi$ with either all non-negative or non-positive coefficients.  The $\zee$-module spanned by $\Pi$ will be denoted by $Q$ and called the root lattice. One defines the notion of positive elements $Q_+$ as those which are non-negative integral linear combinations of $\Pi$, and hence we obtain also a notion of positive and negative roots $R_+$ and $R_-$. If $a \in R$ we often just write $a >0$ (or $a<0$) to mean $a \in R_+$ (or $a\in R_-$). We also define the set of simple coroots as $\Pi^{\vee}:= \{ \av_1, \ldots, \av_r \}$ and define the coroot lattice $\Qv$ as their $\zee$-span. The notion of $\Qv_+, \Qv_-$ is also defined analogously to what is written above. Finally if $\varphi \in R$ is written as $\varphi = \sum_{i \in \Is} c_i a_i$, then we set $\varphi^{\vee}:= \sum_{i \in \Is} a_i \av_i$. In this way, we obtain the set of coroots $R^{\vee}$ which has a base $\Pi^{\vee}.$

\item We denote by $\la \cdot, \cdot \ra: \mf{h} \times \mf{h}^{\ast} \rr \C$ the dual pairing. There is an action of $W(\As)$ on $\mf{h}$ (and also $\mf{h}^{\ast}$) using the usual formulas: for $h \in \mf{h}, h' \in \mf{h}^{\ast}.$ \be{w:act:h} \begin{array}{lcr} s_i (h) = h - \la \av_i, h \ra a_i &\text{ and } & s_i (h') = h' - \la h', a_i \ra \av_i \end{array}. \ee  The roots $R$ are partitioned into the set of $R_{re}$ of real roots, which are roots in the $W(\As)$-orbit of the simple roots $\Pi$, and the set $R_{im}$ of imaginary roots, which are fixed by $W(\As).$ Similarly we partition $R^{\vee}$ into the real coroots $R^{\vee}_{re}$ and the imaginary ones $R^{\vee}_{im}.$  For each $a \in R$ we define the multiplicity of the root $a$ as the integer $\mul(a):= \dim_{\C} \mf{g}^a,$ and record here that if $a \in R_{re}$ then $\mul(a)=1$. It is also known that the form $(\cdot, \cdot)$ on $\mf{h}$ constructed above is $W$-invariant (cf. \cite[Proposition 2.10]{gar:lep}).

\item We can construct automorphisms of $\mf{g}(\As)$ for each $i \in \Is$ using the following expression (cf \cite[\S 3.6]{kac}) \be{s:aut} s_i^{\ast}:= \exp \ad e_i \, \exp \ad f_i \exp \ad e_i. \ee Let $W^{\ast} \subset \Aut(\mf{g})$ be the subgroup generated by $s_i^* \,(i \in \Is).$ The map $\upsilon: W^* \rr W, s_i^* \mapsto s_i$ is a homomorphism and in fact $s_i^*|_{\mf{h}}= s_i.$ For each $a \in R_{re}$ there exists $w^* \in E^*$ and $i \in \Is$ such that $w^*(a_i)= a$. We define the \emph{dual bases} for each $a \in R_{re}$ as  
\be{Ea} E_{a}:= w^*\{ e_i, -e_i \}, \ee where we note that from \cite[(3.3.2) and subsequent remarks]{tits:km} this definition depends only on $a$ and not on the choice of $i$ or $w^*$. The two element sets $E_a$  will play an important role later when deciding certain signs which arise from the action of the Weyl group on unipotent subgroups.   

\newcommand{\m}{\mathsf{m}}

\item An element $\lambda \in \mf{h}^*$ is called a weight if $\lambda(\av_i) \in \zee$ for $i \in \Is$ and if $\lambda(d) \in \zee$ for $d \in \mf{d}.$ A weight is called dominant if $\lambda(\av_i) \geq 0$ for $i \in \Is$. We define the fundamental weights $\omega_i \, (i \in \Is)$ by requiring $\omega_i(\av_j) = \delta_{ij}$ and $\omega_i|_{\mf{d}}=0$. Also, we set \be{rho:def} \rho = \sum_i \omega_i \ee so that $\la \rho, \av_i \ra =1$ for $i \in \Is$.  We define the dominance order $\leq$  as follows: $\lambda \leq \mu$ if $\mu - \lambda \in Q_+.$

\renewcommand{\mult}{\mathsf{m}}

\end{enumerate}

\newcommand{\D}{\mf{D}}

\tpoint{Root datum}\label{s:rd} A \emph{root datum} of type $(\Is, \cdot)$ will be a quadruple $\mf{D}= (Y,\  \{y_i \}_{i \in \Is},\ X,\ \{ x_i \}_{i \in \Is}),$ such that 
\begin{enumerate}[(1)]
\item $Y, X$ are free $\zee$-modules of finite rank equipped with a perfect pairing $\la \cdot, \cdot \ra: Y \times X \rr \zee,$ i.e. \be{duals} Y = \Hom_{\zee}(X, \zee) & \text{ and } & X = \Hom_{\zee}(Y, \zee). \ee 
\item For each $i \in \Is$, $y_i \in Y$ and $x_i \in X.$ These elements are linearly independent (over $\zee$) and \be{cd:1}  \la y_i, x_j \ra = 2 \frac{i \cdot j}{i \cdot i} = a_{ij} \ee
\end{enumerate} We shall often refer to the collection $(I, \cdot, \mf{D})$  or $(\As, \mf{D})$ as a root datum.\footnote{Again, note that the definition of a root datum of type $(\Is, \cdot)$ is only sensitive to the associated gcm $\As.$}  The \emph{dimension} of a root datum will be the rank of $Y$.

Given a root datum we write $Q_{\mf{D}} \subset X$ and $\Qv_{\mf{D}} \subset Y$ to be the $\zee$-span of $\{ x_i \}_{i \in \Is}$ and $\{ y_i \}_{i \in \Is}$ respectively. One can also define the \emph{real} roots (coroots) of a root datum as $R_{re, \D}$ as the $W(\As)$-orbit of $\{ x_i \}$ (resp. $\{ y_i \}$). 

Consider also the maps $\Omega_i\, (i \in \Is)$ from $Y \rr \zee$ which are defined by $\Omega_i(y_j) = \delta_{ij}$ and $\Omega_i(y)=0$ if $y \notin \Qv_{\D}$ and let $\Lambda_{\mf{D}}:= \Span_{\zee} \{ \Omega_i \}_{i \in \Is}.$ Similarly we define $\Omega^{\vee}_i \, (i \in \Is)$ and $\Lv_{\D}.$ 

We say that $\mf{D}$ is a \emph{simply-connected} if $\Omega_i \in X$ for all $i \in \Is$. In this case \be{Y:sc} Y= \Qv_{\mf{D}}  \oplus Y_{0}, \text{ where } Y_{0}= \cap_{i \in \Is} \ker(\Omega_i: Y \rr \zee) \ee and regarding $y \in Y$ as the map $X \rr \zee, \, x \mapsto \la y, x \ra,$ \be{X:sc} X = \Lambda_{\mf{D}} \oplus X_0 \text{ where } X_0= \cap_{i \in \Is} \ker(y_i: X \rr \zee). \ee 

The action of $W(\As)$ on $Y$ and $X$ is implemented via the same formulas as in (\ref{w:act:h}) with $\av_i$ and $a_i$ replaced by $y_i$ and $x_i$ respectively. 

\spoint \label{s:gD} Let $\D= ( Y, \{ y_i \} , X, \{ x_i \})$ be a given root datum with gcm $\As,$ and define $\mf{g}_{\D}$ to be the  Lie algebra constructed as follows (the notation is as in the previous paragraph). \begin{enumerate}\item First construct $\mf{g}'(\As) \rtimes \mf{d}_0$ as in \S \ref{s:la} (1). 
\item Identify the coroots $\av_i$ with $y_i$, the roots $a_i$ with $x_i.$ Then $\Qv_{\D}$ and $Q_{\D}$ can be identified with $\Qv$ and $Q$ respectively, and $R_{re, \D}$ and $R^{\vee}_{re, \D}$ can be identified with $R_{re}$ and $R_{re}^{\vee}$ respectively. 
\item Writing $\mf{h}_{\D} = Y \otimes_{\zee} \C$, our identification in (2) shows $\mf{h}' \subset \mf{h}_{\D}.$ Each element $y \in Y$ can be regarded as scalar operator on $\mf{g}(\As)^a\, (a \in R)$ acting by $\la d, a \ra.$  Letting $\mf{d}_{\D} := Y_0 \otimes_{\zee} \C,$ we can then identify each $\mf{d}_{\D} \subset \mf{d}_0.$
\item Finally we set $\mf{g}_{\D}:= \mf{g}'(\As) \rtimes \mf{d}_{\D}$ which has Cartan subalgebra $\mf{h}' \oplus \mf{d}_{\D}.$ Note that if we require the $a_i (i \in \Is)$ (i.e., the $x_i\, (i \in \Is)$ ) to be linearly independent, we must have $\dim_{\C} \mf{d}_{\D} \geq | \Is | - \rk(\As).$ 

\end{enumerate}

\tpoint{Duality} \label{s:duality} Fix a Cartan datum $(\Is, \cdot)$ and define the dual Cartan datum $(\Is, \cdot^t)$ where $i \, \cdot^t j:= j \cdot i.$ If $\As$ is the gcm associated to $(\Is, \cdot)$ then one can verify that the gcm associated to $(\Is, \cdot^t)$ is just $\As^t,$ the transpose of $\As.$ If we moreover have a root datum $(\Is, \cdot, \mf{D}$) with $\mf{D}= (Y,\ \{ y_i \},\ X,\ \{ x_i \}),$ then we can define the \emph{dual Cartan datum} as $(I, \cdot^t, \mf{D}^t)$ where $\mf{D}^t = (X,\ \{ x_i \},\ Y,\ \{ y_i \}).$ One verifies the axioms easily in this case, as well as the fact that duality preserves the trichotomy of Proposition \ref{indec-class}. Note that the duals of untwisted affine types could however be twisted (see Table \ref{table:metaplecticrootsystem}).

\tpoint{Inversion sets}\label{tpoint:inversion_sets} For each $w \in W,$ we may consider the inversion sets \be{inv-set} \begin{array}{lcr} R(w):= \{ a \in R_+ \mid w ^{-1} a \in R_- \} & \text{ and } & R_{-}(w):= w^{-1} R(w)= \{ a \in R_- \mid w a > 0 \} \end{array}. \ee If $w$ is written as a reduced word $w= s_{k_1} \cdots s_{k_r}$ then it is well known (cf. \cite[Ch. VI, 6, Corollary 2]{bour}) that $R(w)$ is enumerated by the following elements (which is independent of the reduced decomposition):   \be{eq:list_cterms}
\beta_1:=a_{k_{1}},\ \beta_{2}:=s_{k_1}(a_{k_2}),\ \beta_{3}:=s_{k_1}s_{k_2}(a_{k_3}), \ldots, \, \beta_r=s_{k_1}s_{k_2} \, \cdots \, s_{k_{r-1}} ( a_{k_r} ). \ee Moreover, we record here the identity (\cite[Proposition 2.5]{gar:lep}) \be{rho:flip} \rho - w \rho = \sum_{\beta \in R(w)} \beta. \ee In the proof of Lemma \ref{large-len}, we shall need the following probably well-known generalization of the above. As we could not find a proof in the literature, we supply the easy argument below. \begin{nlem}\label{lem:inverseset} Consider the not necessarily reduced product of simple reflections \be{whatform} w=s_{k_1}s_{k_2}\cdots s_{k_r} \ee and let $\hat{w}$ be the reduced word in $W$ corresponding to (\ref{whatform}). The set $\{\beta_1,\beta_2,\ldots , \beta_r\}$ defined by 
\begin{equation}\label{eq:list_cterms}
\beta_1:=a_{k_{1}},\ \beta_{2}:=s_{k_1}(a_{k_2}),\ \beta_{3}:=s_{k_1}s_{k_2}(a_{k_3}), \ldots \beta_r=s_{k_1}s_{k_2}\cdots s_{k_{r-1}} ( a_{k_r} )
\end{equation}
is then the union of $R(\hat{w})$ and some sets of the form $\{a,-a\}$ for $a\in R_+.$
\end{nlem}

\begin{proof} By Lemma \ref{Weyl-moves}, the product in \eqref{whatform} can be built from a reduced word $\hat{w}$ by the repeated application of the following two moves:
\begin{enumerate}
\item[(A)] Change the sequence $\{k_1,\ldots ,k_r\}$ by replacing elements according to the braid relation \eqref{bd:rel}; i.e. replacing the $h:=h_{ij}$ elements $k_{t}=i,$ $k_{t+1}=j,$ $k_{t+2}=i,$ $\ldots $ by the $h$ elements $k_{t}=j,$ $k_{t+1}=i,$ $k_{t+2}=j,$ $\ldots$ 

\item[(B)] Add a pair $k_p=k_{p+1}=i$ somewhere in the sequence. 
\end{enumerate}
It suffices to show that the operation (A) does not change the set in \eqref{eq:list_cterms}, while the operation (B) adds a pair $\{a,-a\}$ for some $a\in R_+.$
To see the effect of operation (A), observe that when $\hat{w_0}:=s_is_js_i\cdots =s_js_is_j\cdots $ ($h$ factors on both sides) is a braid relation, then $\hat{w_0}$ is a reduced word and the above provide two distinct reduced expressions for it. Using the remarks preceding the Lemma, we see that the $h$-element sets 
\begin{equation}\label{eq:setsofposrootsinrank2}
\{a_i,s_ia_j,s_is_ja_i,\dots  \} \text{ and } \{a_j,s_ja_i,s_js_ia_j,\dots  \}
\end{equation}
are identical. The set $\{\beta_{t}, \beta_{t+1},\ldots ,\beta_{t+m}\}$ before and after performing the operation (A) is the image of this $h$-element by $s_{k_1}s_{k_2}\cdots s_{k_{t-1}}.$ Furthermore, the elements $\beta_1,\ldots \beta_{t-1}$ and $\beta_{t+m},\ldots ,\beta_r$ are clearly unchanged by operation (A). 

To see the effect of operation (B), assume that $k_p=k_{p+1}=i.$ First the set $\{\beta_1, \ldots ,\beta_{p-1},\beta_{p+2},\ldots ,\beta_r\}$ is unchanged if $s_{k_p}s_{k_{p+1}}=s_is_i=1$ is omitted from the product \eqref{whatform}. Next note that  $\beta_h$ and $\beta_{h+1}$ are a root and its negative, since 
\be{}\beta_p=s_{k_1}\cdots s_{k_{p-1}}a_{k_p}=s_{k_1}\cdots s_{k_{p-1}}a_i \text{ and } \beta_{p+1}=s_{k_1}\cdots s_{k_{p-1}}s_{k_p}a_{k_p+1}=s_{k_1}\cdots s_{k_{p-1}}(-a_i).\ee Thus the proof is completed. 

\end{proof}

\tpoint{Prenilpotent pairs} \label{prenilp} In the course of defining his group functor, Tits uses the following notion. A set $\Psi \subset R_{re}$ of roots is said to be \emph{pre-nilpotent} if there exists $w, w' \in W$ such that $w \Psi \subset R_+$ and $w' \Psi \subset R_-.$ If such a set $\Psi$ is also closed, i.e. $a, b \in \Psi, a+b \in R$ implies $a + b \in \Psi,$ we shall say that $\Psi$ is \emph{ nilpotent }.  Given any prenilpotent pair of roots $\{ a, b \}$ we then define the sets \be{prenilp:ab} \begin{array}{lcr} [a, b] = (\mathbb{N}a + \mathbb{N}b) \cap R & \text{ and } & ]a, b[ \, = \, [a, b] \setminus \{ a, b \}. \end{array}. \ee These are finite, and if $a, b > 0$ (or $a, b < 0$) form a prenilpotent pair, then we can find $w \in W$ such that $[a,b] \subset R(w)$.

\subsection{Quadratic Forms and Metaplectic Structures} \label{sec-q-st}

\spoint \label{s:qform} Let $(\Is, \cdot)$ be a Cartan datum with gcm $\As$ and $\mf{D}= (Y, \{ y_i \}, X, \{ x_i \})$ be a root datum. Let $\Qs: Y \rr \zee$ be a $W:=W(\As)$-invariant quadratic form, i.e. a quadratic form such that $\Qs(y) = \Qs(w y)$ for $w \in W$ where the $W$-action on $Y$ is as in \S \ref{s:rd}. The associated bilinear form to $\Qs$ is denoted $\Bs: Y \times Y \rr \zee$ is \be{B:Q} \Bs(y_1, y_2) := \Qs(y_1+y_2) - \Qs(y_1) - \Qs(y_2) \text{ for } y_1, y_2 \in Y. \ee Note that if $\Qs$ is $W$-invariant, then so is $\Bs$, i.e. $\Bs(wy_1, wy_2)=\Bs(y_1, y_2)$ for $w \in W$ and $y_1, y_2 \in Y.$ If $B$ is $W$-invariant,  then (cf. \cite[Lemma 4.5]{del:bry} or \cite[Lemma 1.2]{weis:sp}), \be{bq:db} \B(y, y_i) = \la y, x_i \ra \Qs(y_i) \text { for } y \in Y. \ee 

\begin{nlem} Let $i, j \in \Is$ be such that $a_{ij} = \la y_i, x_j \ra$ is non-zero. If $\Qs(y_j) \neq 0,$ then $\Qs(y_i) \neq 0$ and  \be{Q:car} \frac{\Qs(y_i)}{\Qs(y_j)} =  \frac{a_{ij}}{a_{ji}}. \ee \end{nlem}
\begin{proof} If $\Qs(y_j) \neq 0$ and $\la y_i, x_j \ra \neq 0$ then also $\Bs(y_i, y_j) \neq 0$ by (\ref{bq:db}). Using (\ref{bq:db}) again, we write \be{Q:rat} \frac{\Qs(y_i)}{\Qs(y_j)} = \frac{\Bs(y_j , y_i) }{\la y_j, x_i \ra}   \frac{\la y_i, x_j \ra}{\Bs(y_i , y_j) } = \frac{\la y_i, x_j \ra}{ \la y_j, x_i \ra} = \frac{a_{ij}}{a_{ji}}, \ee where in the first equality we have also used the fact that $\la y_j, x_i \ra \neq 0,$ which follows from the assumption that $\la y_i, x_j \ra \neq 0.$  The assertion about $\Qs(y_i)$ follows immediately. \end{proof}

\begin{nrem} If the Dynkin diagram of $\As$ is connected and if $\Qs(y_j) =0$ for some $j$, then $\Qs=0$ on $\Qv_{\D}.$ \end{nrem}

\spoint Let $\As_o$ be a finite-type gcm with associated root system $R_o$, coroot lattice $\Qv_o$ and Weyl group $W_o$. Then one knows (cf. \cite[Proposition 3.10]{weis:pho}) that there exists a unique $W_o$-invariant, quadratic form $\Qs$ on $\Qv_o$ which takes the value $1$ on all short coroots. Moreover, every $\zee$-valued $W$-invariant form on $\Qv_o$ is an integer multiple of $\Qs.$ Generalizing this, we have 

\renewcommand{\qq}{\mathbb{Q}}

\begin{nprop} \label{Q-fill}Let $(\Is, \cdot, \D)$ with $\mf{D}= (Y, \{ y_i \}, X, \{ x_i \})$ be a root datum with symmetrizable, indecomposable gcm $\As$   
\begin{enumerate}[(1)]
\item There exists a $W$-invariant, $\zee$-valued quadratic form on $Y$

\item Every $W$-invariant, $\mathbb{Q}$-valued quadratic form on $Q_{\D}^{\vee} \otimes_{\zee} \mathbb{Q}$ is determined uniquely by the single value $\Q(\av_j),$ where $j \in \Is$ can be chosen arbitrarily.

\end{enumerate}

 \end{nprop}
\begin{proof} We begin with $(1)$: if $\As$ is symmetrizable, then in \eqref{sym:frm} we have constructed a $\mathbb{Q}$-valued invariant, bilinear form on $\mf{h}_{\D}$ and hence a $\mathbb{Q}$-valued invariant, quadratic form on $Y= \Qv \oplus Y_0$ (see \ref{Y:sc}). Some multiple of it will then be $\zee$-valued, so the existence of an integral, $W$-invariant form has been proven. 

As for $(2)$, suppose $\Qs$ is any rational quadratic form on $Q_{\D}^{\vee} \otimes_{\zee} \mathbb{Q}$. Then by the previous Lemma \ref{Q-fill}, if $j \in \Is$ is some node in the Dynkin diagram such that $\Qs(\av_j) \neq 0$, the value of $\Qs(\av_i)$ is determined (and non-zero) for all nodes attached to $j.$  If $\As$ is indecomposable, the Dynkin diagram is connected so the desired claim follows since fixing the values of $\Qs(\av_k)$ for all $k \in \Is$ fixes $\Qs$ (one uses (\ref{B:Q}) and (\ref{bq:db}) to verify this).  

\end{proof}

\newcommand{\rat}{\mathbb{Q}}

\tpoint{Metaplectic structures and root datum}\label{s:met-rts} Following \cite{weis:sp}, we define a \emph{metaplectic structure} on $(\Is, \cdot, \mf{D})$ to be a pair $(\Qs, n)$ where $\Qs$ is a $W:=W(\Is, \cdot)$-invariant quadratic form on $Y$ and $n$ is a positive integer.

\begin{nlem} \label{new-mt-cd} Let $(\Qs, n)$ be a metaplectic structure on $(\Is, \cdot, \mf{D}).$ Define \be{new-mt-cd:1} i \; \wc \;  j := \frac{n^2}{n(y_i) \, n(y_j) } i \cdot j, \ee where the $n(y_i) \, (i \in \Is)$ is the smallest positive integers satisfying \be{ni:def} n(y_i) \, \Qs(y_i) \equiv 0 \mod n. \ee  Then $(\Is, \wc)$ is again a Cartan datum of the same type (where type is defined in Proposition \ref{indec-class}) as $(I, \cdot),$ with associated gcm $(\Is, \wc)$ is \be{met-car} \wt{\As} = \left(\frac{ n(y_i) }{n(y_j)} a_{ij} \right)_{i,j\in \Is}.\ee

 \end{nlem}
\begin{proof} One may verify as in \cite[p.95]{weis:sp} that $(I, \wc)$ is again a Cartan datum (note that the proof there assumed that $(I, \circ)$ was of finite type, but the same argument works in general). The matrix of the form $\wc$ is obtained from that of $\cdot$ from a change of basis of by a diagonal matrix with positive, rational entries. Hence $(I, \cdot)$ is of the same type as $(I, \wc)$ according to the trichotomy of Proposition \ref{indec-class}. The last claim is clear.  \end{proof}

\renewcommand{\t}[1]{\widetilde{#1}}

For fixed $(\Is, \cdot, \mf{D})$ with a metaplectic structure $(\Qs, n) $ construct $(\Is, \wc)$ as in the Proposition. Following \cite[Construction 1.3]{weis:sp} we now set \begin{itemize} 
\item $\t{Y}:= \{ y \in Y \mid B(y, y') \in n \zee \text{ for all } y' \in Y \}$
\item $\t{y}_i:= n(y_i) y_i $ for $i \in I$
\item $\t{X}:= \{ x \in X \otimes \mathbb{Q} \mid \la y, x \ra \in \zee \text{ for all } y \in \t{Y} \}$
\item $\t{x}_i:= n(y_i)^{-1} x_i$ for $i \in I$
\end{itemize}

Then one can verify as in \cite[Construction 1.3]{weis:sp} that $\widetilde{\mf{D}}= (\t{Y}, \{\t{y}^{\vee}_i\}_{i\in \Is},\t{X}, \{\t{x}^{\vee}_i\}_{i\in \Is}) $ is a root datum for $(\Is, \wc)$ and we let $\mf{g}_{\wt{\D}}$ the corresponding Lie algebra.  The roots (coroots) for this Lie algebra will be denoted by $\tR$ (resp. $\tR^{\vee}$)-- its simple roots and coroots are identified with $\tx_i$ and $\ty_i,$ etc.

\subsection{Affine Root Systems} \label{sec:affine}

\spoint \label{aff:DB} Let $\As$ be an affine Cartan matrix of rank $\ell$. Then $\As$ is positive semi-definite and the null space of $\As$ is one dimensional. Let $\delta = (d_1, \ldots, d_{\ell+1})$ be the unique vector in $\zee_{>0}$ with relatively prime entries (cf. \cite[Theorem 4.8 (b)]{kac}) which spans this space. Similarly, the transpose ${}^t\As$ is again an affine Cartan matrix, and we define the analogous vector $\delta^{\vee}=(\dv_1, \ldots, \dv_{\ell+1})$ in its null space. From the classification of affine Cartan matrices, we have that $\dv_{\ell+1}=1$ for all affine types (cf. Table \ref{table:affinerootsystem}).  Defining $\epsilon_{i}:= d_i \, (\dv_i)^{-1} \text{ for } i \in \Is$  we have (cf. \cite[Remark 6.1]{kac}) that $\As= D \cdot B$ where $B$ is a symmetric Cartan matrix and where $D= \diag(\epsilon_1, \ldots, \epsilon_{\ell+1})$ is a symmetrization.

\begin{table}[h!]
$$\begin{array}{l|l|l|l}
\text{Type of $\As$} & \text{Type of ${}^t\As$} & \delta = (d_1,\ldots ,d_{\ell +1}) & \delta ^{\vee }=(\dv_1,\ldots ,\dv_{\ell +1}) \\ \hline \hline 
A_1 ^{(1)} & A_1 ^{(1)} & (1,1) & (1,1)\\ \hline 
A_{\ell} ^{(1)} \ (\ell \geq 2) & A_{\ell} ^{(1)} \ (\ell \geq 2)  & (1,\ldots ,1) & (1,\ldots ,1)\\ \hline 
B_{\ell} ^{(1)} \ (\ell \geq 3) & A_{2\ell-1}^{(2)} \ (\ell \geq 3) & (1,2,\ldots ,2,1) & (1,2,\ldots ,2,1,1)\\ \hline 
C_{\ell} ^{(1)} \ (\ell \geq 2) & D_{\ell+1}^{(2)} \ (\ell \geq 2) & (2,\ldots ,2,1,1) & (1,\ldots ,1)\\ \hline 
D_{\ell}^{(1)} \ (\ell \geq 4) & D_{\ell}^{(1)} \ (\ell \geq 4) & (1, 2, \ldots ,2, 1, 1, 1) & (1, 2, \ldots ,2, 1, 1, 1)\\ \hline 
E_6^{(1)} & E_6^{(1)} & (1,2,3,2,1,2,1) & (1,2,3,2,1,2,1)\\ \hline 
E_7^{(1)} & E_7^{(1)} & (2,3,4,3,2,1,2,1) & (2,3,4,3,2,1,2,1)\\ \hline 
E_8^{(1)} & E_8^{(1)} & (2,3,4,5,6,4,2,3,1) & (2,3,4,5,6,4,2,3,1) \\ \hline 
F_4^{(1)} & E_6^{(2)} & (2,3,4,2,1) & (2,3,2,1,1)\\ \hline 
G_2^{(1)} & D_4^{(3)} & (2,3,1) & (2,1,1)\\ \hline \hline
A_{2}^{(2)} & A_{2}^{(2)} & (1,2) & (2,1) \\ \hline 
A_{2\ell}^{(2)} \ (\ell \geq 2) & A_{2\ell}^{(2)} \ (\ell \geq 2) & (2,\ldots, 2, 1, 2) & (2,\ldots ,2,1)\\ \hline 
A_{2\ell-1}^{(2)} \ (\ell \geq 3) & B_{\ell} ^{(1)} \ (\ell \geq 3) & (1,2,\ldots ,2,1,1) & (1,2,\ldots ,2,1) \\ \hline 
D_{\ell+1} ^{(2)} \ (\ell \geq 2) & C_{\ell}^{(1)} \ (\ell \geq 2) & (1,\ldots ,1) & (2,\ldots ,2,1,1) \\ \hline 
E_6^{(2)} & F_4^{(1)} & (2,3,2,1,1) & (2,3,4,2,1)\\ \hline 
D_4^{(3)} & G_2^{(1)} & (2,1,1) & (2,3,1)
\end{array}$$
\caption{The affine types with dual pairs.}
\label{table:affinerootsystem}
\end{table}

The matrix obtained from $\As$ by deleting the $\ell+1$st row and column will be denoted by $\As_o.$ It is a Cartan matrix of finite type, as is ${}^t\As_o.$ Note that if we fix a symmetrization of $\As$, then this induces one on $\As_o$ as well, which we denote by $\As_o = D_o B_o,$ where both $D_o,\ B_o$ are obtained from $D,\ B$ by deleting the $\ell+1$st row and column.  

Let $\mf{g}:=\mf{g}(\As)$ be the Kac-Moody algebra associated to $\As$ where one chooses $\mf{d}$ to be the span of $\dd:= \dd_{\ell+1}$ (cf. \S \ref{s:la}, (1) for notation). Also, set $\mf{g}_o:= \mf{g}(\As_o)$ be the Kac-Moody algebras to $\As_o$ respectively (where the corresponding $\mf{d}=0$).  We denote with a subscript ``o''  objects corresponding to  $\As_o$: e.g. $R_o$ denotes the set of roots, $W_o:= W(\As_o)$ the Weyl group, $\Pv_o \subset \Pv$ the simple coroots, $\rho_o$ is as in (\ref{rho:def}), etc. and by $R, W, \rho$, etc. similar object for $\As$.  

Below we follow the conventions of \cite[\S 4.7]{kac} in associating a Dynkin diagram to $\As$: the vertices correspond to$i\in \Is$, and the (labelled) edges are constructed as follows (recall for affine types, we always have $a_{ij}a_{ji}\leq 4$): the vertices $i$ and $j$ are connected by $|a_{ij}|$ lines if $|a_{ij}| \geq |a_{ji}|,$ and these lines are equipped with an arrow pointing in the direction of $i$ if $|a_{ij}|>1;$ 


\begin{figure}[h!]
\begin{tabular}{b{2cm} m{5cm} b{1cm} m{4cm}}
$A_1^{(1)}$ & \begin{tikzpicture}
 \dynkdot{0}{0}
 \dynkdot{1}{0}
 \dynkdoublelinetofr{0}{0}{1}{0}
 \draw (0*\dynkstep,0*\dynkstep) node [anchor = north] {$1$};
 \draw (1*\dynkstep,0*\dynkstep) node [anchor = north] {$2$};
 \end{tikzpicture} & $E_{6}^{(1)}$ & \begin{tikzpicture}
 \dynkdot{0}{0}
 \dynkdot{1}{0}
 \dynkdot{2}{0}
 \dynkdot{3}{0}
 \dynkdot{4}{0}
 \dynkdot{2}{1}
 \dynkdot{2}{2}
 \dynkline{0}{0}{1}{0}
 \dynkline{1}{0}{2}{0}
 \dynkline{2}{0}{3}{0}
 \dynkline{3}{0}{4}{0}
 \dynkline{2}{0}{2}{1}
 \dynkline{2}{1}{2}{2}
 \draw (0*\dynkstep,0*\dynkstep) node [anchor = north] {$1$};
 \draw (1*\dynkstep,0*\dynkstep) node [anchor = north] {$2$};
 \draw (2*\dynkstep,0*\dynkstep) node [anchor = north] {$3$};
 \draw (3*\dynkstep,0*\dynkstep) node [anchor = north] {$4$};
 \draw (4*\dynkstep,0*\dynkstep) node [anchor = north] {$5$};
 \draw (2*\dynkstep,1*\dynkstep) node [anchor = west] {$6$};
 \draw (2*\dynkstep,2*\dynkstep) node [anchor = west] {$7$};
 \end{tikzpicture}  \\
$A_{\ell}^{(1)}\ (\ell\geq 2)$ & \begin{tikzpicture}
\dynkdot{0}{0}
\dynkdot{1}{0}
\dynkdot{3}{0}
\dynkdot{4}{0}
\dynkdot{2}{1}
\dynkline{0}{0}{1}{0}
\dynkline{3}{0}{4}{0}
\dynkline{0}{0}{2}{1}
\dynkline{4}{0}{2}{1}
\dynkcdots{1}{0}{3}{0}
\draw (0*\dynkstep,0*\dynkstep) node [anchor = north] {$1$};
\draw (1*\dynkstep,0*\dynkstep) node [anchor = north] {$2$};
\draw (3*\dynkstep,0*\dynkstep) node [anchor = north] {$\ell-1$};
\draw (4*\dynkstep,0*\dynkstep) node [anchor = north] {$\ell$};
\draw (2*\dynkstep,1*\dynkstep) node [anchor = south] {$\ell +1$};
\end{tikzpicture} & $E_{7}^{(1)}$ & \begin{tikzpicture}
\dynkdot{0}{0}
\dynkdot{1}{0}
\dynkdot{2}{0}
\dynkdot{3}{0}
\dynkdot{4}{0}
\dynkdot{5}{0}
\dynkdot{6}{0}
\dynkdot{3}{1}
\dynkline{0}{0}{1}{0}
\dynkline{1}{0}{2}{0}
\dynkline{2}{0}{3}{0}
\dynkline{3}{0}{4}{0}
\dynkline{5}{0}{4}{0}
\dynkline{5}{0}{6}{0}
\dynkline{3}{0}{3}{1}
\draw (0*\dynkstep,0*\dynkstep) node [anchor = north] {$8$};
\draw (1*\dynkstep,0*\dynkstep) node [anchor = north] {$1$};
\draw (2*\dynkstep,0*\dynkstep) node [anchor = north] {$2$};
\draw (3*\dynkstep,0*\dynkstep) node [anchor = north] {$3$};
\draw (4*\dynkstep,0*\dynkstep) node [anchor = north] {$4$};
\draw (5*\dynkstep,0*\dynkstep) node [anchor = north] {$5$};
\draw (6*\dynkstep,0*\dynkstep) node [anchor = north] {$6$};
\draw (3*\dynkstep,1*\dynkstep) node [anchor = west] {$7$};
\end{tikzpicture}  \\
$B_{\ell}^{(1)}\ (\ell \geq 3)$ & \begin{tikzpicture}
\dynkdot{0}{0}
\dynkdot{1}{0}
\dynkdot{2}{0}
\dynkdot{4}{0}
\dynkdot{5}{0}
\dynkdot{1}{1}
\dynkline{0}{0}{1}{0}
\dynkline{1}{0}{1}{1}
\dynkline{1}{0}{2}{0}
\dynkdoublelinefrom{4}{0}{5}{0}
\dynkcdots{2}{0}{4}{0}
\draw (0*\dynkstep,0*\dynkstep) node [anchor = north] {$1$};
\draw (1*\dynkstep,0*\dynkstep) node [anchor = north] {$2$};
\draw (2*\dynkstep,0*\dynkstep) node [anchor = north] {$3$};
\draw (4*\dynkstep,0*\dynkstep) node [anchor = north] {$\ell -1$};
\draw (5*\dynkstep,0*\dynkstep) node [anchor = north] {$\ell$};
\draw (1*\dynkstep,1*\dynkstep) node [anchor = south] {$\ell +1$};
\end{tikzpicture} & $E_{8}^{(1)}$ & \begin{tikzpicture}
\dynkdot{0}{0}
\dynkdot{1}{0}
\dynkdot{2}{0}
\dynkdot{3}{0}
\dynkdot{4}{0}
\dynkdot{5}{0}
\dynkdot{6}{0}
\dynkdot{7}{0}
\dynkdot{5}{1}
\dynkline{0}{0}{1}{0}
\dynkline{1}{0}{2}{0}
\dynkline{2}{0}{3}{0}
\dynkline{3}{0}{4}{0}
\dynkline{5}{0}{4}{0}
\dynkline{5}{0}{6}{0}
\dynkline{7}{0}{6}{0}
\dynkline{5}{0}{5}{1}
\draw (0*\dynkstep,0*\dynkstep) node [anchor = north] {$9$};
\draw (1*\dynkstep,0*\dynkstep) node [anchor = north] {$1$};
\draw (2*\dynkstep,0*\dynkstep) node [anchor = north] {$2$};
\draw (3*\dynkstep,0*\dynkstep) node [anchor = north] {$3$};
\draw (4*\dynkstep,0*\dynkstep) node [anchor = north] {$4$};
\draw (5*\dynkstep,0*\dynkstep) node [anchor = north] {$5$};
\draw (6*\dynkstep,0*\dynkstep) node [anchor = north] {$6$};
\draw (7*\dynkstep,0*\dynkstep) node [anchor = north] {$7$};
\draw (5*\dynkstep,1*\dynkstep) node [anchor = west] {$8$};
\end{tikzpicture}  \\
& & \\
$C_{\ell}^{(1)}\ (\ell \geq 2)$ & \begin{tikzpicture}
\dynkdot{0}{0}
\dynkdot{1}{0}
\dynkdot{3}{0}
\dynkdot{4}{0}
\dynkdoublelinefrom{0}{0}{1}{0}
\dynkdoublelineto{3}{0}{4}{0}
\dynkcdots{1}{0}{3}{0}
\draw (0*\dynkstep,0*\dynkstep) node [anchor = north] {$\ell +1$};
\draw (1*\dynkstep,0*\dynkstep) node [anchor = north] {$1$};
\draw (3*\dynkstep,0*\dynkstep) node [anchor = north] {$\ell -1$};
\draw (4*\dynkstep,0*\dynkstep) node [anchor = north] {$\ell$};
\end{tikzpicture} & $F_{4}^{(1)}$ & \begin{tikzpicture}
\dynkdot{0}{0}
\dynkdot{1}{0}
\dynkdot{2}{0}
\dynkdot{3}{0}
\dynkdot{4}{0}
\dynkline{0}{0}{1}{0}
\dynkline{1}{0}{2}{0}
\dynkdoublelinefrom{2}{0}{3}{0}
\dynkline{3}{0}{4}{0}
\draw (0*\dynkstep,0*\dynkstep) node [anchor = north] {$5$};
\draw (1*\dynkstep,0*\dynkstep) node [anchor = north] {$1$};
\draw (2*\dynkstep,0*\dynkstep) node [anchor = north] {$2$};
\draw (3*\dynkstep,0*\dynkstep) node [anchor = north] {$3$};
\draw (4*\dynkstep,0*\dynkstep) node [anchor = north] {$4$};
\end{tikzpicture}  \\
$D_{\ell}^{(1)}\ (\ell \geq 4)$ & \begin{tikzpicture}
\dynkdot{0}{0}
\dynkdot{1}{0}
\dynkdot{2}{0}
\dynkdot{4}{0}
\dynkdot{5}{0}
\dynkdot{1}{1}
\dynkdot{4}{1}
\dynkline{0}{0}{1}{0}
\dynkline{1}{0}{1}{1}
\dynkline{1}{0}{2}{0}
\dynkline{4}{0}{5}{0}
\dynkline{4}{0}{4}{1}
\dynkcdots{2}{0}{4}{0}
\draw (0*\dynkstep,0*\dynkstep) node [anchor = north] {$1$};
\draw (1*\dynkstep,0*\dynkstep) node [anchor = north] {$2$};
\draw (2*\dynkstep,0*\dynkstep) node [anchor = north] {$3$};
\draw (4*\dynkstep,0*\dynkstep) node [anchor = north] {$\ell-2$};
\draw (4*\dynkstep,1*\dynkstep) node [anchor = south] {$\ell$};
\draw (5*\dynkstep,0*\dynkstep) node [anchor = west] {$\ell-1$};
\draw (1*\dynkstep,1*\dynkstep) node [anchor = south] {$\ell +1$};
\end{tikzpicture} & $G_{2}^{(1)}$ & \begin{tikzpicture}
\dynkdot{0}{0}
\dynkdot{1}{0}
\dynkdot{2}{0}
\dynkline{0}{0}{1}{0}
\dynktriplelinefrom{1}{0}{2}{0}
\draw (0*\dynkstep,0*\dynkstep) node [anchor = north] {$3$};
\draw (1*\dynkstep,0*\dynkstep) node [anchor = north] {$1$};
\draw (2*\dynkstep,0*\dynkstep) node [anchor = north] {$2$};
\end{tikzpicture} 
\end{tabular}
\caption{The untwisted affine Dynkin diagrams.}
\label{figure:untwistedaffineDynkins}
\end{figure}
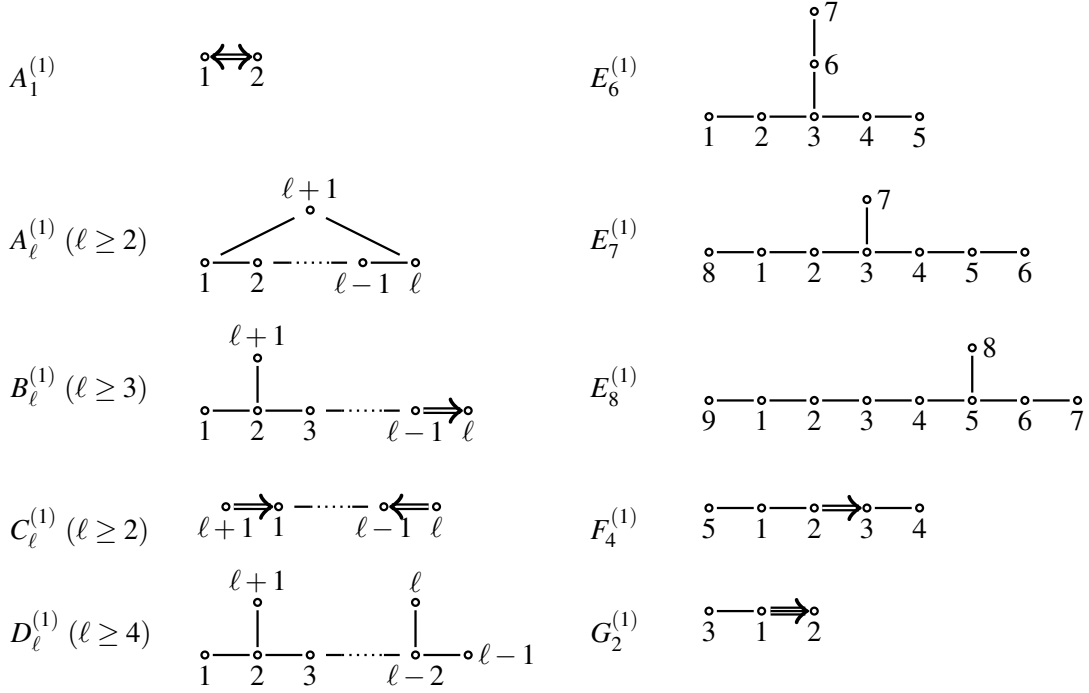

\begin{figure}[h!]
\begin{tabular}{b{2.5cm} m{4.5cm} b{2cm} m{4cm}}
$A_{2}^{(2)}$ & \begin{tikzpicture} 
\dynkdot{0}{0}
\dynkdot{1}{0}
\dynkquadruplelineto{0}{0}{1}{0}
\draw (0*\dynkstep,0*\dynkstep) node [anchor = north] {$2$};
\draw (1*\dynkstep,0*\dynkstep) node [anchor = north] {$1$};
\end{tikzpicture} & $D_{\ell+1}^{(2)}\ (\ell \geq 2)$ & \begin{tikzpicture}
\dynkdot{0}{0}
\dynkdot{1}{0}
\dynkdot{3}{0}
\dynkdot{4}{0}
\dynkdoublelineto{0}{0}{1}{0}
\dynkdoublelinefrom{3}{0}{4}{0}
\dynkcdots{1}{0}{3}{0}
\draw (0*\dynkstep,0*\dynkstep) node [anchor = north] {$\ell +1$};
\draw (1*\dynkstep,0*\dynkstep) node [anchor = north] {$1$};
\draw (3*\dynkstep,0*\dynkstep) node [anchor = north] {$\ell -1$};
\draw (4*\dynkstep,0*\dynkstep) node [anchor = north] {$\ell$};
\end{tikzpicture}\\
$A_{2\ell}^{(2)}\ (\ell \geq 2)$ & \begin{tikzpicture} 
\dynkdot{0}{0}
\dynkdot{1}{0}
\dynkdot{3}{0}
\dynkdot{4}{0}
\dynkdoublelineto{0}{0}{1}{0}
\dynkdoublelinefrom{3}{0}{4}{0}
\dynkcdots{1}{0}{3}{0}
\draw (0*\dynkstep,0*\dynkstep) node [anchor = north] {$\ell +1$};
\draw (1*\dynkstep,0*\dynkstep) node [anchor = north] {$1$};
\draw (3*\dynkstep,0*\dynkstep) node [anchor = north] {$\ell -1$};
\draw (4*\dynkstep,0*\dynkstep) node [anchor = north] {$\ell$};
\end{tikzpicture} & $E_{6}^{(2)}$ & \begin{tikzpicture}
\dynkdot{0}{0}
\dynkdot{1}{0}
\dynkdot{2}{0}
\dynkdot{3}{0}
\dynkdot{4}{0}
\dynkline{0}{0}{1}{0}
\dynkline{1}{0}{2}{0}
\dynkdoublelineto{2}{0}{3}{0}
\dynkline{3}{0}{4}{0}
\draw (0*\dynkstep,0*\dynkstep) node [anchor = north] {$5$};
\draw (1*\dynkstep,0*\dynkstep) node [anchor = north] {$1$};
\draw (2*\dynkstep,0*\dynkstep) node [anchor = north] {$2$};
\draw (3*\dynkstep,0*\dynkstep) node [anchor = north] {$3$};
\draw (4*\dynkstep,0*\dynkstep) node [anchor = north] {$4$};
\end{tikzpicture} \\
$A_{2\ell-1}^{(2)}\ (\ell \geq 3)$ & \begin{tikzpicture}
\dynkdot{0}{0}
\dynkdot{1}{0}
\dynkdot{2}{0}
\dynkdot{4}{0}
\dynkdot{5}{0}
\dynkdot{1}{1}
\dynkline{0}{0}{1}{0}
\dynkline{1}{0}{1}{1}
\dynkline{1}{0}{2}{0}
\dynkdoublelineto{4}{0}{5}{0}
\dynkcdots{2}{0}{4}{0}
\draw (0*\dynkstep,0*\dynkstep) node [anchor = north] {$1$};
\draw (1*\dynkstep,0*\dynkstep) node [anchor = north] {$2$};
\draw (2*\dynkstep,0*\dynkstep) node [anchor = north] {$3$};
\draw (4*\dynkstep,0*\dynkstep) node [anchor = north] {$\ell -1$};
\draw (5*\dynkstep,0*\dynkstep) node [anchor = north] {$\ell$};
\draw (1*\dynkstep,1*\dynkstep) node [anchor = south] {$\ell +1$};
\end{tikzpicture} & 
$D_{4}^{(3)}$ & \begin{tikzpicture}
\dynkdot{0}{0}
\dynkdot{1}{0}
\dynkdot{2}{0}
\dynkline{0}{0}{1}{0}
\dynktriplelineto{1}{0}{2}{0}
\draw (0*\dynkstep,0*\dynkstep) node [anchor = north] {$3$};
\draw (1*\dynkstep,0*\dynkstep) node [anchor = north] {$1$};
\draw (2*\dynkstep,0*\dynkstep) node [anchor = north] {$2$};
\end{tikzpicture}
\end{tabular}
\caption{The twisted affine Dynkin diagrams.}
\label{figure:twistedaffineDynkins}
\end{figure}

\spoint We would like to describe in more detail the relation between $\mf{h}_o$ and $\mf{h}.$ To do so, first set \be{cc} \cc := \dv_1 \av_1 + \cdots \dv_{\ell+1} \av_{\ell+1} = \dv_1 \av_1 + \cdots +\dv_{\ell} \av_{\ell} + \av_{\ell+1} \ee and define $ \theta^{\vee}:= \cc - \av_{\ell+1},$ and observe that $\la \cc , a_i \ra =0$ for $i \in \Is$ and that $\cc$ is the minimal, positive imaginary coroot of $\mf{g}(\As)$ i.e. every other positive imaginary coroot of $\mf{g}(\As)$ is an integral multiple of $\cc.$   We have chosen $\dd \in \mf{h}$ which satisfies the condition \be{d:def} \la \dd, a_i \ra =0 \text{ for } i \in \Is_o, \la \dd , \, a_{\ell+1} \ra =1. \ee With these choices, we find that $\{ \av_1, \ldots, \av_{\ell+1}, \dd \}$ form a basis of $\mf{h}.$ On the other hand, $\{ \av_1, \ldots, \av_{\ell} \}$ form a basis for $\mf{h}_o$ and we have decompositions \be{H:dec} \mf{h} = \C \cc \oplus \mf{h}_o \oplus \C \dd, \ee with respect to which we write elements of $\mf{h}$ as $y = m \cc + y_o + n \dd$ where $m, n \in \C$ and $y_o \in \mf{h}_o.$ The non-degenerate bilinear form $(\cdot, \cdot)$ introduced in \S \ref{s:la} (with respect to the standard symmetrization of $\As$) takes the following explicit form, \be{ext-form} (\av_i, \av_j) &=& d_i (\dv_j)^{-1} a_{ij}  \text{ for } i, j \in \Is\\ 
(\av_i, \dd) &=& 0 \, (i \in \Is_o); \ \ 
(\av_{\ell+1}, \dd) = d_{\ell+1}; \ \   
(\dd, \dd)=0. \ee We also record here the additional formulas, \be{pair:h} \begin{array}{lcr} (\cc, \av_i) =0  \, \, (i \in \Is); & (\cc, \cc)=0 ; & (\cc, \dd)=1 .\end{array}. \ee

\newcommand{\Qve}{Q_e^{\vee}}

\tpoint{Simply connected root datum}  Let $\mf{D}= (Y, \{ y_i \} , X , \{ x_i \})$ be a simply-connected root datum and $\mf{g}_{\D}$ the corresponding Lie algebra Then by \eqref{Y:sc} and \S \ref{s:gD}, we can make the following identifications:  $y_i$ with $\av_i,$ $x_i$ with $a_i$, $\mf{h}:= Y \otimes_{\zee} \C,$ $Y_0 = \zee \dd$ and \be{XY;sc} \begin{array}{lcr} Y=\Qv \oplus \zee \dd& \text{ and } & X= \{ \lambda \in \mf{h}^* \mid \la \av_i, \lam \ra \in \zee \text{ and } \la \dd, \lam \ra \in \zee \}. \end{array} \ee

\begin{nlem}  \label{Q-fill:aff} Let $(I, \cdot)$ be a Cartan datum of affine type, and $\mf{D}=(Y, \{ y_i \}, X, \{ x_i \})$ be a simply connected root datum with identifications made as above. 
\begin{enumerate} 
\item Fix an integral, $W_o$-invariant form $\Qs_o.$ Then there exists a unique $W$-invariant, $\zee$-valued form on $Y$ extending $\Qs_o$ and satisfying $\Qs(\dd)=0.$ 
\item There exists a $W$-invariant, integral quadratic form $\Qs$ on $Y$ such that $\Qs(\dd)=0$ and such that $\min_{i \in \Is} \Qs(\av_i) = 1$. Every $W$-invariant, integral quadratic form on $Y$ which takes value 0 on $\dd$ is an integer multiple of it. \end{enumerate}

\end{nlem}
\begin{proof} The first part follows from Proposition \ref{Q-fill} (2). As for the first, note from (\ref{Q:rat}) that $\Qs(\av_i)/ \Qs(\av_j) = \varepsilon_i/ \varepsilon_j$, i.e. $(\Qs(\av_1),\ldots ,\Qs(\av_{\ell +1}))$ is the multiple of $(\varepsilon_1,\ldots ,\varepsilon_{\ell+1}).$ However, an inspection of Table \ref{table:affinerootsystem} (and noting that $\varepsilon_i = d_i(\dv_i)^{-1}$) shows that there is always $i \in \Is$ such that $\varepsilon_i=1$. Thus $\Qs$ exists, and the rest of (2) follows again from Proposition \ref{Q-fill}.\end{proof}

\begin{nrem} \label{rmk:Qsets} In Table \ref{table:metaplecticrootsystem}, we list the values of $\Qs(\av_i)$ for $i \in \Is.$ Observe that $\{\Qs(\av_i)\mid i\in \Is \}$ is $\{1\},$ $\{1,2\},$ or $\{1,3\}$ for every affine type except $A_{2\ell}^{(2)}$ ($\ell\geq 1$).  \end{nrem}

\tpoint{Metaplectic Cartan matrices} Fix $(I, \cdot, \D)$ a simply-connected root datum and let $\Qs$ be the $W$-invariant, integral quadratic form on $Y$ constructed in Lemma \ref{Q-fill:aff}. For each integer $n \geq 1$, we thus have a metaplectic structure $(\Qs, n)$ on $\mf{D}$ to which we can apply the constructions of \S \ref{s:met-rts}. In particular we have defined the metaplectic Cartan matrices $\wt{\As}$ in \eqref{met-car}, and we wish to tabulate the possibilities here. 

Write $\wt{\As} = \left(\wt{a_{ij}}\right)_{i,j\in \Is}.$ Since $n(\av_i) \, (i \in \Is)$ is the smallest positive integers satisfying $n(\av_i) \, \Qs(\av_i) \equiv 0 \mod n,$ we have $n(\av_i) = n/\gcd(n, \Qs(\av_i)).$ Hence we may write
\be{eq:met-car-Q} \wt{a_{ij}} = \frac{ \gcd(n, \Qs(\av_j)) }{\gcd(n, \Qs(\av_i))} a_{ij} .\ee

%

\begin{nlem}\label{lem:sameordual}
Let $\wt{\As}$ be the metaplectic gcm corresponding to the simply connected root datum $(\Is, \cdot, \mf{D})$ and metaplectic structure $(\Qs, n)$ where the root datum is of an affine type {\em{other than}} $A_{2\ell}^{(2)}$ ($\ell\geq 1$). Then the type of $\wt{\As}$ is equal to $\As$ or to $\As^t.$
\end{nlem}
\begin{proof}
If $\Qs(\av_i)=1$ for every $i\in \Is$ then \eqref{eq:met-car-Q} implies that $\wt{\As}=\As.$ If this is not the case, then Remark \ref{rmk:Qsets} above implies that $\{\Qs(\av_i)\mid i\in \Is \}=\{1,p\}$ where $p=2$ or $p=3.$ If $p\nmid n$ then $\gcd(n,\Qs(\av_i))=1$ for any $i\in \Is$ and hence \eqref{eq:met-car-Q} implies that $\wt{\As}=\As .$ We show that if $p|n$ then $\wt{\As}=\As ^t.$ This amounts to proving that $\wt{a_{ij}}=a_{ji}$ for any $i,j\in \Is.$ 

First note that since $\Qs(\av_i)$ is equal to $1$ or $p$ ($i\in \Is$) and $p|n,$ we have 
\begin{equation}\label{eq:sameasgcd}
\Qs(\av_i) = \gcd(n,\Qs(\av_i)) \ (i\in \Is).
\end{equation}
Let us fix $i,j\in \Is.$ If $a_{ij}=a_{ji}=0$ then $\wt{a_{ij}} = \wt{a_{ji}}=0$ as well, in particular $\wt{a_{ij}}=a_{ji}=0$. If $a_{ij}a_{ji}\neq 0,$ then \eqref{eq:met-car-Q}, \eqref{Q:car} and \eqref{eq:sameasgcd} implies
$$\wt{a_{ij}} = \frac{ \gcd(n, \Qs(\av_j)) }{\gcd(n, \Qs(\av_i))} a_{ij} = \frac{ \Qs(\av_j) }{\Qs(\av_i)} a_{ij} =\frac{a_{ji}}{a_{ij}}\cdot a_{ij}=a_{ji}.$$
This completes the proof.
\end{proof}

Table \ref{table:metaplecticrootsystem} shows each affine type with the type of $\wt{A}$ for each positive integer $n.$ Lemma \ref{lem:sameordual} above implies that the last column of the table is correct in every row other than the rows corresponding to $A_{2\ell}^{(2)}$ ($\ell\geq 1$). It remains to check the case of $A_{2\ell}^{(2)}$ ($\ell\geq 1$).

\begin{table}[h!]
$$\begin{array}{l|l|l}
\text{Type of $\As$} & (\Qs(\av_1),\ldots ,\Qs(\av_{\ell +1}))& \text{Type of $\wt{A}$} \\ \hline \hline 
A_1 ^{(1)} & (1,1)  & A_1 ^{(1)} \text{ for any $n$} \\ \hline 
A_{\ell} ^{(1)} \ (\ell \geq 2)  & (1,\ldots ,1) & A_{\ell} ^{(1)} \text{ for any $n$} \\ \hline 
B_{\ell} ^{(1)} \ (\ell \geq 3)  & (1,\ldots ,1, 2,1) & B_{\ell} ^{(1)} \text{ if $n$ is odd} \\ 
 & &  A_{2\ell-1}^{(2)}  \text{ if $n$ is even}\\ \hline 
C_{\ell} ^{(1)} \ (\ell \geq 2) & (2,\ldots ,2,1,1) & C_{\ell} ^{(1)} \text{ if $n$ is odd} \\ 
 & &  D_{\ell+1}^{(2)}  \text{if $n$ is even}\\ \hline 
D_{\ell}^{(1)} \ (\ell \geq 4) &  (1, \ldots , 1) & D_{\ell}^{(1)} \text{ for any $n$} \\ \hline 
E_6^{(1)} & (1,\ldots ,1) & E_6^{(1)} \text{ for any $n$} \\ \hline 
E_7^{(1)} & (1,\ldots ,1) & E_7^{(1)} \text{ for any $n$} \\ \hline 
E_8^{(1)} & (1,\ldots ,1) & E_8^{(1)} \text{ for any $n$} \\ \hline 
F_4^{(1)} & (1,1,2,2,1) & F_4^{(1)} \text{ if $n$ is odd}\\ 
 & &  E_6^{(2)} \text{ if $n$ is even}\\ \hline 
G_2^{(1)} &  (1,3,1) & G_2^{(1)} \text{ if $3\nmid n$} \\
 & & D_4^{(3)} \text{ if $3\mid n$} \\ \hline \hline
A_{2}^{(2)} &  (1,4) & A_{2}^{(2)} \text{ if $n$ is odd}\\
 & & A_{1}^{(1)} \text{ if $n$ is even, $4\nmid n$}\\ 
 & & (A_{2}^{(2)}) ^{\vee}\equiv A_{2}^{(2)}  \text{ if $4|n$}\\ \hline 
A_{2\ell}^{(2)} \ (\ell \geq 2) &  (2,\ldots ,2, 1, 4) & A_{2\ell}^{(2)} \text{ if $n$ is odd}\\
 & & D_{\ell+1}^{(2)} \text{ if $n$ is even, $4\nmid n$}\\ 
 & & (A_{2\ell}^{(2)})^{\vee}\equiv A_{2\ell}^{(2)}  \text{ if $4|n$}\\ \hline 
A_{2\ell-1}^{(2)} \ (\ell \geq 3) &  (2, \ldots, 2, 1, 2) & A_{2\ell-1}^{(2)} \text{ if $n$ is odd} \\ 
 & & B_{\ell} ^{(1)} \text{ if $n$ is even}\\ \hline 
D_{\ell+1} ^{(2)} \ (\ell \geq 2) & (1,\ldots ,1,2,2) & D_{\ell+1}^{(2)}  \text{ if $n$ is odd} \\ 
 & & C_{\ell} ^{(1)}  \text{ if $n$ is even}\\ \hline 
E_6^{(2)} &  (2,2,1,1,2) & E_6^{(2)} \text{ if $n$ is odd}\\ 
 & & F_4^{(1)}  \text{ if $n$ is even}\\ \hline 
D_4^{(3)} &  (3,1,3) & D_4^{(3)} \text{ if $3\nmid n$} \\
 & & G_2^{(1)} \text{ if $3\mid n$} 
\end{array}$$
\caption{Metaplectic systems corresponding to each affine type}
\label{table:metaplecticrootsystem}
\end{table}

To check the last column of the table for $A_{2\ell}^{(2)},$ note that if $n$ is odd, $\gcd(n,\Qs(\av_i))=1$ and hence $\wt{a_{ij}}=a_{ij}$ is immediate. Furthermore if $4|n,$ then $\gcd(n,\Qs(\av_i))=\Qs(\av_i)$ ($i\in \Is$), so by the proof of Lemma \ref{lem:sameordual} $\wt{\As}=\As$ follows. If $4\nmid n$ but $n$ is even, then the statements can be proved by explicitly computing the Cartan matrix. 

The gcm of type $A_{2}^{(2)}$ can be read off from the Dynkin diagram in Figure \ref{figure:twistedaffineDynkins}, it is $\As= \begin{pmatrix} 2 & -1 \\ -4 & 2 \end{pmatrix}$
%
Let $n$ be even, but not divisible by $4,$ then $\gcd(n,\Qs(\av_1))=1$ and $\gcd(n,\Qs(\av_1))=2,$ hence 
\begin{equation}\label{eq:tildeCartanm:A22}
\wt{\As} = \left(\begin{array}{cc}
2 & -2\\
-2 & 2
\end{array}\right),
\end{equation}
a gcm of type $A_2^{(1)}.$

The gcm of type $A_{2}^{(2 \ell)}$ can again be read off from \ref{figure:twistedaffineDynkins}, it is 
\begin{equation}\label{eq:Cartanm:A22l}
\As = \left(\begin{array}{ccccccc}
2 & -1 & 0 & \cdots & 0 & 0 & \boxed{-1} \\
-1 & 2 & -1 & & & & \\
0 & -1 & 2 & & & & \vdots \\
\vdots & & &  & -1 & & \\
0 & & & -1 & 2 & \boxed{-2} & 0\\
0 & & & & \boxed{-1} & 2 & 0\\
\boxed{-2} & 0 & 0 & \cdots & 0 & 0  & 2\\
\end{array}\right)
\end{equation}
The boxed entries $a_{ij}$ are the only nonzero ones with $\Qs(\av _i)\neq \Qs(\av_j).$ When $4\nmid n$ but $n$ is even, then $\gcd(n,\Qs(\av_i))=2$ unless $i=\ell,$ and $\gcd(n,\Qs(\av_{\ell}))=1,$ hence $\wt{\As}$ is as follows:
\begin{equation}\label{eq:Cartanm:A22lt}
\wt{\As} = \left(\begin{array}{ccccccc}
2 & -1 & 0 & \cdots & 0 & 0 & \boxed{-1} \\
-1 & 2 & -1 & & & & \\
0 & -1 & 2 & & & & \vdots \\
\vdots & & &  & -1 & & \\
0 & & & -1 & 2 & \boxed{-1} & 0\\
0 & & & & \boxed{-2} & 2 & 0\\
\boxed{-2} & 0 & 0 & \cdots & 0 & 0  & 2\\
\end{array}\right).
\end{equation}
This gcm corresponds to the Dynkin diagram 
\begin{equation}\label{proofeq:Dynk:A22l}
\begin{tikzpicture} 
\dynkdot{0}{0}
\dynkdot{1}{0}
\dynkdot{3}{0}
\dynkdot{4}{0}
\dynkdoublelineto{0}{0}{1}{0}
\dynkdoublelinefrom{3}{0}{4}{0}
\dynkcdots{1}{0}{3}{0}
\draw (0*\dynkstep,0*\dynkstep) node [anchor = north] {$\ell +1$};
\draw (1*\dynkstep,0*\dynkstep) node [anchor = north] {$1$};
\draw (3*\dynkstep,0*\dynkstep) node [anchor = north] {$\ell -1$};
\draw (4*\dynkstep,0*\dynkstep) node [anchor = north] {$\ell$};
\end{tikzpicture}
\end{equation} 
and hence it is of type $D_{\ell+1}^{(2)}.$

\subsection{Rank 2 Considerations} \label{rank2}

\newcommand{\bavpx}{{m}}
\newcommand{\abvpy}{{n}}
\newcommand{\prpairz}{{z}}

In this section we shall present some elementary results on rank two root systems which will be used in the proofs of Lemma \ref{w:h} and in \S \ref{lp:case2} and \S \ref{lp:case3} below.

%
%

\spoint We begin with a purely combinatorial result,

\begin{nlem} \label{gf:fla} There exist polynomials $f_k(X), g_k(X) \in \zee[X]$ ($k\geq 0$) such that 
\begin{eqnarray}
g_{k+1}(X)= & f_k(X)-g_k(X); \label{eq:g_rec}\\ 
f_{k+1}(X)= & X \cdot g_{k+1}(X)-f_k(X) \label{eq:f_rec}
\end{eqnarray} and $f_0(X)=1,$ $g_0(X)=0.$
\end{nlem}
The above uniquely determine $f_k(X), g_k(X),$ and a straightforward induction shows that in fact \be{g:f-1}
g_k(X) &= & \sum_{i=0}^{k-1} (-1)^i\cdot \binom{2k-1-i}{i}\cdot X^{k-1-i} \\ \label{g:f-2}
f_k(X) & = & \sum_{j=0}^{k} (-1)^j\cdot \binom{2k-j}{j}\cdot X^{k-j}. \ee

\spoint Let $\Is = \{ 1, 2 \}$ and let  \be{As:2} \As := \begin{pmatrix} 2 & \bavpx \\ \abvpy & 2 \end{pmatrix}. \ee Denote the simple roots by $\Pi = \{ a, b \},$ and let $W$ be the Weyl group, every element of which is of the form 
\begin{equation}\label{eq:w_form}
(s_as_b)^k,\ (s_b s_a)^k,\ s_b(s_as_b)^k\text{ or }s_a(s_bs_a)^k \ (k\geq 0).
\end{equation} Using the polynomials constructed in the previous paragraph, we can now state the following result.

\begin{nlem} \label{lem:rk2wgprops} The elements of $W$ act on the simple roots $a$ and $b$ as follows: 
\begin{eqnarray}
(s_as_b)^k(a)= & f_k(\prpairz)\cdot a+(-\abvpy)\cdot g_k(\prpairz)\cdot b \label{eq:even_st_a_on_a}\\
s_b(s_as_b)^k(a)= & f_k(\prpairz)\cdot a+(-\abvpy)\cdot g_{k+1}(\prpairz)\cdot b \label{eq:odd_st_b_on_a}\\
(s_bs_a)^k(a)= & -f_{k-1}(\prpairz)\cdot a+(-\abvpy)\cdot (-g_{k}(\prpairz))\cdot b \label{eq:even_st_b_on_a}\\
s_a(s_bs_a)^k(a)= & (-f_k(\prpairz))\cdot a+(-\abvpy)\cdot (-g_k(\prpairz))\cdot b \label{eq:odd_st_a_on_a}\\
(s_bs_a)^k(b)=  & f_k(\prpairz)\cdot b+(-\bavpx)\cdot g_k(\prpairz)\cdot a\label{eq:even_st_b_on_b}\\
s_a(s_bs_a)^k(b)=  & f_k(\prpairz)\cdot b+(-\bavpx)\cdot g_{k+1}(\prpairz)\cdot a \label{eq:odd_st_a_on_b}\\
(s_as_b)^k(b)=  & -f_{k-1}(\prpairz)\cdot b+(-\bavpx)\cdot (-g_{k}(\prpairz))\cdot a \label{eq:even_st_a_on_b}\\
s_b(s_as_b)^k(b)=  & (-f_k(\prpairz))\cdot b+(-\bavpx)\cdot (-g_k(\prpairz))\cdot a \label{eq:odd_st_b_on_b}
\end{eqnarray}
where $\bavpx=\la b,\av \ra ,$ $\abvpy=\la a,\bv \ra $ and $z=\bavpx \, \abvpy.$ \end{nlem} 
\begin{proof} We have by definition that 
\begin{equation}\label{eq:simpleactions}
s_a(a)=-a,\ s_a(b)=b-\bavpx a,\ s_b(a)=a-\abvpy b,\ s_b(b)=-b.
\end{equation} 
The identities \eqref{eq:odd_st_b_on_a}, \eqref{eq:even_st_b_on_a} and \eqref{eq:odd_st_a_on_a} all follow from \eqref{eq:even_st_a_on_a} by \eqref{eq:simpleactions}. Note also that \eqref{eq:even_st_b_on_b} follows from \eqref{eq:even_st_a_on_a} by symmetry (exchanging $a$ and $b$ exchanges $x$ and $y$). Then again \eqref{eq:odd_st_a_on_b} \eqref{eq:even_st_a_on_b} and \eqref{eq:odd_st_b_on_b} follow from \eqref{eq:even_st_b_on_b} by \eqref{eq:simpleactions}. Thus it suffices to show \eqref{eq:even_st_a_on_a} for every $k.$ One may prove \eqref{eq:even_st_a_on_a} and the recursions \eqref{eq:g_rec} and \eqref{eq:f_rec} at the same time by a straightforward induction on $k$ using \eqref{eq:simpleactions} and the linearity of the Weyl group action. (For instance for $k=0,$ we have $f_k(z)=1$ and $g_k(z)=0.$)
\end{proof}

\spoint \label{s:swap} Next, we analyze those root systems in which there are non-trivial elements in $W$ stabilizing a simple root. These are actually all finite dimensional.

\begin{nlem}\label{lem:eltsfixingroot} Suppose there exists $w \in W, w \neq 1$ so that $w(a)=a$. Then one of the following holds: 
\begin{enumerate}[(i)]
\item $A_1 \times A_1,$ \ \  $\la b,\av\ra =\la a,\bv\ra =0$ and $w=s_b$
\item $B_2,$ \ \ $\la b,\av\ra \cdot \la a,\bv\ra =2$ and $w=w^{-1}=s_bs_as_b$
\item $G_2,$ \ \ $\la b,\av\ra \cdot \la a,\bv\ra =3$ and $w=s_b(s_as_b)^2$ or its inverse $s_a(s_bs_a)^3.$
\end{enumerate}
\end{nlem}
\begin{proof}
We use the notation and results of Lemma \ref{lem:rk2wgprops} freely. First note that $m=0$ holds if and only if $n=0.$ In this case $s_a$ and $s_b$ commute, $W \, \cong (\ZZ/2\ZZ)^2$ and $w(a)=a$ for a nontrivial $w$ implies $w=s_b,$ as in (i). From now on we assume $xy\neq 0.$ Then it follows from the formulas in Lemma \ref{lem:rk2wgprops} that if $w$ is nontrivial and $w(a)=a$ then there is a $k>0$ such that $g_k(z)=0.$ Then for $v=(s_as_b)^k$ we have that by \eqref{eq:even_st_a_on_a} and \eqref{eq:even_st_a_on_b}
\begin{equation}\label{eq:v_on_aandb}
v(a)=f_k(z)\cdot a \text{ and } v(b)=-f_{k-1}(z)\cdot b
\end{equation}
and by \eqref{eq:f_rec} $f_k(z)=z\cdot g_k(z)-f_{k-1}(z)=-f_{k_1}(z).$
This implies that $v(a)=a$ and $v(b)=b$ or $v(a)=-a$ and $v(b)=-b$ \cite[Proposition 5.1 b)]{kac}. Thus either $v$ or $v^2$ is the identity in $W$ \cite[(3.12.1)]{kac}, and the order of $(s_as_b)$ is finite in $W$ (a divisor of $2k$). By \cite[Proposition 3.13]{kac}) (cf. \ref{h:tab}) this implies $z=xy\leq 3.$ Then $W=\la s_a,s_b\ra$ is finite, and $a,b$ are simple roots of a root system of type $A_2,$ $B_2=C_2$ or $G_2.$ In each of these cases we can easily compute $f_k(z)$ and $g_k(z)$ for the finitely many elements of $W,$ and check that (ii) or (iii) is satisfied.
\end{proof}

\spoint \label{s:fix} We shall also need to analyze those rank two root systems in which the two simple roots are in the same Weyl group orbit. Again, this only happens for finite-dimensional root systems.

\begin{nlem}\label{lem:weylatob} If there exists $w\in W$ so that $w(a)=b$ then $\la b,\av\ra=\la a,\bv\ra=-1$ (i.e. $a$ and $b$ are simple roots of a root system of Dynkin type $A_2$) and $w=s_as_b$ or its inverse $(s_bs_a)^2.$
\end{nlem}
\begin{proof}
It follows from the formulas of Lemma \ref{lem:rk2wgprops} that if $w(a)=b$ then $f_k(z)=0$ for some $k\geq 0,$ where $z=\la b,\av\ra\cdot \la a,\bv\ra$ 
It follows from Lemma \ref{gf:fla} that the leading coefficient of $f_k$ is $1$ and the constant term of $f_k$ is $\pm 1$ for every $k.$\footnote{This also follows more directly by induction from $f_0(z)=1,$ $g_z=0$ and the recursions \eqref{eq:g_rec} and \eqref{eq:f_rec}.} The integer $z$ is a root of $f_k$, thus $|z|=1.$ This implies that $z=1$ and $\la b,\av\ra=\la a,\bv\ra=-1,$ and the order of $(s_as_b)$ is $3.$ One may compute $f_k(z)$ and $g_k(z)$ for $k\leq 3$ from $g_0(z)=0$ and $f_0(z)=1$ and \eqref{eq:g_rec} and \eqref{eq:f_rec}. From these values and the formulas of Lemma \ref{lem:rk2wgprops} we can conclude that $w(a)=b$ implies $w=s_as_b$ or $w=(s_bs_a)^2.$ \end{proof}

\section{Combinatorics of Infinite Symmetrizers}

In this section we study two (infinite) sums over the Weyl group of a Kac-Moody root system.The first, termed a ``simple symmetrizer'' (cf. \S \ref{sym-sim}), corresponds roughly to the \emph{answer} provided by  Casselman and Shalika \cite{ca:sh} for the unramified Whittaker function. It is also, with the exception of the correction factor $\mf{m}$ described below, the same symmetrizer analyzed by Lee and Zhang \cite{lee:zh}. The second, termed a ``Hecke symmetrizer'' (cf. \S \ref{hecke-sym}), is built out of certain Demazure-Lusztig type operators and corresponds to a group-theoretic decomposition of the unramified Whittaker function (cf. \eqref{Iwa-Whit}).  The ``non-metaplectic'' analogue of the results presented here have been studied elsewhere \cite{cher:ma}, \cite{bkp} for (untwisted) affine root systems. We revisit the argument in \cite{cher:ma} and show  in \S \ref{sec:pf-prop-non-met} how it can can be extended, using crucially a result of Viswanath \cite{vis}, to a more general Kac-Moody setting. The metaplectic generalization then follows easily in \S \ref{sec-met-sym} using standard properties of the Chinta-Gunnells action and the formalism of metaplectic root data.

\subsection{Non-metaplectic Symmetrizers } \label{sec-non-met-sym}

\newcommand{\Cfv}{\C_v^{\fin}}
\newcommand{\Cv}{\C_v}

\tpoint{Notation}  Throughout this section,  $(\Is, \cdot)$ will be a Cartan datum with associated gcm $\As$ and $\mf{g}(\As)$ the corresponding Lie algebra. We keep the notations of \S \ref{s:la} and additionally write $r:= | \Is |,$ so the sets $\Pi$ and $\Pi^{\vee}$ will be enumerated as $\{a_1, \ldots, a_r \}$ and $\{ \av_1, \ldots, \av_r \}.$ Also introduce a formal parameter $v$ and let \be{Cfin}\begin{array}{lcr} \Cfv:=\C[v] & \text{ and } & \Cv:= \C[[v]] \end{array} \ee denote the ring of polynomials and power series respectively in this parameter.

\spoint \label{expansions} Let  $\mc{Q}:=  \C[[e^{-\av_1}, \ldots, e^{-\av_r}]]$ be the ring of power series in the indeterminates $e^{-\av_i} (i \in \Is)$ subject to the usual group algebra relations $e^{-\av} e^{-\bv} = e^{- \av - \bv}$. We then set \be{q:qfin} \begin{array}{lcr}
\qnf = \Cfv \otimes_{\C} \mc{Q} & \text{ and } & \hqn = \C[[v]] \otimes_{\C} \mc{Q} \end{array} \ee Sometimes we refer to the former ring as the set of \emph{$v$-finite} elements in the latter. The action of $W$ on $\Qv$ does \emph{not} induce an action on the completion $\mc{Q}$ (or $\qnf$ or $\hqn$): the multiplication \be{mult:try} a_{\sigma} [ \sigma] b_{\tau} [\tau] = a_{\sigma} b_{\tau}^{\sigma} [\sigma \tau] \text{ with } \sigma, \tau \in W \text{  and } a_{\sigma}, b_{\tau} \in \hqn \ee may not be a well-defined in $\hqn$ as $b_{\tau}^{\sigma}$, the application of $\sigma$ to $b_{\tau},$ may not be contained in $\hqn$. So we just view $\qn$ and $\hqn$ as $\C$-vector spaces in the sequel, but \eqref{mult:try} is an important heuristic rule to bear in mind (and indeed sometimes it does produce meaningful answers). We also consider the vector space  \be{qn:vs} \hqn[W]^{\vee}:= \left\lbrace \sum_{\sigma \in W} a_{\sigma} [\sigma], \,  \text{ where } a_\sigma \in \hqn \right\rbrace. \ee  We may repeat the above construction replacing $\hqn$ with $\qn.$ The space so obtained will be denoted $\qn[W]^{\vee}$, and called the set of $v$-finite elements in $\hqn[W]^{\vee}.$

\tpoint{Some rational functions and their expansions} \label{rats} Next, we introduce the rational functions of a formal variable $X$ which are to  play a central role in this part\footnote{In the sequel, expressions with a ${}^{\flat}$ will signify ``Whittaker'' variants of their unadorned ``spherical'' counterparts}, \be{b:c} \begin{array}{lcr} \b(X) = \frac{v -1}{1- X}, &  \cw(X) = \frac{ 1 - v X^{-1}}{1 - X}, &  \c(X) = \frac{ 1 - v X}{1 - X} \end{array}. \ee 
We record here the expansions of these rational functions in powers of $X$ and $X^{-1}$ respectively, 
\be{b:exp} 
\b(X) & = & (v- 1) ( 1+ X + X^2 + \cdots) = (1-v) ( X^{-1} + X^{-2} + \cdots ) \\ 
\c(X) &=& 1 + (1-v) X + (1-v) X^2 + \cdots = v + (v - 1) X^{-1} + (v - 1) X^{-2} + \cdots \\ 
\cw(X) &=& -v X^{-1} + (1-v) + (1-v)X + \cdots = - X^{-1} + (v-1) X^{-2} + (v-1) X^{-3}  + \cdots \ee 
\comment{Changed last line from 
\be{} \cw(X) &=& -v X^{-1} + (1-v) + (1-v)X + \cdots = - X^{-1} + v X^{-2} + (v-1) X^{-3} + (v-1)X^{-4} + \cdots. \ee 
PA-09-06}
We shall also need the following simple observation, 
\be{c:aa} \c(X) \c(X^{-1})  = \cw(X) \cw(X^{-1}); \ee 
moreover, both expressions are equal to 
\be{c:aa-exp}  v + (-1+ 2 v - v^2) X + (-2 + 4 v - 2v^2)  X^2 + \cdots  \ee

\noindent For a root $a \in R$ we shall write \be{bc:2} \b(\av):= \b(e^{\av}), \, \c(\av):= \c(e^{\av}), \, \cw(\av):= \cw(e^{\av}). \ee  Using the expansions above, they can be seen to lie in $\qnf.$ If $a  \in R_{+}$ for example, one takes the above expansions in $X^{-1}$ and substitutes the value $X= e^{\av}.$ 

\tpoint{A result on Poincare polynomials} \label{sec:mac-vis} For each root $a \in R$, recall the notion of root multiplicity $\mul(\av)$ introduced in \S \ref{s:la} (4) . Consider the following infinite product  \be{Del} \Delta = \prod_{ a \in R_+} \left( \frac{1 - v e^{ -\av}}{1 - e^{  -\av}} \right)^{\mul(\av)} . \ee  
It is easy to see that the expansion of $\Delta$ in negative powers of the coroots lies in $\qn$ and in fact it is a unit in this ring.  For each $w \in W$ we may also consider the following element whose expansion again lies in $\qn$, \be{Del:w} \Delta^w =\prod_{ a \in R_+} \left( \frac{1 - v e^{ - w \av}}{1 - e^{  - w \av}} \right)^{\mul(\av)}. \ee Generalizing a result of Macdonald \cite[(3.8)]{mac:aff} for affine root systems, Viswanath has shown the following,

\begin{nprop}\label{vis} \cite[\S 7.1]{vis}\footnote{The element $\mf{m}$ is $P_0(t)^{-1}$ in  notation of \cite{vis}.} \label{visw} There exists an element $\mf{m} \in \qn$ such that \be{vis:id} \mf{m} \, \sum_{w \in W} \Delta^w =  \sum_{w \in W} v^{\ell(w)}. \ee  \end{nprop} 

\newcommand{\hI}{\widehat{\mc{I}}}
\newcommand{\hIw}{\widehat{\mc{I}}^{\flat}}

\tpoint{Simple Symmetrizers} \label{sym-sim} For each $w \in W,$ the expression $\mf{m} \, \Delta^w$ has an expansion in $\qn$ as both $\mf{m}$ and $\Delta^w$ do.  Hence, we may consider the following expressions in $\qnd,$ 
\be{I}  \I &:=&  \, \sum_{ w \in W}  \Delta^w [w], \hskip .5cm  \hI :=\mf{m}\I \\
\label{I:w} \Iw &:=& \  \, \Delta \, \sum_{w \in W}  (-1)^{\ell(w)}  \left( \prod_{a \in R(w) } e^{ - \av} \right)  [w], \hskip .5cm  \hIw :=\mf{m}\Iw.    \ee 
Note that we may also write $\Iw=  \, \Delta \, \sum_{w \in W}  (-1)^{\ell(w)} e^{ w \rho^{\vee} - \rho^{\vee}} [w],$ which draws out the connection to the Weyl character formula more explicitly. In other words, if we ``apply'' $\Iw$ to the element $e^{\lv}$ with $\lv \in \Lv_+$ we obtain the expression (which lives in the space defined in \ref{loo-space} below) \be{app-Iw} \Iw(e^{\lv}) = \prod_{a \in R_+} (1 - v e^{-\av})^{\mul(\av)} x_{\lv} \ee where $x_{\lv}$ is the Weyl-Kac character of the irreducible representation of $\mf{g}$ with highest weight $\lv.$

\tpoint{Demazure-Lusztig operators} \label{hecke-sym} 
For $i\in I$, we use the rational functions introduced in \S \ref{rats}, to define the following elements which have expansions that lie in $\qnf[W]^{\vee}$ 
\be{Ts} \T_{w_i} &:=& \c(a_i) [w_i] + \b(a_i) [1], \\ \Tw_{w_i} &:=& \cw(a_i) [w_i] + \b(a_i) [1]. \ee 
One verifies that the braid relations holds for these elements (see \cite[Proposition 6.1]{pat:whit} and \cite{lus-K}), and hence we can form $\Tw_w$ and $\T_w$ in the usual way:  if $w=w_{i_1}\cdots w_{i_k}$ is a reduced decomposition with each $w_{i_j}$ a simple reflection in $W$, we have a well-defined expression (independent of the reduced decomposition chosen)
\be{T:exp-1}
\T_w &= &\T_{w_{i_1}} \, \cdots \, \T_{w_{i_k}}\\
&= &(\c(a_{i_1})[w_{i_1}]+\b(a_{i_1})[1])\cdots (\c(a_{i_r})[w_{i_k}]+\b(a_{i_k})[1]) \label{T:exp-2}.\ee 

\noindent Expanding and moving all the expression involving $\c$ and $\b$ to the left, we obtain formally 
\be{T:exp-3} \T_w = \sum_{\sigma \in W} A_{\sigma}(w) [\sigma], \ee 
where $A_{\sigma}(w)$ are some rational functions in $\b, \c$. Similarly, we may write \be{T:exp-3:w} \Tw_w = \sum_{\sigma \in W} A^{\flat}_{\sigma}(w) [\sigma], \ee 
where $A_{\sigma}(w)$ are some rational functions in $\b, \cw$.

\newcommand{\Cw}{C^{\flat}}

\tpoint{Hecke Symmetrizers} \label{heck-sym} Consider now the following \emph{Hecke symmetrizers} \be{Ts} \begin{array}{lcr} \mc{P}:= \sum_{ w \in W} \T_w & \text{ and } & \mc{P}^{\flat} := \sum_{ w \in W} \Tw_w \end{array}. \ee  We need to give a precise meaning to these expressions. To formulate this problem more explicitly, we first use \eqref{T:exp-3} and \eqref{T:exp-3:w} to write  
\be{C:sig} \begin{array}{lcr} C_{\sigma} = \sum_{w \in W} A_{\sigma}(w) & \text{ and } & \Cw_{\sigma} = \sum_{w \in W} \Aw_{\sigma}(w)  \end{array} \ee 
so that formally
\be{P:Csig} \begin{array}{lcr} \P = \sum_{ w \in W} \T_w = \sum_{\sigma \in W} C_{\sigma} [\sigma] & \text{ and } & \Pw = \sum_{ w \in W} \Tw_w = \sum_{\sigma \in W} \Cw_{\sigma} [\sigma]. \end{array}  \ee 
It remains to see that the expansion of $C_{\sigma}, \Cw_{\sigma}$ is well-defined for each $\sigma \in W.$ To be more precise, we introduce the following terminology: for a formal sum $f = \sum_{\lv \in \Lv} c_{\lv} e^{\lv}$ with $c_{\lv}$ coefficients (in some space), we write \be{coeff:term} [e^{\mv}] \, f = c_{\mv} \text{  if } f = \sum_{\lv \in \Lv} c_{\lv} e^{\lv}. \ee 

\begin{de} We say that the element $\mc{P}$ (resp. $\Pw$) has a well-defined expansion in $\hqn[W]^{\vee}$ if the following is satisfied: \begin{enumerate}[(a)] \item For each $\mv \in \Lv$ we have $[e^{\mv}] C_{\sigma} \in \hqn$ \item If $\mv \notin \Lv_-$ then $[e^{\mv}] C_{\sigma}=0$ \end{enumerate} Similarly we can define when $\mc{P}$ (resp. $\Pw$) has a well-defined expansion in $\qn[W]^{\vee}.$ \end{de}

\tpoint{Statement of Main proportionality result}  We now have all of the terminology to state our main result on the two symmetrizers introduced above.  The following result is due to Cherednik  \cite[Lemma 2.19]{cher:ma} in the affine context.

\begin{nthm} \label{form:sym} With the notion of expansions defined above, the elements $\P$ and $\Pw$ above are well-defined. More precisely, we have that \begin{enumerate} \item (Weak Cherednik Lemma) The elements $\P$ and $\Pw$ have well-defined expansions in $\hqn[W]^{\vee}.$ 

\item There exist $W$-invariant elements $\mf{c}$ and $\cprow$ in $\hqn$ such that as elements in $\hqn[W]^{\vee}$  \be{eq:sph} \begin{array}{lcr} \mc{P} =  \mf{c} \, \I & \text{ and } &  \Pw = \cprow \,  \Iw \end{array}, \ee where $\I, \Iw$ were the ``simple symmetrizers'' defined in (\ref{I}, \ref{I:w}). Moreover, we have $\mf{c} = \cprow = \mf{m} ,$ where $\mf{m}$ is as in (\ref{vis:id}). 

 \item (Strong Cherednik Lemma) The elements $\P$ and $\Pw$ have well-defined expansions in $\qn[W]^{\vee}$ and both equalities in (\ref{eq:sph}) can be viewed in $\qn[W]^{\vee}.$

\end{enumerate}
\end{nthm}

We explain the proof of this result in \S \ref{sec:pf-prop-non-met}, but in the remainder of this section we explain some more explicit formulas for the factors $\mf{c}$ and $\mf{c}^{\flat}$ in the affine case.

\tpoint{Formulas for $\mf{c}$ and $\cprow$ in the affine case} \label{sec:ctfla} For affine root systems, the constant term conjecture of Macdonald and its resolution by Cherednik provides an explicit formula for $\mf{c}$, and hence also one for $\cprow$ by Theorem \ref{form:sym}(2).  We follow here the exposition in \cite{mac:for}, but note that our  expression $\Delta$ differs from the one in \emph{op. cit} by the presence of the imaginary roots: to make this more explicit, write \be{Del'} \begin{array}{lcr} \Delta' = \prod_{ a\in R_{+, re}} \frac{1 - v e^{- \av}}{1- e^{-\av}} & \text{ and } & \mf{i} = \prod_{i=1}^{\ell} \prod_{n=1}^{\infty} \frac{ 1- v e^{  -n \cc}}{1 - e^{  -n \cc}}  \end{array}  \ee where $\cc$ is the minimal imaginary coroot introduced in (\ref{cc}). Then $\Delta = \Delta' \, \mf{i}.$ Changing every root to its negative in what Macdonald refers to as $\Delta$ gives $1/\Delta'$ in our terminology. Since the root system is isomorphic to its negative, identities in \emph{op. cit} remain true if we replace Macdonald's $\Delta$ by our $1/\Delta'.$


If $f \in \hqn,$ which we may write as a (possibly infinite) sum  \be{f-ct} f = \sum_{\mv \in Q_-^{\vee} } c_{\mv} \, e^{\mv}, \ee its \emph{constant term} is defined as \be{ct} \ct(f):= \sum_{n \in \zee_{\leq 0} } c_{n \cc} \, e^{ n \cc}. \ee   
\begin{nprop}\cite[(3.8)]{mac:for} \label{mac-ct}  For affine root systems, we have $\mf{m} = {\ct(\Delta ^{-1})}.$ \end{nprop}

To go further, we need to introduce some more notation. For $R$ our affine root system, let $R_o$ denote the underlying finite dimensional one as specified in \S \ref{aff:DB}. Consider the Poincare series of $W_o,$ \be{poin-fin} W_o(v) = \sum_{w \in W_o} v^{\ell(w)} .\ee Then, it is known \cite[(2.6)]{mac:poin} that $W_o(v)$ has a product decomposition, \be{poin-fin:dec} W_o(v) = \prod_{i=1}^{\ell} \frac{1 - v^{m_i+1}}{1-v} \ee for (uniquely specified) integers $m_1, \ldots, m_{\ell}$ which are called the \emph{exponents} of $R_o.$ The following is a rephrasal by Macdonald \cite[(3.5)]{mac:poin} of a characterization of the exponents due to Kostant (cf. \cite{kostant}). Let $\psi: \zee \rr A$ be \emph{any} map from $\zee$ to a multiplicative Abelian group $A$, and for each $\alpha^{\vee} \in R_o^{\vee},$ we set $| \alpha^{\vee} | := \la \alpha^{\vee}, \rho_o \ra.$  Then 
\begin{eqnarray} \label{mac} \prod_{\alpha \in R_{o, +}^{\vee}}  \frac{\psi (1 + |\alpha^{\vee}| )}{\psi (| \alpha^{\vee} |)} = \prod_{j=1}^{\ell} \frac{\psi (m_j+1 )}{\psi (1)}. \end{eqnarray}  Now from the proof of the constant term conjecture by Cherednik (cf. \cite[p.201, (3.12)]{mac:for}) one finds that for affine, untwisted ADE types,  
\begin{eqnarray} \label{che} \ct(\Delta^{-1})  = \mf{i}^{-1} \ct(\Delta'^{-1}) = \mf{i}^{-1} \prod_{\alpha^{\vee} \in R^{\vee}_{o, +}} \prod_{i=1}^{\infty} \frac{ (1 - v^{|\alpha^{\vee}|}e^{-i \cc})  (1 - v^{|\alpha^{\vee}|}e^{-i \cc}) }{ (1 - v^{|\alpha^{\vee}|-1}e^{-i \cc}) (1 - v^{|\alpha^{\vee}|+1}e^{-i \cc}) }. \end{eqnarray}  Applying \eqref{mac} we obtain that 
\begin{eqnarray} \ct(\Delta'^{-1}) = \left( \prod_{i=1}^{\infty}  \frac{ 1 - ve^{-i \cc}}{1- e^{-i \cc}} \right) ^{\ell} \prod_{j=1}^{\ell} \prod_{i=1}^{\infty} \frac{1 - v^{m_j}e^{-i \cc}}{1 - v^{m_j+1} e^{-i \cc} } . \end{eqnarray} 


\begin{nprop} \label{mcformula} For an untwisted affine root system $R$ of type ADE with underlying finite root system $R_o$ , 
\begin{eqnarray} \label{corr} \mf{m} =  \ct(\Delta^{-1}) =  \prod_{j=1}^{\ell} \prod_{i=1}^{\infty} \frac{1 - v^{m_j}e^{-i \cc } }{1 - v^{m_j+1} e^{-i \cc }}, 
\end{eqnarray} where the $m_j$ are the exponents of $R_o$. 
\end{nprop} 

For \emph{every} affine root system, there is a similar product formula. As the notation is a bit more involved, we simply refer to \cite[5.8.20]{mac:aff}.

\subsection{Proofs of Proportionality (non-metaplectic case)} \label{sec:pf-prop-non-met}

Our goal in this section is to prove Theorem \ref{form:sym}. We keep the notation from the previous section here.

\spoint \label{pf-out} Let us first outline the argument. To do so we need to introduce some further notation. Let $w \in W$ have reduced decomposition $w= w_{i_1} \cdots w_{i_k}$ and recall that the expansions (\ref{T:exp-3}) and (\ref{T:exp-3:w}). We need a more precise version of these expansions which we write as  %
%
%
%
%
%
\be{Tw:p2} \T_w = \sum_{\vec{p}} A_{\vec{p}}(w) [\vec{p}], \ee and which we now proceed to explain.  The sum (\ref{Tw:p2}) is indexed over increasing chains $\p =\{p<p'<p''<\cdots \}$ where the terms in this chain are integers between $1$ and $k,$ representing the places where terms with $[w_{i_j}]$ are chosen in the product \eqref{T:exp-2}. For each such $\p$ we also define an element in $W$ as 
\be{vec:p} [\vec{p}] := w_{i_{p}}w_{i_{p'}}w_{i_{p''}}\cdots \text{ if } \p =\{p<p'<p''<\cdots \} .\ee 
To obtain the expression for $A_{\p}(w)$ we begin with 
\be{Ap:w} \b(a_{i_1})[1]\cdots \b(a_{i_{p-1}})[1]\cdot  \c(a_{i_{p}})[w_{i_p}]\cdot  \b(a_{i_{p+1}})[1]\cdots \b(a_{i_{p'-1}})[1]\cdot  \c(a_{i_{p'}})[w_{i_{p'}}]\cdot  \b(a_{i_{p'+1}})[1] \cdots, \ee 
and pass the rational functions $\b$ and $\c$ ``to the left'' applying the rule (\ref{mult:try}) to pass rational functions past Weyl group elements. To describe the expression we obtain by this process, we first introduce the roots $\beta_i \in R_{re}$ as follows,

\begin{equation}\label{eq:betanotation}
\begin{split}
&\beta_1:=a_{i_1} \, , \ldots ,  \, \beta_p:=a_{i_{p}}, \\
&\beta_{p+1}:=w_{i_{p}}(a_{i_{p+1}}),\ldots ,\beta_{p'}:=w_{i_{p}}(a_{i_{p'}}),\\
&\beta_{p'+1}:=w_{i_{p}}w_{i_{p'}}(a_{i_{p'+1}}),\ldots , \beta_{p''}:=w_{i_{p}}w_{i_{p'}}(a_{i_{p''}}),\text{ etc.}
\end{split}
\end{equation}

\noindent Then one may verify that  \begin{equation}\label{eq:prodform}
A_{\p}(w) = \prod_{j\notin \p}\b(\beta_j)\cdot \prod _{j\in \p}\c (\beta _j),
\end{equation} where we write $j \in \p$ to mean that $j$ is one of the indices $p, p', \dots$ in the chain $\p,$  and we write $j \notin \p$ to mean those indices $j \in \{1, \ldots, k \} $ which are in the complement of the set $\{ j \mid j \in \p \}$.
 
In an entirely analogous manner, if we replace $\c$ with $\cw$ we may define the terms $\Aw_{\p}(w),$ and write an expression similar to \eqref{eq:prodform}. Note that the indexing sets of $\p$ in (\ref{Tw:p2}) and the $\beta_i$ defined in (\ref{eq:betanotation}) do not change from $A_{\p}(w)$ to $\Aw_{\p}(w).$ Expanding the rational functions $A_{\p}(w)$ and $\Aw_{\p}(w)$ in $\qn,$ we can view $\Tw_w, \T_w \in \qnd.$ Comparing with \eqref{Tw:p2} we find 
\be{comp-exp} A_{\sigma}(w) = \sum_{\substack{\p\\ [\p]=\sigma}} A_{\p}(w), \ee 
where of course the sum is over those $\p$ obtained from the chosen reduced decomposition of $w$ we started with. Similar results hold in the ${}^{\flat}$ case.  

For future reference we define two integers measuring the complexity of the $\p$ as above.
\begin{de} We define the the positive integers $\wid(\p)$ and $\len(\p)$ of $\p$ as follows,  \be{wid:len} \wid(\p) &=& \max\{ p, p'-p, p''-p', \ldots \}, \\ \len(\p) &=& \text{ number of terms in the chain }  \p. \ee  \end{de} 

Note that if $\ell(w)$ is large, then either $\wid(\p)$ or $\len(\p)$ must also be large.  The proof of Theorem \ref{form:sym} proceeds in several steps. Let $\sigma \in W.$ The argument below will pertain to establishing various properties of $C_{\sigma},$ the results for $\Cw_{\sigma}$ are analogous.

\begin{enumerate}[(1)]
\item First, we prove a result which implies that if $[\p]=\sigma$ and $\wid(\p)$ is large, then the contribution for any $w\in W$ from $A_{\p}(w)$ to $C_{\sigma}$ will arise with a factor of $e^{-\mv}$ with $\mv \in Q_+$ large (i.e. $\la \rho, \mv \ra$ is large). 

\item Next, we prove a result which implies that if $[\p]=\sigma$ and $\len(\p)$ is large, then the contribution for any $w \in W$ from $A_{\p}(w)$ to $C_{\sigma}$ will arise with a factor of $v^{n}$ with $n$ large.  Combining (1) and (2) we obtain the weak Cherednik Lemma.

\item Then we state a proportionality result in $\hqn[W]^{\vee}$ that shows $\P$ and $\Pw$ are proportional to the simple symmetrizers $\I,$ $\Iw$ respectively, with coefficients of proportionality $\cpro,$ $\cprow.$

\item Finally, we prove that the coefficients of proportionality $\cpro,$ $\cprow$ are equal to one another. In fact are both equal to $\mf{m}$ and hence lie in $\qn.$ Thus, we show that $\P$ and $\Pw$ are \emph{equal} to the symmetrizers $\hI,$ $\hIw$. As the latter are well-defined in $\qn[W]^{\vee}$ the strong Cherednik lemma follows. 

\end{enumerate}

\tpoint{Bounds in terms of $\wid(\p)$} For an element $\mv \in Q_+$, we shall write $|\mv | = \la \rho, \mv \ra$ from now on. Given $f = \sum_{\lv} c_{\lv} e^{\lv}$, with $c_{\lv}$ is some ring of coefficients, we say that $\mv$ occurs in the support of $f$ and write $\mv \in \Supp(f)$ if $[e^{\mv}] f \neq 0.$ 

\begin{nlem} \label{large-wid} Assume that $w \in W$ and $\p= \{ p < p' < \dots \}$ as above such that  $A_{\p}(w)\neq 0$. There exists a constant $\kappa=\kappa(\As)$ such that if $\mv$ occurs in the support of $A_{\p}(w),$ then \be{mv:kap} |\mv | \geq \kappa \wid(\p). \ee  \end{nlem} 

\begin{proof} Choose any two consecutive indices $g, g'$ from our chain $\p$, so $0\leq g < g'$ and  $A_{\p}(w)$ is divisible by the factor \begin{equation}\label{eq:idprod}
s_{i_p} s_{i_{p'}} \cdots s_{i_g} \left(\prod_{j=g+1}^{g'-1}\b(a_{i_j})\right) = \left(\prod_{j=g+1}^{g'-1}\b(\hat{w} a_{i_j})\right) 
\end{equation} where $\hat{w}:= s_{i_p} s_{i_{p'}} \cdots s_{i_g}.$ Now if $\gamma:= \hat{w} \av_{i_j} > 0$ then when we expand the corresponding $\b( \hat{w}  a_{i_j})$ in negative powers of the coroots, every term which arises (cf. \eqref{b:exp}) will be divisible by $e^{ - \gamma}.$ We shall argue that if $\wid(\p)$ is large, then some proportion of the $a_{i_j}$ will in fact satisfy the property that $\hat{w}(a_{i_j}) > 0$. Indeed, this follows from the following simple result. 

\begin{nclaim}\label{nclaim:boundedlongword}
Let $\Is_{-}(w):= \{ i \in \Is \mid w(\a_i) < 0 \},$ and let $\Sigma_{+, w} := \la s_i \mid i \in \Is_-(w) \ra \subset W.$ Then $\Sigma_{+, w}$ is finite. Furthermore, if $d_w$ denotes the length of the longest element in $\Sigma_{+,w},$ then there exists a $d=d(\As)$ such that $d_w\leq d$ for every $w\in W.$
\end{nclaim}
\begin{proof}
Note that for any positive linear combination $\beta $ of the roots $\{a_i\mid i\in \Is_{-}(w)\}$, we have $w(\beta )<0.$ Recall the inversion set $R(w^{-1})$ has cardinality $\ell (w)$ (cf. \ref{tpoint:inversion_sets}). For a subset $\Js \subseteq \Is$ let $\As_{\Js}$ denote the gcm consisting of the rows and columns of $\As$ corresponding to the indices in $\Js.$ Then the above means that the set of positive real roots corresponding to the gcm $\As_{\Is_{-}(w)}$ is finite. Hence indecomposable components of $\As_{\Is_{-}(w)}$ are of finite type according to the trichotomy of Proposition \ref{indec-class}, and $\Sigma_{+, w}=W(\As_{\Is_{-}(w)})$ is finite as well \cite[Proposition 4.9]{kac}. The length $d_w$ of the long word in $W(\As_{\Is_{-}(w)})$ depends only on the subset $\Is_{-}(w)\subset \Is$ (not $w$). The finite set $\Is$ has finitely many subsets. Consequently, the lengths of such long words is bounded. 
\end{proof}

Now we note that $s_{i_{g+1}} \cdots \, s_{i_{g'-1}}$ is a segment of a reduced decomposition hence itself a reduced decomposition in $W$, and so has length $\wid(\p)$.  Using the claim, we let $d=d(\As)$ denote the upper bound for the length of the long word in $\Sigma_{+, \hat{w}}.$ Then we claim the sequence $i_{g+1}, \ldots, i_{g'-1}$ contains at least $\lfloor \frac{\wid(\p)}{d+1} \rfloor$ copies of the indices belonging to the set \be{Is+} \Is_+(\hat{w}):= \{ i \in \Is \mid w(a_i) > 0 \}. \ee Indeed, if there were a consecutive string of entries in the sequence $i_{g+1}, \ldots, i_{g'-1}$ of length greater than $d$ which did not contain an element from $\Is_+(\hat{w})$ this would imply the existence of a reduced word in $\Sigma_{+, \hat{w}}$ of length greater than $d.$ By the remarks on expansions of $\b(\cdot)$ preceding the Claim, the Lemma now follows. \end{proof} 

\newcommand{\zv}{\zeta^{\vee}}

\begin{nrem} \label{rmk:largewidth} From the Lemma, if $A_{\p}(w)$ contributes to $C_{\sigma}$ with large $\wid{p}$, then the corresponding contribution arises with a power of $e^{ - \mv}$ where $\mv \in \Qv_+$ with $| \mv | $ also large. Hence for fixed $\zv \in \Qv_-$ and $\sigma \in W$ we obtain from (\ref{C:sig}) and (\ref{comp-exp}) that  
\be{rest:wid} [e^{\zv}] C_{\sigma}=  \sum_{\substack{w, \p \\ [\p]=\sigma,\\ \wid(\p) < C }} [e^{\zv}] A_{\p}(w) ,\ee 
where $C:=C(\As, \zv)  > 0$ is some constant depending on $\zv$ and $\As.$ Note that the right hand side is still an infinite sum as there can be arbitrarily long Weyl group elements which admit $\p$ of small width. However, such $\p$ will necessarily have large length.

\end{nrem}

\tpoint{Bounds in terms of $\len(\p)$}  We now analyze the contributions of $A_{\p}(w)$ in terms of $\len(\p).$

\begin{nlem} \label{large-len} Fix $\mv \in \Qv$ and $\sigma \in W$, and let $\p$ be such that $[\p]=\sigma.$ Then for any $w \in W$, the expression $[e^{\mv}]A_{\p}(w)$  is divisible by $v^k$ for $k= \mathsf{O}(\len(\p))$ as $\len(\p) \rr \infty$ where the implicit constant\footnote{This is the usual ``big O'' notation} is dependent on $\mv$ and $R.$ \end{nlem} 
\begin{proof} Consider again the expression \eqref{eq:prodform} for $A_{\p}(w)$, and let us now focus on the second factor, namely $\prod _{j\in \p}\c (\beta_j).$ Recall that by \eqref{eq:betanotation}, 
\begin{equation}\label{eq:list_cterms.1}
\beta_p:=a_{i_{p}},\ \beta_{p'}:=s_{i_{p}}(a_{i_{p'}}),\ \beta_{p''}:=s_{i_{p}}s_{i_{p'}}(a_{i_{p''}}), \ldots   
\end{equation}
and  $\sigma= [\p] = s_{i_p}s_{i_{p'}}s_{i_{p''}} \cdots .$  For simplicity, let us focus on the case that $\sigma=1.$  By Lemma \ref{lem:inverseset}, the set $\{\beta_p,\beta_{p'},\beta_{p''},\ldots \}$ consists \emph{entirely} of pairs $\{a,-a\},$ i.e.,  the factor $\prod _{j\in \p}\c (\beta_j)$ is a product of expressions of the form $\c(a) \c(-a)$ with $a \in R_+$. By the expansion in \eqref{c:aa-exp}, each such factor either contributes $v$ or $e^{- \av}$ to $A_{\p}(w).$ As $\mv$ is fixed, the number of times such a factor can contribute $e^{-\av}$ is finite (and depends on $\mv$). As $\len(\p) \rr \infty$ and the number of pairs $\{ a, -a \}$ in the list of $\beta$'s above grows, and so their contribution to $e^{\mv}[A_{\p}(w)]$ will almost always  appear with a factor of $v.$ The Lemma follows from these observations. \end{proof}

\tpoint{Proof of Theorem \ref{form:sym}(1)} We are now in a position to conclude the proof of the weak Cherednik Lemma. Suppose $\sigma \in W$ and $\mv \in \Lv$ are fixed, and let us analyze $[e^{\mv}]C_{\sigma}.$ By Remark \ref{rmk:largewidth}, we may restrict to terms with width bounded by $C:= C(\As, \mv)$. Since $\wid(\p)(\len(\p)+1)\geq \ell(w)$ if $\wid(\p)\leq C$, then $\len(\p)\geq \ell(w)/C-1.$ Hence, for $\ell(w)$ large any contribution must come from $\p$ with $\len(\p)$ large.  By the previous Lemma, the contribution $[e^{\mv}] A_{\p}(w)$ with a power of $v$ that is proportional to $\len(\p)$ and hence $\ell(w)$ (since $\wid(\p)$ was assumed bounded). So, for fixed $\mv$ and fixed $k,$ only finitely many $A_{\p}(w)$ may contribute to the coefficient of $v^k$ in $[e^{\mv}]C_{\sigma}.$ This implies that $[e^{\mv}]C_{\sigma} \in \hqn.$
 
\tpoint{Proof of Theorem \ref{form:sym}(2): proportionality of symmetrizers} In the previous section, we have made sense of the element $\mc{P}, \Pw$ as elements in $\hqn[W]^{\vee}.$ The expressions $\I, \Iw$ are also well-defined elements in $\hqn[W]^{\vee}$ (actually in $\qn[W]^{\vee}$) and we may show the following identities in $\hqn[W]^{\vee}$

\begin{nprop}\label{prop:sph} There exist $W$-invariant elements $\mf{c}$ and $\cprow$ in $\hqn$ such that as elements in $\hqn[W]^{\vee}$  \be{eq:prop:sph} \begin{array}{lcr} \mc{P} =  \mf{c} \, \I & \text{ and } &  \Pw = \cprow \,  \Iw \end{array} \ee \end{nprop}
For the first equality, see Proposition 7.3.12 in \cite{bkp} (which is written for affine root systems, but the same argument works here). As for the second, see Proposition 6.5 in \cite{pat:whit}. Note that in \emph{loc. cit} these identities are proven under the assumption that  $\mc{P},$ $\Pw$ are in $\qn[W]^{\vee}$, however the proofs given in these sources continue to hold under the assumption that $\mc{P},$ $\Pw$ only lie in $\hqn[W]^{\vee}.$

\tpoint{On the equality $\cpro = \cprow$} In \cite{pat:whit}, it was argued from $p$-adic considerations, namely a comparison between Whittaker and spherical functions, that $\cprow= \cpro$ (when $v$ is specialized to some power of a prime). In fact, one does not need any such $p$-adic considerations and a simple, direct argument is possible. 

\begin{ncor} \label{consts} We have $\mf{c} = \cprow$. \end{ncor}
\begin{proof} Comparing the expansions $\mc{P} = \sum_{\sigma} C_{\sigma}[\sigma]$ and $\mathcal{P}_{\flat}= \sum_{\sigma \in W} C^{\flat}_{\sigma}[\sigma]$ with the definition of $\I$ and $\Iw$ we find that $\mf{c} \Delta = C_1 $ and $\mf{c}^\flat \Delta= C_1^{\flat}.$ Thus if we show that $C_1=C_1^{\flat}$ the Corollary follows. However, the fact that $C_1= C_1^{\flat}$ is a consequence of proof  of Lemma \ref{large-len}. Indeed, first note that \be{C:1} \begin{array}{lcr} C_1 = \sum_{w, \p , \, [\p]=1} A_{\p}(w)& \text{ and } C_1^{\flat} = \sum_{w, \p , \,  [\p]=1} A^{\flat}_{\p}(w) \end{array} \ee where \be{A:p:sharp} \begin{array}{lcr} A^{\flat}_{\p}(w)= \prod_{j\notin \p}\b(\beta_j) \cdot \prod _{j\in \p}\c^{\flat} (\beta _j) & \text{ and } & A_{\p}(w)= \prod_{j\notin \p}\b(\beta_j) \cdot \prod _{j\in \p}\c (\beta _j) \end{array} \ee with the $\beta_j$ defined as in (\ref{eq:betanotation}). By Lemma \ref{lem:inverseset} the collection of all $\beta_j, j \in \p$ which appear in either product are a collection of pairs $\{ a, -a \}$ for $a \in R_+.$ But $\c^{\flat}(a) \c^{\flat}(-a) = \c(a) \c(-a)$ as we observed in (\ref{c:aa}), and so indeed $C_1 = C_1^{\flat}$. \end{proof}

\tpoint{On $\cpro,$ $\cprow$ and $\mf{m}$} To complete (4) of the plan outlined in \ref{pf-out} it remains to show that $\cpro =\mf{m} .$ 

\begin{nlem}\label{lem:cproismfm}
We have $\cpro = \mf{m}.$
\end{nlem}
\begin{proof}
Note that by Proposition \ref{prop:sph} we have $\mc{P}=\cpro \cdot \I$ and $\cpro \in \hqn$ hence it suffices to show that $\mc{P}$ and $\mf{m}\I $ are equal when ``applied'' to the element $1=e^0 \in \hqn.$ Note that $\T_{w_i}(1)=v,$ and hence $\mc{P}(1) = \sum_{w\in W}v^{\ell (w)}.$ By Proposition \ref{vis} this is indeed equal to $\mf{m}\I (1).$ 
\end{proof}

\tpoint{Proof of Theorem \ref{form:sym}(3): the strong Cherednik Lemma} Combining Proposition \ref{prop:sph}, Corollary \ref{consts} and Lemma \ref{lem:cproismfm}, we see that as elements of $\hqn[W]^{\vee}$ we have $\P = \hI$ and $\Pw = \hIw.$ However, both $\mf{m} \, \I$ and $\mf{m} \, \Iw$ clearly lie in $\qn[W]^{\vee}$, therefore $\P$ and $\Pw$ must lie in $\qn[W]^{\vee},$ and the identities \eqref{eq:prop:sph} also hold in this smaller space. 

\subsection{Metaplectic Symmetrizers} \label{sec-met-sym}

\newcommand{\mrd}{\wt{\mf{D}}}
\newcommand{\tQv}{\wt{Q^{\vee}}}
\newcommand{\tLam}{\wt{\Lambda}}
\newcommand{\tqnf}{\mc{\wt{Q}}_v^{\fin}}

\tpoint{Metaplectic Notations} \label{cg-not} Let $(I, \cdot, \mf{D})$ be a root datum where $\mf{D}= (\Lv, \{ \av_i \} , \Lambda, \{ a_i \})$. Let $(\Qs, n)$ be metaplectic structure on this root datum, and construct the metaplectic root datum $\wt{\mf{D}}= (\tLv, \{ \tav_i \}, \tLam, \{ \ta_i \} )$ as in \S \ref{s:met-rts}, with associated gcm $\wt{\As}$. Let $\Qv$ and $\tQv$ be the coroot lattices associated to $\mf{D}$ and $\wt{\mf{D}}$ respectively and note that by definition we have $\tQv \subset \Qv$ as well as $\tLv \subset \Lv$. Usually we just write $n_i$ in place of $\n(\av_i),$ though in places where we use both $\av_i$ and $n_i \av,$ we still write $\n(\av_i).$ Also in this section we fix a family of formal parameters $v$ and $\gf_k$ (with $k \in \zee)$ that satisfy the conditions (cf. \S \ref{p-spe}) \be{g:form} \gf_k = \gf_l \text{ if } n | k-l, \, \gf_0 = -1, \text{ and if } k \neq 0 \mod n, \text{ then } \gf_k \gf_{-k} = v^{-1} . \ee In analogy with (\ref{Cfin}), but keeping in mind the $v^{-1}$ appearing above, we set \be{Cvg} \C^{\fin}_{v, \gf} := \C[v, v^{-1}, (\gf_k)_{k \in \zee}] . \ee Define the rings $\tqnf$, $\thqn$ as before but replacing the coroots $\Pv$ with their metaplectic counterparts $\{ \tav_1, \ldots, \tav_r \}.$

\tpoint{Chinta-Gunnells action} \label{subs:CGaction} Let $\mult$ be the smallest subset of $\C^{\fin}_{v, \gf}[\tLv]$ closed under multiplication containing $1 - e^{-\tav_i}$ and $1 - v e^{-\tav_i}$ for every $i \in I,$ and let $\ring$ (respectively, $\nice$) denote the localization of $\C^{\fin}_{v, \gf}[\Lv]$ (respectively, of $\C^{\fin}_{v, \gf}[\tLv]$) by $\mult .$  For $i \in \Is$ we also define the residue map \be{resid} \resi_{n_i}: \zee \rr \{ 0, 1, \ldots, n_i-1 \} \ee in the obvious manner. Now for $\lv \in \Lv$ and $a \in \Pi$, following Chinta and Gunnells \cite{cg} we set  
{\small \be{cg-act} s_a \star e^{\lv} = \frac{e^{s_a \lv}}{1 - v e^{-\tav}} \left[ (1-v) e^{ \resi_{\n(\av)} \left( \frac{ B(\lv, \av)}{\Q(\av)} \right) \av } - v \gf_{ \Q(\av) + \B(\lv, \av) } e^{ \tav - \av} (1 - e^{-\tav}) \right].\ee} Extend this by $\C^{\fin}_{v, \gf}$-linearity to define $s_a \star f$ for every $f \in \C^{\fin}_{v, \gf}[\Lv],$ and then use the formula \be{cg-act-quot} s_a \star \frac{f}{h} = \frac{s_a \star f}{h^{s_a}} \text{ for } f \in \C^{\fin}_{v, \gf}[\Lv], h\in \mult, \ee to extend this to $s_a \star -: \ring \rr \ring.$ 

\begin{nrem} \label{CG:polv} In fact, the presence of the $v$ in front of the $\gf$ in (\ref{cg-act}) together with (\ref{g:form}) allows us to assert the following: for $\lv \in \Lv$, if we write  $s_a \star e^{\lv}= \sum_{\mv} c_{\mv} e^{\mv}$, then each $c_{\mv}$ can be written as a polynomial in $v$ (so there are no $v^{-1}$ which appear) and the variables $\gf_0, \ldots, \gf_{n-1}.$ \end{nrem} 

 Let $f \in \C^{\fin}_{v, \gf}[\Lv]$ and $a \in \Pi.$ Then we have \be{cgw} s_a \star (h f) = h^{s_a} (s_a \star f) \ee for $h\in \nice$ and $f \in \C^{\fin}_{v, \gf}[\Lv]$ (note: this fails for general $h$). In the sequel we shall need to observe the following result, which follows immediately from the definition (\ref{cg-act}) and (\ref{bq:db}).
  
\begin{nlem} \label{exp:cg:a} Let $\lv \in \Lv$ and $a \in \Pi.$ Then the expansion of $s_a \star e^{\lv}$ in negative powers of the coroots satisfies the following condition: every $\mv$ which lies in the support of $e^{ - \tav} \, s_a \star e^{\lv}$ satisfies the inequality $\mv \leq \lv.$ \end{nlem} 

\tpoint{Metaplectic Demazure-Lusztig operators} \label{s:metDL} We can now introduce the metaplectic version of the Demazure-Lusztig operators as before, but now using the metaplectic root datum. Set  
\be{met:cb} \begin{array}{lcr} \cw(\tav) = \frac{ v - 1}{1 - e^{ \tav}}, &  \c(\tav) = \frac{1 - v e^{  \tav}}{1 - e^{ \tav}}, & \cw(\tav) = \frac{1 - v e^{ - \tav}}{1 - e^{ \tav}}. \end{array} \ee and for each $a \in \Pi,$ we define the elements in $\C^{\fin}_{v, \gf}(\tQv)[W]^{\vee},$  \be{T:c-b} \Tm_a &=& \c(\tav) [s_a]  + \b(\tav) [1] \\ 
\Tmw_a &=& \cw(\tav) [s_a]  + \b(\tav) [1]. \ee Then we consider their action on $\C_{v, \gf, S}[\Lv]$ by the formulas \be{Tma:act} \Tm_a(e^{\lv}) &=& \c(\tav) s_a \star e^{\lv} + \b(\tav)  e^{\lv} \\ \Tmw_a(e^{\lv}) &=& \cw(\tav) s_a \star e^{\lv} + \b(\tav)  e^{\lv} \ee etc. As noted in \cite{pat:pus}, $\Tm_a, \Tmw_a$ both preserve $\C^{\fin}_{v, \gf}[\Lv].$

\begin{nrem} \label{T:polv} Using Remark \ref{CG:polv}, note that if $f \in \C^{\fin}_{v, \gf}[\Lv]$ is actually a polynomial in $v$ and $\gf_0, \ldots, \gf_{n-1}$, then the same is true of its image under $\Tm_a$ and $\Tmw_a$. \end{nrem} 

\tpoint{Braid relations} \label{w-defs} Let $w \in W$ and choose any reduced decomposition $w=s_{b_1} \cdots s_{b_r}$ with $b_i \in \Pi$ for $i=1, \ldots, r.$ Then we may define operators $w \star - : \nice \rr \nice$ and $\Tm_w: \C^{\fin}_{v, \gf}[\Lv] \rr \C^{\fin}_{v, \gf}[\Lv]$ as, \be{w-def} w \star f &:=& s_{b_1} \star \cdots \star s_{b_r} \star f \text{ for } f \in \nice  \\ \label {Tw:def} \Tm_w (f) &:=& \Tm_{b_1} \cdots \Tm_{b_r} (f) \text{ for } f \in \C^{\fin}_{v, \gf}[\Lv] \ee but it remains to be seen that these definitions are well-defined.

\begin{nprop} \cite[Theorem 3.2]{cg}, \cite[Proposition 2.11]{lee:zh}  \begin{enumerate}[(a)] \item The operation $ \star$ defines an action of $W(\As)$ on $\C^{\fin}_{v, \gf, S}[\Lv].$
\item The expressions $\Tm_a$ satisfy the braid relations (\ref{bd:rel}). \footnote{We remark here that the condition that $\gf_k \gf_{-k} =v$ is used to verify the property $s_a^2 =1$}
\end{enumerate}
\end{nprop}

\begin{proof} This amounts to a rank two check, and in fact we may assume our rank two root system is of finite type.  Indeed by \eqref{bd:rel} and Table \ref{h:tab} if $s_i$ and $s_j$ satisfy a braid relation ($h_{ij}<\infty $), then $a_{ij}\cdot a_{ji}<4.$ This implies that (since $\{a_{ij},a_{ji}\}$ may only be $\{0\},$ $\{-1,-1\},$ $\{-1,-2\},$ or $\{-1,-3\}$) the rank $2$ root system determined by $a_i$ and $a_j$ is of finite type ($A_1\times A_1,$ $A_2,$ $B_2=C_2$ or $G_2$). 

Now for (a), we may refer to \cite[Lemma 4.3, Lemma 6.5]{pat:pus} or \cite[Theorem 3.2]{cg} or \cite[Proposition 2.11]{lee:zh}, and for (b) to \cite[Proposition B.1.]{pat:pus} or \cite[Proposition 7.]{cgp}.
\end{proof}

\newcommand{\tm}{\widetilde{m}}

\tpoint{``Simple'' Hecke symmetrizers}We define the analogue of \eqref{Del} for our root system $\wt{R}$ as \be{mD} \mD:= \prod_{\ta \in \tR_+} \frac{ 1- v e^{-\tav}}{1 - e^{-\tav}}.\ee It lies in $\tqnf[W]^{\vee},$ and similarly we define the objects $\wt{\mf{m}}, \, \wt{\cpro}, \,  \wt{\mf{c}}^{\flat}.$ In the case when $R$ (and hence also $\wt{R}$) is affine, we let $\wt{\cc}$ be the minimal imaginary root for $\wt{R}$ and set  (\ref{corr}),  $\wt{\mf{m}}:=  \ct(\wt{\Delta}^{-1}).$ In the case when $R$ is of untwisted of ADE type, we have \be{tm} \wt{\mf{m}}= \prod_{j=1}^{\ell} \prod_{i=1}^{\infty} \frac{1 - v^{\tm_j}e^{-i \tcc } }{1 - v^{\tm_j+1} e^{-i \tcc }}, 
\end{eqnarray} where the $\tm_j$ are the exponents of $\tR_o$ (the underlying finite-dimensional root system to $\tR$).  For general affine root systems $R$, again we refer to \cite[5.8.20]{mac:aff} for the precise formula. Also, we consider the metaplectic analogues of the simple symmetrizers \be{mI} \begin{array}{lcr} \mI = \sum_{ w \in W}  \mD^w [w] & \text{ and } & \mIw := \mD \, \sum_{w \in W}  (-1)^{\ell(w)}  \left( \prod_{\ta \in \tR(w) } e^{ - \tav } \right)  [w] . \end{array} \ee where again $\mI, \, \mIw$ and $\mm \mI, \, \mm\mIw$ are well-defined elements in $\tQ_v[W]^{\vee}.$

\tpoint{Metaplectic symmetrizers of Hecke type} Define the \emph{metaplectic symmetrizers of Hecke type} \be{m:Ts} \begin{array}{lcr} \Pm:= \sum_{ w \in W} \Tm_w & \text{ and } & \Pmw := \sum_{ w \in W} \Tmw_w \end{array}. \ee  
Analogously to \eqref{hecke-sym} let us write 
\be{Tm:exp} \Tm_w  = \sum_{\sigma \in W} \tA_{\sigma}(w) [\sigma] & \text{ and } & \Tmw_w  = \sum_{\sigma \in W} \tAw_{\sigma}(w) [\sigma] \\
\tC_{\sigma} = \sum_{w \in W} \tA_{\sigma}(w) & \text{ and } & \tCw_{\sigma} = \sum_{w \in W} \tAw_{\sigma}(w)  \ee 
so that we formally have 
\be{P:Cmsig} \begin{array}{lcr} \Pm = \sum_{ w \in W} \Tm_w = \sum_{\sigma \in W} \tC_{\sigma} [\sigma] & \text{ and } & \Pmw = \sum_{ w \in W} \Tmw_w = \sum_{\sigma \in W} \tCw_{\sigma} [\sigma]. \end{array}  \ee

\begin{nthm} \label{form:sym_met} The elements $\Pm$ and $\Pmw$ above have well-defined expansions. More precisely, we have that \begin{enumerate} \item (Weak Cherednik Lemma) the elements $\Pm$ and $\Pmw$ have well-defined expansions in $\thqn[W]^{\vee},$ i.e. the elements $\tC_{\sigma}, \tCw_{\sigma}$ for $\sigma \in W$ have expansions which lie in $\thqn.$

\newcommand{\mcpro}{\tilde{\cpro}}
\newcommand{\mcprow}{{\tilde{\cpro}^{\flat}}}

\item There exist $W$-invariant elements $\mcpro$ and $\mcprow$ in $\thqn$ such that as elements in $\thqn[W]^{\vee}$  \be{prop:metsph} \begin{array}{lcr} \Pm =  \mcpro \, \mI & \text{ and } &  \Pmw = \mcprow \,  \mIw \end{array}. \ee Moreover, we have elements $\mcpro = \mcprow = \ct(\mD ^{-1}).$ 

 \item (Strong Cherednik Lemma) the elements $\Pm$ and $\Pmw$ have well-defined expansions in $\tqn[W]^{\vee},$  i.e. the elements $\tC_{\sigma}, \tCw_{\sigma}$ for $\sigma \in W$ have expansions which lie in $\tqn.$

\end{enumerate}
\end{nthm}

This is just Theorem \ref{form:sym} for the metaplectic root system.

\spoint Our final task is to show that $\mf{c} \, \mIw$ can be applied to the element $e^{\lv}$ with $\lv \in \Lv_+$ to obtain a well-defined $v$-finite element. For $\lv \in \Lv_+$ we denote by \be{cat:o} \mf{o}(\lv) = \{ \mv \in \Lv \mid \mv \leq \lv \}. \ee Following Looijenga \cite{loo}, we also define the \emph{dual coweight algebra} as 

\begin{de} \label{loo-space} We shall write $\C_{v, \gf \leq}[\Lv]$ (or $\C^{\fin}_{v,\gf \leq}[\Lv]$) for the set of all $f = \sum_{\lv \in \Lv} c_{\lv} e^{\lv}$ with $c_{\lv} \in \C_{v, \gf}$ (resp. $c_{\lv} \in \C^{\fin}_{v, \gf}$) such that there exists $\lambda_1, \ldots, \lambda_r \in \Lv_+$ so that  $\Supp(f):= \{ \mv \in \Lv \mid c_{\mv} \neq 0 \}$ is contained in $\mf{o}(\lambda_1) \cup \cdots \cup \mf{o}(\lambda_r).$ \end{de} 

Note that $\hqn \subset \C_{v, \gf, \leq}[\Lv]$ and $\qnf \subset \C^{\fin}_{v, \gf, \leq}[\Lv]$ since $\mf{o}(0) = \Qv_-$ where $0 \in \Lv_+$ is the zero coweight. 

\begin{nlem} \label{well-def-met-Iw} For each $\lv \in \Lv_+$ the expression \be{met-Iw-lv}  \mIw(e^{\lv}) =  \, \mD \, \sum_{w \in W}  (-1)^{\ell(w)}  \left( \prod_{\ta \in \tR(w) } e^{ - \tav } \right)  (w \star e^{\lv}) \ee defines a well-defined element of $\C^{\fin}_{v, \gf, \leq}[\Lv].$ \end{nlem}
\begin{proof} Fix $\mv \in \Lv.$ We will show that there are only finitely many $w \in W$ such that $e^{\mv}$ appears in corresponding summand of the right hand side of (\ref{met-Iw-lv}). First note that the product $ \left( \prod_{a \in \tR(w) } e^{ - \tav } \right)$ is of the form $e^{ - \beta^{\vee}(w)}$ with $\beta^{\vee} \in \tQv_-$ is such that $\la \rho, \beta^{\vee}(w) \ra$ is large if $\ell(w)$ is large. As there are only finitely many $w \in W$ such that $\mv \geq \lv - \beta^{\vee}(w)$ in the dominance order, the following claim, which is proven by induction using Lemma \ref{exp:cg:a}, concludes the proof

\begin{nclaim} \cite[Lemma 2.15]{lee:zh} If $\mv$ lies in the support of (the expansion of)  $\prod_{\ta \in \tR(w)} e^{ - \ta} \,  \, w \star e^{\lv}$ then $\mv \leq \lv.$ \end{nclaim} \end{proof}

\tpoint{``Operator'' Casselman-Shalika formula} Finally we obtain the following result which will be used in our computation of the metaplectic Whittaker functions in \S \ref{s:CS}. 

\newcommand{\Cf}{\C^{\fin}}

\begin{ncor} For $\lv \in \Lv_+$ we have the following equality in $\Cf_{v, \leq}[\Lv]$ \be{fin-comb} \Pmw(e^{\lv} ) = \mm \,  \mIw(e^{\lv}) =  \mm \, \mD \, \sum_{w \in W}  (-1)^{\ell(w)}  \left( \prod_{a \in \tR(w) } e^{ - \tav } \right)  (w \star e^{\lv}) \ee \end{ncor}

\begin{nrem} \label{fin-comb:v} In light of Remark \ref{CG:polv}, the expression on the right hand side can actually be written in the form $\sum_{\mv \in \Lv} c_{\mv}e^{\mv}$ where $c_{\mv}$ is a polynomial in $v, \gf_0, \ldots, \gf_{n-1}$.  \end{nrem} 

\section{Kac-Moody groups}

\renewcommand{\un}{\mc{U}}
\newcommand{\dpow}[2]{{#1}^{(#2)}}

\newcommand{\bU}{\mathbf{U}}

\subsection{The Tits Kac-Moody group functor } 

Let $(\Is, \cdot, \mf{D})$ be a root datum with associated gcm $\As$ and write $\mf{D}= (Y, \{ \av_i \}, X, \{ a_i \}).$ Our aim in this section is to review some aspects of the construction of the functor $\bG_{\mf{D}}$ due to Tits \cite{tits:km} from rings to groups. In \S \ref{sec:N} we assume that $\mf{D}$ is of simply-connected type.  

Let $\mf{g}_{\D}$ be the Kac-Moody algebra as in \S \ref{s:gD}. As $\mf{D}$ is fixed, we shall sometimes drop it from our notation, e.g. we write $\mf{g}$ in place of $\mf{g}_{\D}$ and $Q, \Qv$ in place of $Q_{\D}, \Qv_{\D}$, etc. The roots (coroots) of this Lie algebra will be written as $R$ (resp. $R^{\vee}$). We shall also write $\Pi$ and $\Pi^{\vee}$ for the roots and coroots (identified with $\{ a_i \}$ and $\{ \av_i \}$). 

Recall our conventions for functors of groups from \S \ref{functor}. Let $F$ be an arbitrary field, and $S$ an arbitrary commutative ring with unit, and write $S^*$ for its units.

\tpoint{The torus $H$} The group which shall play the role of the maximal torus is $\bH(S) := \Hom_{\zee}(X, S^*).$  For $s \in S$ and $\lv \in Y$ we write $s^{\lv} \in \bH(S)$ for the element which sends $\mu \in X$ to  $s^{\langle \lv, \mu \rangle}.$ For a fixed basis\footnote{As in \S \ref{more-sc-not}, we use $I_e$ to denote a set indexing a basis of $Y,$ though of course the basis need not be comprised of coroots.} $\{ \Lambda_i \}_{i \in \Is_e}$ of $Y$, every element in $\bH(S)$ may be written uniquely in the form $\prod_{i \in \Is_e} s_i^{\Lambda_i}$ with $s_i \in S^*.$ We also use the following notation: if $t \in \bH(S)$ we shall write $t(x) \in S^*$ for the application of $t \in \Hom_{\zee}(X, S^*)$ to $x \in X.$  The action of $W(\As)$ on $X$ induces one on $\bH(S)$ and we shall denote the image under $w \in W(\As)$ of $t \in \bH(S)$ as $w(t).$ For $\lv \in Y, s \in S^*$ and $i \in \Is$ with $s_i \in W $ the corresponding simple reflection, we have \be{w:H:act} s_i (r^{\lv}) = r^{\lv} r^{- \la \lv,\ x_i \ra \av } \text{ for } r \in S^*. \ee

\newcommand{\bSt}{\mathbf{St}}
\newcommand{\bB}{\mathbf{B}}
\newcommand{\kmr}[1]{\textbf{KM} {#1}}
\newcommand{\gad}{\mathbb{G}_{+}}

\tpoint{Unipotent subgroups} In this paragraph we introduce certain unipotent subgroups. To do so we recall the notion of prenilpotent pairs from \S \ref{prenilp} as well as the notion of dual root bases from \S \ref{s:la} (5). For each $a \in R_{re}$ we let $\bU_a$ be the group scheme over $\zee$ isomorphic to the one-dimensional additive group scheme $\gad$ with Lie algebra $L_{a, \zee}$ equal to the $\zee$-subalgebra of $\mf{g}(\As)$ generated by the dual base $E_a$. This determines $\bU_a$ uniquely up to unique isomorphism. More concretely, every choice of $e_a \in E_a$ determines an isomorphism $x_a: \gad \rr \bU_a$ and so we write $\bU_a(S)=  \{ x_a(s) \mid s \in S \},$ where of course one has \be{one-rel} x_a(s) x_a(t) = x_a(s+t) \text{ for } s, t \in S. \ee  For any nilpotent set of roots $\Psi \subset R_{re}$ the complex Lie algebra $L_{\Psi}:= \oplus_{a \in \Psi} \mf{g}^a$ is easily seen to be nilpotent, and we may construct the corresponding complex, unipotent group $U_{\Psi}.$ Tits has shown \cite[Proposition 1]{tits:km} that there exists a group scheme $\bU_{\Psi}$ equipped with inclusions of group schemes $\bU_a \hookrightarrow \bU_{\Psi}$ for each $a \in \Psi$ satisfying two conditions: first $\bU_{\Psi}(\C)= U_{\Psi};$ and second,  if we pick any total order on $\Psi,$ the resulting map, \be{Psi:sch} \prod_{a \in \Psi} \bU_a \rr \bU_{\Psi} \ee is an isomorphism of schemes (not groups). Concretely, picking $e_a \in E_a$ for each $a \in R_{re}$ this implies that for any prenilpotent pair $\{a, b \}$ there exist well-defined integers $k(a, b; c),\, c \in ]a, b[$ such that in $\bU_{[a,b]},$ \be{kmg:sg_rel1}  [ x_a(u), x_b(u') ] = \prod_{\stackrel{c \in \, ]a, b[ \, }{c= m a + n b}} \, x_c(k(a, b; c) u^m u'^n) \text{ for } u, u' \in S. \ee   Moreover, one can show that if $\Psi' \subset \Psi$ is an ideal, i.e. if $a \in \Psi',$ $b \in \Psi$ such that if $a + b \in R$ then $a+ b \in \Psi',$ then $\bU_{\Psi'}(S) \subset \bU_{\Psi}(S)$ is a normal subgroup.

\spoint \label{s:stein-functor} The \emph{ Steinberg functor} $\bSt$ is defined to be the quotient of the free product of the groups $\bU_a(S), a \in R_{re}$ by the smallest normal subgroup containing the relations (\ref{kmg:sg_rel1}) and (\ref{one-rel}).  To understand the structure of this group, we define  \be{ws} \begin{array}{lcr} w_{a}(u) := x_{a}(u) x_{-a}(-u^{-1}) x_{a}(u), &    h_{a}(u) := w_{a}(u) w_{a}(-1)  & \text{ for } a \in \Pi, u \in S^*.   \end{array} \ee From the definition (\ref{ws}) we see immediately that $ w_a(s)^{-1} = w_a(-s).$ As a shorthand, we shall henceforth write $\dw_a:= w_a(-1)$ for $a \in \Pi.$  The \emph{Tits functor} $\bG_{\mf{D}}$ (or usually just $\bG$ for short) assigns to a ring $S$ the quotient of the free product of $\bSt(S)$ and $\bH(S)$ by the following relations: \begin{enumerate}[\textbf{KM} 1]
\item $t x_a(s) t^{-1} = x_a(t(a) r)$ where $a \in \Pi, s \in S$ and $t \in \bH(S)$ 
\item For $a \in \Pi$ and $t \in \bH(S)$ we require $\dw_a t \dw_a^{-1} = s_a(t)$ (cf. \eqref{w:H:act}) 
\item For $a \in \Pi$ and $s \in S^*$ we require $h_a(s) = s^{\av}.$ 
\item For $a \in \Pi$, $b \in R_{re},$ and $s \in S$ we have $\dw_a x_b(s) \dw_a^{-1} = x_{w_a(b)}(\eta(a, b) \, s ) $ where $\eta(a, b) \in \{ \pm 1 \}.$ 
\end{enumerate}

\label{signs} Let us now comment on the signs $\eta(a, b)$ for $a, b \in \Pi$. For each $a \in R_{re}$ we have fixed some $e_a \in E_a$. Recall that we have defined for each $a \in \Pi$ an automorphism $s_a^*$ using the formula (\ref{s:aut}). From \eqref{Ea}, we can define $\eta(a, b) \in \pm 1$ so that \be{} s_a^*(e_b) = \eta(a, b) e_b \text{ for } a \in \Pi, b \in R_{re}. \ee This procedure may be extended in a natural way to define the signs $\eta(a, b)$ for each $a, b \in R_{re}.$ and one may verify the following rules \cite[Lemma 5.1(c)]{mat}. Let $a, b \in R_{re}$, \begin{enumerate} \item $\eta(a, b) = \eta(a, -b)$, $\eta(a, a) = \eta(a, -a) =-1$
\item $\eta(a, b)=1$ if $a \pm b \neq 0$ and $a \pm b \notin R_{re}$
\item $\eta(a, b) \, \eta(b, a) = -1$ if $ \la a, \bv \ra = \la b, \av \ra =-1$
\item $\eta(a, b)=-1$ if $\la a, \bv \ra = 0$ and $a \pm b \in R_{re}$
\end{enumerate}  

\tpoint{Subgroups}  \label{tits:sbgp}  Let $\bU$ and $\bU^-$ to be the subfunctors of $\bG$ generated by $\bU_a$ for $a \in R_{re, +}$ and $a \in R_{re, -}$ respectively. We also let $\bB$  (resp. $\bB^-$) be the subfunctor generated by $\bU$ and $\bH$ (resp. $\bU^-$ and $\bH$). For any $a \in R_{re}$ we also let $\bB_a$ be the subfunctor generated by $\bU_a$ and $\bH.$ Finally we let $\bG_a $ be the subfunctor generated by $\bB_a, \bB_{-a}$ for $a \in R_{re}.$ We note the following two properties of Tits' construction: for any field $F$ the natural maps $\bH(F) \rr \bG(F)$ and $\bU_{\Psi}(F) \rr \bG(F)$ (for $\Psi$ a nilpotent set) are injections. 

Recalling again our conventions from \S \ref{functor}, we define the groups $G, B, U, H$, etc. 

\tpoint{Tits axioms} \label{s:axioms} In the finite-dimensional situation, the Tits axioms (or BN-pair axioms) imply the Bruhat decomposition as well as a number of other results for the structure theory of the group (cf. \cite[Ch. IV]{bour}). For general Kac-Moody root systems, there are two opposite classes of parabolic subgroups and one has Bruhat decompositions with respect to each of these as well as mixed or Birkhoff-type decompositions. In \cite[\S 5]{tits:km}, an axiomatic framework for studying the algebraic structure theory of Kac-Moody groups is introduced, which we review here. The notation here is related, but distinct, from the other sections of this paper since we would like to also use these same axioms again for metaplectic covers.

The starting point is gcm $\As$ and a set of simple roots $\Pi$ from which we can construct a Weyl group $W$ and the set $R_{re}$ of real roots as the orbit of $\Pi$ under $W$.  Assume given a group $G$, together with a system of subgroups $(B_a)_{a \in R_{re}}$ such that $G$ is generated by $(B_a)_{a \in R_{re}}$. Let $B$ and $B^-$ denote the subgroups of $G$ generated by $B_a$ for all positive and negative (real) roots respectively. Let us also set \be{H:int} H:= \bigcap_{a \in R_{re}} \, B_a, \ee and write $G_a$ for the group generated by $B_a$ and $B_{-a}$. One then imposes the following axioms on this data.

\newcommand{\rd}[1]{{\textbf{RD {#1}}}}

\begin{enumerate}[\bf{RD} 1]

\item There exist an element  in $G_{a_i}$ which, for each $b \in R,$ conjugates $B_b$ to $B_{w_i b}.$  

\vspace{1.1mm}
\noindent For each $i \in I$, write $\dw_i$ for an element as stipulated in \rd{1}. Let $N$ be the subgroup generated by $\dw_i$ and $H$.
\vspace{1.1mm}

\item Each $B_a\,  (a \in R_{re})$ has a semi-direct product decomposition $B_a = U_a \ltimes H$ such that $U_a$ are permuted under conjugation by the elements of $N.$ Moreover, if $a, b$ form a prenilpotent pair of distinct roots, then the commutator of $U_a$ and $U_b$ is contained in the group generated by $U_{\gamma}$ for $\gamma \in [a, b] \setminus \{ a, b \}.$

\item We have $B_{a_i} \cap B_{-a_i}=H$ for each $i \in \Is$.

\item The set $B_{a_i} \setminus G_{a_i} / B_{a_i} $ consists of two elements.

\item For each $i \in \Is$, the group $B_{a_i}$ is not contained in $B^-$ and the group $B_{-a_i}$ is not contained in $B$. 

\end{enumerate}

From these axioms, a number of results follow (cf. \cite[\S 5]{tits:km}). In particular the pairs $(B, N)$ and $(B^-, N)$ both form Tits systems in the language of \cite[Ch. 4, \S 2]{bour}, and we have the following.

\begin{nthm} \label{bruhat-gen} Assume $(G, (B_a)_{a \in R_{re}})$ satisfy the axioms \rd{1}- \rd{5}.

\begin{enumerate} 
\item \label{phi:N} \cite[Lemma 4, \S 5.4]{tits:km} There exists a unique homomorphism $\phi: N \rr W$ such that for $n \in N$ and $a \in R_{re}$ one has $n B_a n^{-1} = B_{\phi(n)(a)}$. Moreover, the kernel of $\varphi$ is $H$.
\item \cite[Proposition 4, \S 5.8]{tits:km} We have the following Bruhat-type decompositions \be{bru} G = B N B = B^- \, N \, B^- \ee where both decompositions are disjoint, i.e. if $n, n' \in N$, then $B n B \cap B n' B$ is either empty or the two double cosets are equal (and similarly for the $B^-$-double cosets). 
\item \cite[Corollary 2, \S 5.13]{tits:km} We have $G = B N B^- = B^- N B$ where both of these decompositions are also disjoint (in a similar sense to the above).   \end{enumerate} 
\end{nthm} 

Returning to the setting of Tits functors, note that the subgroups as defined in \S \ref{tits:sbgp} satisfy the axioms \rd{1}-\rd{5} so the results of Theorem \ref{bruhat-gen} hold.

\subsection{On the group $N$, its ``integral version'' $N_{\zee},$ and an explicit Bruhat decomposition} \label{sec:N}

\newcommand{\rnn}[1]{{\textbf{N {#1}}}}
\newcommand{\rnnz}[1]{\mathbf{N}_{\zee} {\textbf{#1}}}
In this section we further assume that our root datum $\mf{D}$ is of simply-connected type, and construct the group $G:= \bG(F)$ as in the previous section.

\tpoint{Further notation in the simply connected case} \label{more-sc-not} We adopt here some further notation in the simply connected case. Let $e$ denote the dimension of the root datum (i.e. the rank of $Y$). Recall from (\ref{Y:sc}) that we can identify $Y= \Qv \oplus Y_0.$ We extend the basis $\Pv$ of $\Qv$ to a basis $\Pv_e$ of $Y$, and denote a set indexing this basis by $\Is_e$ (which we assume contains $\Is$), so that we  write $\av_i \,(i \in \Is_e)$ for a basis of $Y$ where $\av_i\, (i \in \Is)$ is the basis of coroots that are given as part of the root datum. For $a \in \Pi$ and $s \in F^*$ we have identified the element $h_a(s)= s^{\av}$ (see \S \ref{s:stein-functor}), and we continue to do this for $\bv \in \Pv_e$, i.e. we shall write $h_b(s)$ for $s^{\bv}$ (with an abuse of notation as we have not really defined the symbol ``$b$''). The reason we work in the simply connected case is so that we can use the following.

\begin{nlem} \label{H:sc} Every element in $H$ can be written uniquely in the form \be{h:prod} h = \prod_{\bv \in \Pv_e } h_{b}(s_b) \text{ where } s_b \in F^*. \ee \end{nlem} In particular for $a,\ b \in \Pi,$ $a\neq b$, if we define the subgroups for  $H_a:= \{ h_a(s) \mid s \in F^* \}$ and $H_b:= \{ h_b(s) \mid s \in F^* \}$ then $H_a \cap H_b =1$. In the non simply connected case, this could fail.

\tpoint{A presentation for the group $N$} Recall the morphism $\phi: N \rr W$ from Theorem \ref{bruhat-gen}(1). 

\begin{nprop} (cf. \cite[Corollary 2.1]{kp}) \label{N:pres} The group $N$ has the following presentation: it has generators $\dw_a (a \in \Pi)$ and $h_{a}(s) \, (a \in \Pv_e, s \in F^*)$ satisfying the following relations:  \begin{enumerate}[\bf{N} 1] \label{N:rels}

\item $\dw_a^2 = h_{a}(-1)$ for each $a \in \Pi.$
\item  $\dw_a \, (a \in \Pi) $ satisfy the Braid relations
\item  $\dw_a h_{b}(t) \dw_a^{-1} = h_{b}(t) h_{a}(t^{-\la a, \bv \ra })$ for $a \in \Pi,$ $b \in \Pv_e,$ $t \in F^*$
\item $h_{b}(t) h_{b}(t') = h_{b}(t t')$ for $b \in \Pv_e,$ $t,\ t' \in F^*$
\item $h_{a}(t) h_{b}(t') = h_{b}(t') h_{a}(t)$ for $a,\ b \in \Pv_e,$ $t,\ t' \in F^*$
\end{enumerate}

\end{nprop}

\begin{nrem} Note that in the above result we could have actually chosen $\dw_a:= w_{a}(s)$ for any $s \in F^*$.  Also if $w \in W$ has a reduced decomposition $w = w_{b_1} \cdots w_{b_k}$ with each $b_i \in \Pi$ we shall write \be{dw} \dw:= \dw_{b_1} \cdots \dw_{b_k}, \ee which, according to \rnn{2} is independent of the chosen reduced decomposition.  
\end{nrem} 

The proof is standard, but we reproduce it here as we shall need variants of it in the sequel. 
\tpoint{Proof of Proposition \ref{N:pres}: Step 1} By definition, $N$ is generated by the indicated elements, and we need to show that \rnn{1}-\rnn{5} hold in $N$. Condition \rnn{1} and \rnn{3}-\rnn{5} follow easily. It remains to verify condition \rnn{2}, i.e. we need to show

\begin{nlem}\cite[Prop. 3]{tits:cts} \label{n:bd}The elements $\dw_a\, (a \in \Pi)$ satisfy the braid relations \eqref{bd:rel}. \end{nlem} 

\begin{proof}[Proof of Lemma \ref{n:bd}] For each $a \in R_{re}$, let $V_a \subset G_a$ be the subgroup generated by $U_a,$ $U_{-a},$ and the subgroup $H_a:= \{ h_a(s) \mid s \in F^* \}.$ By definition \eqref{ws} the elements $\dw_a:= w_a(-1)$ lie in $V_a.$ Moreover we find that if $a, b \in \Pi$ we have $V_a \cap V_b =1$. Indeed, From \rd{3} and the disjointness of the Bruhat decomposition, we have $V_a \cap V_b = H_a \cap H_b,$ but the latter intersection is trivial from the remarks after Lemma \ref{H:sc}.  Further, the $\varphi: N \rr W$ introduced above clearly also satisfies the condition that $V_{\varphi(n)(c)} = n V_c n^{-1}$ for $n \in N$ and $c \in \Pi.$ 

After these preliminaries, the proof of \cite[Prop. 3]{tits:cts} carries over. Let $q, q'$ be respectively the left and right hand side of \eqref{bd:rel}, where instead of $s_i$ we have $\dw_{a_i}.$ We wish to show that $q=q'$. Let $d$ be equal to either $j$ or $i$ depending on whether $h_{ij}$ is even or odd, so that one has $q' = \dw_j q \dw_{d}^{-1}.$ Since the Braid relations hold in $W(\As)$, we have $\varphi(q') = \varphi(q).$ Thus  $\varphi(q) s_{d} \varphi(q) = s_j,$ and so $\varphi(q) (a_d) = \pm a_j.$ Further,  \be{q':q-2} q' q^{-1} = \dw_j \, q \, \dw_d^{-1} q^{-1} \in \dw_j q \, V_{a_d} q^{-1} = \dw_j V_{\varphi(q)(a_d)} = \dw_j V_{a_j} = V_{a_j}. \ee Reversing the roles of $i$ and $j$ we obtain similarly that $q q'^{-1} \in V_{a_i}$ and so $q q'^{-1}=1$. \end{proof}

\newcommand{\sH}{\mathsf{H}}
\newcommand{\sN}{\mathsf{N}}
\newcommand{\sn}{\mathsf{n}}

\tpoint{Proof of Proposition \ref{N:pres}, Step 2}  \label{N-pres-2} Now, let $\sN$ be the group generated by elements $\w_a (a \in \Pi)$ and $\sh_{a}(s)$ for $a \in \Pv_e, s \in F^*$ satisfying \rnn{1}-\rnn{5}. Then there exists a natural surjective map $\psi: \sN \rr N$ sending $\sw_a \mapsto \dw_a (a \in \Pi)$ and $\sh_{a}(s) \mapsto h_{a}(s) (a \in \Pv_e, s \in F^*)$, which we need to verify is injective. Let $\sH$ be the subgroup of $\sN$ generated by  $\sh_{a}(s) (a \in \Pv_e, s \in F^*).$   By Lemma \ref{H:sc} have that $\psi|_\sH: \sH \rr H$ is an isomorphism.  Condition \rnn{3} shows that $\sH \subset \sN$ is normal, and so using \rnn{1} and the indicated generators for $\sN$, we may write every element of $\sN$ as $\sn= \sh \sw$ where $\sh \in \sH$ and $\sw$ is a product of the $\sw_{a} \, (a \in \Pi).$ If $\psi(\sn) = \psi(\sh \, \sw)=1$, then $\psi(\sw) \in H$. If we can show that $\psi^{-1}(H) \subset \sH$ then the result will follow: indeed, then $\sn = \sh \sw \in \sH,$ but $\psi|_{\sH}$ was injective.  Let us now argue that $\psi^{-1}(H) \subset \sH$. Consider the composition $f: \sN \stackrel{\psi}{\rr} N \stackrel{\phi}{\rr} W$ which sends $\sw_a \mapsto s_a$ for $a \in \Pi.$ Let $\sw= \sw_{i_1} \cdots \sw_{i_k} \in \sN$ (here we write $\sw_i:= \sw_{a_i}, s_i:= s_{a_i}$ as usual) be such that $\psi(\sw) \in H$ and so \be{f:w} f(\sw)= s_{i_1} \cdots s_{i_k} = 1. \ee  From Lemma \ref{Weyl-moves}, we can transform $w= s_{i_1} \cdots s_{i_k}$ to $1$ using the moves $E_1$ and $E_2$: diagrammatically \be{w->1} w \stackrel{E_{k_1}}{\rr} \, w' \,  \stackrel{E_{k_1}}{\rr} \, w''  \, \cdots \stackrel{E_{k_r}}{\rr} \, 1, \ee where each $k_i \in \{ 1, 2 \}$ and each of the $w, w', w'', \ldots $ is some product of the $s_i$. Now consider this same sequence of moves in the group $\sN$:  the move $E_2$ (braid relations for the $\sw_i$) stays valid in $\sN$ but the move $E_1$ should be changed to  \be{e:1p}  E_{1'}: \,  \text{ replace an occurence of }  \sw_i \sw_i  \text{ with } \sh_{a_i}(-1). \ee Apply the following procedure:  start from $\sw$ and apply the same sequence of moves as in (\ref{w->1}) to $\sw,$ but replacing $E_1$ with $E_{1'},$ and after each occurrence of $E_{1'}$ move the $\sh_{a}(-1)$ which is produced to the extreme left of the word (this can be done as $\sH$ is normal in $\sN$). The end result of this procedure will now be an element in $\sH,$ and in fact will actually lie in the normal subgroup generated by $\sh_a(-1), (a \in \Pi).$ Thus $\sw$ is equal in $\sN$ to an element in $\sH$ which is indeed what we wanted to show.

\tpoint{The integral Weyl group $N_{\zee}$} \label{s:nz} Let us define $N_{\zee} \subset N$ as the subgroup generated by $\dw_a, (a \in \Pi).$\footnote{In the usual theory of algebraic groups, $N_{\zee}$ can in fact be identified with the $\zee$-points of the group scheme defining $N.$ Here we are just adopting the suggestive notation.}  Set \be{Hto} H_{\zee}:= N_{\zee} \cap H, \ee and note that it is normal in $N_{\zee}$ by \rnn{4}. 
Using the same argument as at the end of \S \ref{N-pres-2}, we may show

\begin{nlem} \label{Hto:lem} The group $H_{\zee}$ is a finite abelian $2$-group generated by elements $h_{a}(-1) = \dw_a^2 \,\ (a \in \Pi).$ \end{nlem}

Using this Lemma and the same argument as in the proof of Proposition \ref{N:pres}, we obtain.

\newcommand{\rnz}[1]{\mathbf{N_{\zee}} \, {#1}}

\begin{ncor} \label{Nzee:pres}  The group $N_{\zee}$ is isomorphic to the abstract group which has generators $r_a\ (a \in \Pi)$ subject to the following relations. Set $t_a:= r_a^2 \,\ (a \in \Pi).$ 
\begin{enumerate}[$\mathbf{N_{\zee}} \, 1$.] 
\item The elements $(r_a \mid a \in \Pi)$ satisfy the braid relations (\ref{bd:rel}). 
\item For $a, b \in \Pi$ we have \be{t:r} t_a \, r_b \,   t_a^{-1} = \begin{cases} r_b & \text{ if } \la b, \av \ra \text{ even; } \\  r_b^{-1} & \text{ if } \la b, \av \ra \text{ odd. } \end{cases}  \ee 
\item $t_a^2=1.$ 
\end{enumerate}
\end{ncor}

\begin{nrem} Using $\rnz{2},\ \rnz{3}$, we find $N_{\zee}[2]:= \la t_a \mid a \in \Pi \ra \subset N_{\zee}$ is an abelian and isomorphic to $H_{\zee}.$ \end{nrem}

\tpoint{Refined Bruhat decomposition}  Using the presentation of $N$ and $N_{\zee}$ given, we can refine the Bruhat decompositions from Theorem \ref{bruhat-gen} as follows. For each $w \in W$ let $\dw \in N_{\zee}$ as in (\ref{dw}).  Then we have the following (disjoint) decompositions \be{G:w} G = \bigsqcup_{w \in W} B \, \dw \, B = \bigsqcup_{w \in W} B^- \, \dw \, B^- = \bigsqcup_{w \in W} B \, \dw \, B^-  = \bigsqcup_{w \in W} B^- \, \dw \, B. \ee For each $w \in W$ we have subsets (the first was already defined in \eqref{inv-set}) \be{R:w} \begin{array}{lcr} R(w) := \{ a \in R_{re,+} \mid w^{-1} a < 0 \}  & \text{ and } & R^w := \{ a \in R_{re, +} \mid w^{-1} a > 0 \} \end{array}. \ee  Let  us define correspondingly the groups $U_w$ and $U^w$ generated by the root groups $U_a$ with $a \in R(w^{-1})$ and $R^{w^{-1}}$ respectively. If $w=s_a, a \in \Pi$, then we just write $U_a:= U_{w_a}$ and $U^a:= U^{w_a}$ and note further that we have a semi-direct product decomposition \be{Ua:dec} U = U_a \ltimes U^a. \ee In general for $w \in W$ we still have \be{U:dec} U = U_w U^w = U^w U_w, \ee though $U^w$ is not normal. For a proof of (\ref{Ua:dec}) we refer to \cite[Lemma 6.3]{car:gar} which technically works in a homomorphic image of the Tits group (constructed from representation theory), but the result can be lifted back to $G$. For (\ref{U:dec}), we can refer for example to \cite[Corollary 6.5]{gar:alg}, again in a slightly different context, but a similar argument applies.  Note further that the set $R(w^{-1})$ is nilpotent (cf. \S \ref{prenilp}) so we can also identify $U_w:= \bU_{R(w^{-1})}(F)$ and use (\ref{Psi:sch}) to obtain an explicit set of coordinates on $U_w.$ Using the decomposition $B = U \rtimes H$, the fact that $N_{\zee}$ normalizes $H,$ and (\ref{U:dec}) we deduce from (\ref{G:w}) \be{G:w:ref} G = \bigsqcup_{w \in W} B \dw U_w \ee where we also note here that $\dw U_w \dw^{-1} \subset U^-.$

\tpoint{The map $\nu$ and some explicit results} \label{sec:bru} For any $n \in N$ we note that \be{U:n} U \, n \, U \cap N = \{ n \}. \ee  Indeed, as $H$ normalizes $U$ and since (\ref{U:dec}) implies that $U \, \dw \, U = U \, \dw \, U_w$ for $\dw \in N_{\zee}$ we conclude that $U \, n \, U \, n^{-1}  \subset  U \,  U^-.$ But from the disjointness of the Birkhoff decomposition and the fact that $N= HN_{\zee}$ we have $U \, U^- \cap N = \{ 1 \}.$ The claim (\ref{U:n}) follows and hence we have a well-defined map \be{nu} \nu: G \rr N \ee which to $g \in G$ having Bruhat decomposition $g = u n u'$ assigns the element $n \in N.$ For $g \in G,$  \be{nu:HU} \nu(u g u') = \nu(g) \text{ and } \nu(h g h') = h \nu(g) h' \text{ where } u,\ u' \in U,\ h,\ h' \in H. \ee Moreover, for $g, g' \in G,$ using basic properties of the Bruhat decomposition (\cite[Ch. IV, \S 2.5, Prop. 2]{bour}) we can show \be{nu:bru} \nu(g \, g') = \nu(g) \nu(g') \text{ if and only if } \ell\left( \phi(\nu(g))  \, \phi(\nu(g')) \right) =  \ell( \phi(\nu(g)) ) + \ell( \phi(\nu(g') )), \ee where we have used the map $\phi$ of Theorem \label{bruhat-gen}(1). Finally we shall need the following explicit rank one result.

\begin{nprop} \label{bru:two} Let $g \in G$, $a \in \Pi$. 
\begin{enumerate} 
\item We have $\nu(\dw_a \, g)$ is either equal to $\dw_a \, \nu(g)$ or there exists $s \in F^*$ such that $\nu(\dw_a \, g)= h_a(s^{-1}) \nu(g).$ 
\item We have  $\nu( g \, \dw_a^{-1} )$ is either equal to $\nu(g) \, \dw_a^{-1}$ or there exists $s \in F^*$ such that $\nu(g \, \dw_a^{-1})= \nu(g) h_a(s)$ 
\end{enumerate} \end{nprop} 

We suppress the full proof (cf. \cite{mat}) but record here the main rank one computations involved \be{rk1} x_{-a}(s) &=& x_a (s^{-1}) \, h_a(-s^{-1}) \, \dw_a \, x_a(s^{-1}) \text{ for } a \in \Pi, s \in F^* \\ x_a(s) &=& x_{-a}(s^{-1}) h_a(s) w_a x_{-a}(s^{-1}) .\ee Indeed, using this we conclude that if $a \in \Pi$, $s \in F^*,$ and $w^{-1} a <0,$ then  
\be{rk1:l} \nu(\dw_a \, x_a(s) \dw) = h_a(-s^{-1} ) \dw_a \, \dw_a \dw = h_a(-s^{-1} ); \, \dw  \ee and if $w a < 0$ then \be{rk1:r} \nu(\dw \, x_a(s) \dw_a^{-1}) = \dw \, h_a( s). \ee In both instances we have also used the identity (cf. \ref{s:stein-functor}) that $\dw_a x_a(s) \dw_a^{-1} = x_a(-s)$.


\section{Covers of Kac-Moody Groups} \label{sec-mat}

Let $A$ be an abelian group, $F$ an arbitrary field, and $(\cdot, \cdot): F^* \times F^* \rr A$ be a bilinear Steinberg symbol (cf. \S \ref{stein-symb}). Fix a root datum $(I, \cdot, \mf{D}),$ written as $\mf{D}=(\Lv, \{ \av_i \}, \Lambda, \{ a_i \})$, and assumed of \emph{simply connected type}. Construct the group $G= \bG(F)$ as in the previous section, and keep the notation on simply connected root datums from \S \ref{s:rd} and \S \ref{more-sc-not}. Furthermore, fix a $W$-invariant quadratic form $\Qs: Y \rr \zee$ (cf. \S \ref{s:qform}).

The arguments in this section follow very closely those of Matsumoto \cite[\S 6 - \S 7]{mat}. We could perhaps have referred to this source systematically throughout and thereby, but we choose to provide here a more or less complete version of Matsumoto's argument as our setup is slightly different (e.g. there is no $Q$ in Matsumoto). 

\subsection{The cover of the torus $\tH$}

\spoint Let $\tH$ be the group generated by $A$ and the symbols $\h_{a}(s)$ with $s \in F^*, \, a \in \Pv_e$ subject to the following conditions: $A$ is an abelian subgroup and \begin{enumerate}[\bf{H}1]
\item $\h_{a}(s) \h_{a}(t) \h_{a}(st)^{-1} = (s, t)^{\Qs(\av)}$ for $a \in \Pv_e$ and $s, t \in F^*$
\item $[ \h_{a}(s), \h_{b}(t) ] = (s, t)^{\B(\av, \bv)}$ for $a,b \in \Pv_e$
\end{enumerate}

\begin{nrem}  The relations $\rh{1}$ and $\rh{2}$ are compatible in the sense that if $a \in \Pv_e$ we may compute using $\rh{1}:$  \be{compat:h} [ \h_{a}(s), \h_{a}(t) ] &=& \h_a(s) \h_a(t) \left( \h_a(t) \h_a(s) \right)^{-1} = \h_a(st)(s, t)^{\Qs(\av)} \h_a(ts)^{-1} (t, s)^{ - \Qs(\av)} \\ &=& (s, t)^{\Qs(\av)} (t, s)^{ - \Qs(\av)} = (s, t)^{ 2 \Qs(\av)} ,\ee where we have used the skew-symmetry of the symbol in the last equality. On the other hand, from $\rh{2}$, the above must equal  $(s, t)^{\B(\av, \av)}$. By (\ref{bq:db}), we have $\B(\av, \av) = \Qs(2\av)- 2 \Qs(\av) = 2 \Qs(\av)$. We also note here that for any $b \in \Pv_e$ and $s \in F^*$, we obtain from $\rh{1}$, \be{ha:inv} \h_b(s)^{-1} = \h_b(s^{-1}) (s, s)^{-Q(\bv)} = \h_b(s^{-1}) (s, s)^{Q(\bv)} \ee
\end{nrem}

\spoint The map $\varphi: \tH \rr H$ which sends every element of $A$ to the identity and $\h_a(s) \mapsto h_a(s)$ for each $a \in \Pi,\ s \in F^*$ is clearly a homomorphism. The simply connected assumption gives the following.
\begin{nlem} \label{H:straight} Fix an order on the set $\Pv_e.$ Then every element in $\tH$ can be written uniquely in the form \be{h:prod} h = \zeta \, \prod_{\bv \in \Pv_e } \h_{s_b}(s_b) \text{ where } \zeta \in A,\ s_b \in F^*. \ee \end{nlem} 
\begin{proof} The existence of such a factorization for any $h \in \tH$ is obvious from the definitions. If we have an equality $1= \zeta \prod_{b \in \Pv_e} \h_{b}(s_b)$ then in $H$ we must have that $1 = \prod_{b \in \Pv_e} h_{b}(s_b)$. However, from Lemma \ref{H:sc}, this implies each $s_b=1$ and thus $\zeta=1$ as well. The uniqueness follows. \end{proof}

\spoint From Lemma \ref{H:straight} it follows that $\ker(\varphi)= A$, and so we have a central extension of groups \be{H:ext} 0 \rr A \rr \wt{H} \stackrel{\varphi}{\rr} H \rr 1. \ee We record here the following generalizations of the formulas $\rh{1},\ \rh{2}$. To state them, we first note that there exist integers $\Qs_{ij}$ and $\B_{ij}$ for $i, j \in \Is_e$ defined by the formulas: for $c_i, d_i \in \zee$  \be{Q:fla} \Qs\left( \sum_{i \in \Is_e} {c_i} \av_i\right) &=& \sum_{i, j} c_i c_j \Qs_{ij} \\ 
\Bs\left(\sum_{i \in \Is_e} {c_i} \av_i, \sum_{i \in \Is_e} {d_i} \av_i\right)
&=& \sum_{i, j} c_i d_j \B_{ij}. \ee

\begin{nlem} Fix an order on $\Is_e$ and let $s_i, t_i \in F^* (i \in \Is_e).$ Define (with respect to this order) $\h:= \prod_{i \in \Is_e} \h_{a_i}(s_i)$ , $\h':= \prod_{i \in \Is_e} \h_{a_i}(t_i)$, and $\h'':= \prod_{i \in \Is_e} \h_{a_i}(s_it_i).$ Then \be{cocy:comm} \begin{array}{lcr} \h \, \h' \, (\h'')^{-1} = \prod_{i, j} (s_i, t_j)^{\Qs_{ij}} & \text{ and } &  [ \h, \h'] = \prod_{i, j} (s_i, t_j)^{\Bs_{ij}} \end{array} \ee \end{nlem}

\spoint \label{s:tHz}  Let $\tH_{\zee} \subset \tH$ be the subgroup generated by the elements $\h_{\a}(-1)$ for $a \in \Pi$ (not $\Pv_e$).

\begin{nlem} \label{H2:rels} For $a, b \in \Pi$, the following relations hold in $\tH_{\zee} \subset \tH$: , \be{th:2}    \h_b(-1)^{-1} \h_a(-1) \h_b(-1)   =  \begin{cases} \h_a(-1) & \text{ if } \la b, \av \ra \text{ is even }  \\  \h_a(-1) \, \h_b(-1)^2 & \text{ if } \la b, \av \ra \text{ is odd } \end{cases} \ee \end{nlem} 
\begin{proof} For the first relation, we note that \be{h2:1} \left[ \h_a(-1), \h_b(-1) \right] = (-1, -1)^{\B(\av, \bv)} \ee but if $\la b, \av \ra$ is even then so is $\B(\av, \bv) = \Qs(\bv) \la b, \av \ra$ and hence the claim since $(-1, -1)= \pm1$ from Lemma \ref{stein-symb}. On the other hand, if $\la b, \av \ra$ is odd, then $\B(\av, \bv) + \Qs(\bv) = \Qs(\bv)(\la b, \av \ra +1)$ which is therefore even. By $\rh{1}$, we have $\h_b(-1) \h_b(-1) = (-1, -1)^{\Qs(\bv)}$ and by $\rh{2}$ we obtain  \be{h2:2}   \h_b(-1) \h_a(-1) \h_b(-1)^2 =  \h_a(-1) \h_b(-1) (-1, -1)^{\B(\av, \bv)} (-1, -1)^{\Qs(\bv)} = \h_a(-1) \h_b(-1), \ee proving the second case. 
 \end{proof}

\subsection{A family of automorphisms of $\tH$} \label{sec:H-aut}

\renewcommand{\a}{\gamma}
\newcommand{\si}{\s^{-1}}

\spoint For each $a \in \Pi$, we can define a map (the inverse will be justified in the Lemma below) $\si_{a}: \wt{H} \rr \wt{H}$ \be{th-a} \si_a(x) &=& x \text{ for } x \in A \\ \si_a(\h_{b}(s)) &=& \h_b(s) \h_{a}(s^{ -\la a, \bv \ra}) \text{ for } b \in \Pv_e,\,\  t \in F^*. \ee This map descends, under $\varphi: \tH \rr H,$ to the simple reflection $s_a^{-1} \in W$ acting on $H$ (cf. \kmr{2} in \S \ref{s:stein-functor}), and our goal in this part is to generalize certain properties of $s_a$ to the elements $\s_a$. 

\begin{nlem} \cite[Lemma 6.5]{mat} \label{s-hom} Let $a \in \Pi$ and $ s \in F^*$. Then,
\begin{enumerate} \item  We have $\si_a (\h_a(s)) = \h_a(s^{-1})$ 
\item The map $\si_a: \tH \rr \tH$ is an isomorphism of groups. Explicitly, one has for $b \in \Pv_e,$  \be{thi-inv} \s_a(\h_b(s)) = \h_a(s^{- \la a, \bv \ra}) \h_b(s) = \h_b(s) \h_a(s^{\la a, \bv \ra})^{-1}\ee \end{enumerate} 
\end{nlem} 

\begin{proof} 
To prove (1), we may compute from the definitions and the fact that $\la a, \av \ra =2$ that  \be{ref:thi} \si_a(\h_a(s)) &=& \h_a(s) \h_a(s^{ - \la a, \av \ra}) = \h_a(s) \h_a(s^{-2}) \\ &=& \h_a(s^{-1}) (s, s^{-2})^{\Q(\av)} =  \h_a(s^{-1}) (s, s)^{-2 \Q(\av)}= h_a(s^{-1}), \ee where we have used Lemma \ref{stein} (ii) and (v) for the second-to-last and last equality respectively. 

Turning to (2), we first show $\si_{\a}, \a \in \Pi$ is a homomorphism of groups: i.e., for $ a, b \in \Pv_e$ and $s, t \in F^*,$ 

\begin{enumerate}[(i)]
\item $\si_{\a} (\h_{b}(s)) \si_{\a}(\h_{b}(t) ) \si_{\a}( \h_{b}(st) )^{-1} = (s, t)^{\Qs(\bv)}$ 
\item $\left[ \si_{\a}(\h_{a}(s)), \si_{\a}(\h_{b}(t)) \right] = (s, t)^{\B(\av, \bv)} $ \end{enumerate} 

As for (i), note that from the definitions we obtain \be{thi:1} \si_{\a} (\h_{b}(s)) \si_{\a}(\h_{b}(t) ) =  \h_b(s) \, \underbrace{\h_{\a}(s^{ -\la \a, \bv \ra}) \h_b(t) } \, \h_{\a}(t^{ -\la \a, \bv \ra}). \ee Using the commutator relation $\rh{2}$ on the underbraced terms, we obtain \be{thi:1b} \underbrace{\h_b(s)  \h_b(t) } \, \underbrace{ \h_{\a}(s^{ -\la \a, \bv \ra})  \h_{\a}(t^{ -\la \a, \bv \ra}) } (s^{ -\la \a, \bv \ra}, t)^{\B(\gamv, \bv)}. \ee Now using relation $\rh{1}$ on each pair of underbraced terms, we obtain \be{thi:1c} 
\underbrace{\h_b(st) (s, t)^{\Qs(\bv)} \h_\a( (st)^{ - \la \a, \bv \ra}) }(s^{ -\la \a, \bv \ra}, t^{ -\la \a, \bv \ra})^{\Qs(\gamv)} (s^{ -\la \a, \bv \ra}, t)^{\B(\gamv, \bv)}.\ee By definition, the underbraced term is $(s, t)^{\Qs(\bv)} \si_{\a}(\h_b(st)),$ so using (\ref{bq:db}) we find that (\ref{thi:1c}) is equal to \be{thi:1c-a} \si_{\a}(\h_b(st))  (s, t)^{\Qs(\bv)}  (s^{ -\la \a, \bv \ra}, t)^{-\Bs(\gamv, \bv) } (s^{ -\la \a, \bv \ra}, t)^{\B(\gamv, \bv)} =  \si_{\a}(\h_b(st))  (s, t)^{\Qs(\bv)}. \ee 
Now we turn to (ii). We begin by writing \be{thi:2} \si_{\a}(\h_{a}(s))\si_{\a}(\h_{b}(t)) = \h_a(s) \h_{\a}(s^{ -\la \a, \av \ra}) \h_b(t) \h_{\a}(t^{ -\la \a, \bv \ra}). \ee Applying $\rh{2}$ twice to the right hand side we can move $\h_b(t)$ to the front of the expression and obtain \be{thi:2b}
\h_b(t) \h_a(s) \, \underbrace{ \h_{\a}(s^{ -\la \a, \av \ra})   \h_{\a}(t^{ -\la \a, \bv \ra}) } \, (s^{ -\la \a, \av \ra}, t)^{\B(\gamv, \bv)} (s, t) ^{ \B(\av, \bv)}. \ee Using $\rh{2}$ on the underbraced terms and noting that $\B(\gamv, \gamv) = 2 \Qs(\gamv)$, we obtain  \be{thi:2c} \h_b(t) \, \underbrace{ \h_a(s)  \h_{\a}(t^{ -\la \a, \bv \ra})  } \, \h_{\a}(s^{ -\la \a, \av \ra})  (s^{ -\la \a, \av \ra}, t)^{\B(\gamv, \bv)} (s, t) ^{ \B(\av, \bv)} ( s^{ -\la \a, \av \ra}, t^{ -\la \a, \bv \ra} )^{2 \Qs(\gamv)}. \ee Applying $\rh{2}$ to the underbraced term transforms the previous line to \be{thi:2c} \si_{\a}(\h_b(t)) \si_{\a}(\h_a(s))  (s,  t^{ - \la \a, \bv \ra})^{\B(\av, \gamv)} (s^{ -\la \a, \av \ra}, t)^{\B(\gamv, \bv)} (s, t) ^{ \B(\av, \bv)} ( s^{ -\la \a, \av \ra}, t^{ -\la \a, \bv \ra} )^{2 \Qs(\gamv)}. \ee Noting again  $\B(\gamv, \bv) = \Q(\gamv) \la \gamma, \bv \ra$ we find that \be{thi:2d} 
 ( s^{ -\la \a, \av \ra}, t^{ -\la \a, \bv \ra} )^{2 \Qs(\gamv)} &=&  ( s^{ -\la \a, \av \ra}, t^{ -\la \a, \bv \ra} )^{ \Qs(\gamv)}  ( s^{ -\la \a, \av \ra}, t^{ -\la \a, \bv \ra} )^{\Qs(\gamv)} \\
 &=&  ( s, t^{ -\la \a, \bv \ra} )^{-\B(\av, \gamv)}  ( s^{ -\la \a, \av \ra}, t )^{-\B(\gamv, \bv)}, \ee which, together with (\ref{thi:2c}) proves (ii).
 
Let us now prove \eqref{thi-inv} by checking $\si_a \left(  \s_a (h_b(s) \right) = \h_b(s).$ We compute as follows, \be{sa:inv-1} \si_a \left( \h_b(s) \h_a(s^{\la a, \bv \ra})^{-1} \right) &=& \h_b(s) \h_a(s^{ -\la a, \bv \ra}) \si_a ( \h_a(s^{ \la a, \bv \ra})^{-1} ) \\
&=&  \h_b(s) \h_a(s^{ -\la a, \bv \ra}) \si_a ( \h_a(s^{ \la a, \bv \ra}))^{-1} \\
&=&  \h_b(s) \h_a(s^{ -\la a, \bv \ra})  \h_a(s^{ -\la a, \bv \ra})^{-1} \\
&=& \h_b(s) \h_a(s^{ -\la a, \bv \ra})  \h_a(s^{ \la a, \bv \ra}) \left( s^{ \la a, \bv \ra}, s^{ \la a, \bv \ra} \right)^{- \Q(\av)} \\
&=&  \h_b(s) \left( s^{ -\la a, \bv \ra},   s^{ \la a, \bv \ra} \right)^{\Q(\av)} \left( s^{ \la a, \bv \ra}, s^{ \la a, \bv \ra} \right)^{- \Q(\av)} \\
&=& \h_b(s) \ee   where we have used that fact that $(s, s)^{2n}=1$ for any integer $n$ in the last line. \end{proof}

\spoint  Writing $\Ad_g$ for the map which sends $h \mapsto g h g^{-1}$, we have the following:

\begin{nlem} \cite[Lemma 6.5]{mat} \label{th:sq} For each $a \in \Pi$, we have $(\s_{a}^{-1})^{2} = \s_{a}^2 = \Ad_{\h_{a}(-1)}$  as automorphisms of $\tH.$ \end{nlem} 
\begin{proof} Using Lemma \ref{s-hom}, (1) we compute that \be{th:sq-2} \s_{a}^{-2}(\h_b(s)) &=& \si_{a} (\h_b(s) \, \h_{a}(s^{ - \la a, \bv \ra}) ) \\
&=& \h_b(s) \,  \h_{a}(s^{ -\la a, \bv \ra}) \,   \h_{a}(s^{  \la a, \bv \ra}) \\ 
&=& \h_b(s) (s^{ -\la a, \bv \ra}, s^{ \la a, \bv \ra})^{\Qs(\av)}. \ee From Lemma \ref{stein} (vi), this last line is equal to \be{th:sq-3}  \h_b(s) (s^{ -\la a, \bv \ra}, -1)^{\Qs(\av)} &=&   \h_b(s) (s, -1)^{-\Bs(\av, \bv)}. \ee Now using $\rh{2}$ we have $(s, -1)^{-\Bs(\av, \bv)} = (-1, s)^{\Bs(\av, \bv)}=[\h_{a}(-1), \h_b(s)]$ and so we finally obtain \be{s:sq-1} (s, -1)^{-\Bs(\av, \bv)} \h_b(s) = \h_{a}(-1) \h_b(s) \h_{a}(-1)^{-1}. \ee 

To contend with the case of $(\s_{a}^{-1})^2= (\Ad_{\h_{a}(-1)})^{-1},$ we just note that $\Ad^2_{\h_{a}(-1)}$ acts as the identity. Indeed  \be{ad^2} \Ad_{\h_{a}(-1)} \Ad_{\h_{a}(-1)} ( \h_b(s) )  = \h_{a}(-1) \h_{a}(-1) \h_b(s) (\h_{a}(-1) \h_{a}(-1))^{-1} = \h_b(s), \ee  since $\h_{a}(-1) \h_{a}(-1) = (-1, -1)^{\Q(\av)}$ and again $(-1, -1)^{2 n}=1$ for any integer $n$.

\end{proof}

\begin{nrem} \label{s:pres:hz} Recall the subgroup $\tH_{\zee} \subset \tH$ introduced in \S \ref{s:tHz}. Then for each $a \in \Pi$, $\s_a$ preserves $\tH_{\zee}:$ \be{s:h2} \si_a (\h_b(-1) ) = \begin{cases} \h_b(-1) & \text{ if } \la a, \bv \ra \text{ even, } \\ \h_b(-1) \h_a(-1) & \text{ if } \la a, \bv \ra \text{ odd; } \end{cases} \ee  for $b \in \Pv_e$. Moreover, an easy computation gives \be{sa:2:hz} \s_a^{-2}( \h_b(-1)) = \h_b(-1) \left(-1, -1\right)^{\Bs(\av, \bv) \la a,  \bv \ra}, \ee from which it follows that $\s_a^{-4} = \s_a^4=1$ on $H_{\zee}$.

\end{nrem} 

\tpoint{Braid relations} The proof of the following result is a simple verification which we suppress here.

\begin{nlem} Let $k = - \la a, \bv \ra,$ $\ell = -\la b, \av \ra$ and write $m = - \la a, \gv \ra$ and $n = - \la b, \gv. \ra$ Then \be{B:bd-s} \Bs(\av, \bv) = - k \Qs(\av) = - \ell \Qs(\bv) \ee and for any $\gamma \in \Pv_e$ if $k \ell \neq 0$ we have 
 \be{bd:s-1} \si_a \left(\h_{\gamma}(t) \h_a(t^p) \h_b(t^q)\right) &=& \h_{\gamma}(t) \, \h_a(t^{m-p+ qk}) \, \h_b(t^q) \,  (t, t)^{ (mp/k + k q^2 + (p-m)q) \Bs(\av, \bv) } \\ \si_b \left(\h_{\gamma}(t) \h_a(t^p) \h_b(t^q)\right) &=& \h_{\gamma}(t) \, \h_a(t^p) \, \h_b(t^{p \ell + n -q}) \,  (t, t)^{ (n/\ell + p) q  \Bs(\av, \bv) }. \ee \end{nlem} 

 \begin{nrem} As $(t, t)= \pm1$ the exponent of the last term in the above expressions need only be regarded modulo 2, and we interpret the fractional powers in the natural way. \end{nrem} 

Using this Lemma, we can now prove

\begin{nprop} \label{bd:sa} The elements $\si_a \, (a \in \Pi)$ satisfy the braid relations. \end{nprop}
\begin{proof}  If $i, j \in \Is$ are such that $a=a_i, b=a_j $ we shall now write $h:= h_{ij}.$ Then we need to verify the following four cases (see Table \ref{h:tab} and (\ref{bd:rel})). Below we again assume $\gamma \in \Pv_e$ and  write $B:= \Bs(\av, \bv)$ 

 \begin{enumerate} \item If $h=2,$ then  $\la b, \av \ra = \la a, \bv \ra =0$ and we need to show that $\si_a \si_b = \si_b \si_a.$  We compute directly \be{h=2} \si_a \si_b h_{\gamma}(t) =  \si_b \si_a \h_{\gamma}(t) = \h_{\gamma}(t) \h_a(t^m) \h_b(t^n). \ee

\item If $h=3$, then  without loss of generality we may take $k=-\la a, \bv \ra = 1$ and $\ell= - \la b, \av \ra =1.$ We need to show $\si_a \si_b \si_a = \si_b \si_a \si_a.$ Applying the previous Lemma,  we have  \be{h=3}  \si_a \si_b \si_a  h_{\gamma}(t) &=& \h_{\gamma}(t) \h_a(t^{m+n}) \h_b(t^{m+n}) (t, t^B)^{m^2 + (m+n)^2}  \\  \si_b \si_a \si_b \h_{\gamma}(t) &=&\h_\gamma(t) \h_a(t^{m+n}) \h_b(t^{m+n}) (t, t^B)^{3n^2} .  \ee As $3n^2 \equiv m^2 + (m+n)^2 \mod 2$ this braid relation follows.

\item If $h=4$, then without loss of generality we may assume $k = - \la a, \bv \ra =1$ and $\ell= - \la b, \av \ra =2.$ We need to show $(\si_a \si_b)^2 = (\si_b \si_a)^2.$ Again, we apply the previous Lemma,   \be{h=3}  \si_a \si_b \si_a \si_b h_{\gamma}(t) &=& \h_{\gamma}(t) \h_a(t^{2m+n}) \h_b(t^{2(m+n)}) (t, t^B)^{5m^2 + 11mn +(17/2) n^2}  \\  \si_b \si_a \si_b \si_a \h_{\gamma}(t) &=&\h_\gamma(t) \h_a(t^{2m+n}) \h_b(t^{2(m+n)}) (t, t^B)^{9m^2+ 9 mn + (5/2) n^2}.  \ee Again a parity argument establishes the desired braid relation.

\item If $h=6$, then without loss of generality we may assume $k = - \la a, \bv \ra= 1$, and $\ell = - \la b, \av \ra =3$. In this case, we need to verify $(\si_a \si_b)^3 = (\si_b \si_a)^3.$  Using the Lemma, we compute \be{h=6}  \si_a \si_b \si_a \si_b \si_a \si_b h_{\gamma}(t) &=& \h_\gamma(t) \h_a(t^{4m+2n}) \h_b(t^{6m+4n}) (t, t^B)^{70m^2+ 108 mn + (136/3) n^2} \\
 \si_b \si_a \si_b \si_a \si_b \si_a\h_{\gamma}(t) &=& \h_{\gamma}(t) \h_a(t^{4m+2n}) \h_b(t^{6m+4n}) (t, t^B)^{94m^2 + 88mn +(64/3) n^2}, \ee and again achieve the desired conclusion using a parity argument. \end{enumerate} \end{proof}
 
\begin{nrem} As the braid relations have been satisfied for $\s_a$ and $\s_a^{-1}$ we now adopt the following unambiguous notation: if $w \in W$ has a reduced decomposition $w= s_{b_1} \cdots s_{b_d}$ with $b_i \in \Pi$ then we set \be{sw:inv}  \s_w := \s_{b_1} \cdots \s_{b_d}. \ee Note of course $\s_w^{-1} \neq \s_{w^{-1}}$. \end{nrem}

\tpoint{Another rank two result} In this paragraph, we recall the notation from \S \ref{rank2}. 

\begin{nlem} \label{w:h} Let $\Is = \{ a, b \}$ and $(I, \cdot)$ any Cartan datum. Let $w \in W(I, \cdot).$ 
\begin{enumerate} 
\item If $w \, a = a$ then $\s_w \, \h_a(s) = \h_a(s)$
\item If $w\, a = -a$ then $\s_w \, \h_a(s) = \h_a(s^{-1})$
\item If $w\, a = b$ then $ \s_w \, \h_a(s) = \h_b(s) $  \end{enumerate} \end{nlem} 
\begin{proof} In the notation of \S \ref{rank2} we write $\la a, \bv \ra = m $ and $\la b, \av \ra = n,$ so that $\Bs(\av, \bv) = m \Qs(\av) = n \Qs(\bv).$ It is then easy to see that 
\be{wa:fla1} \si_a \,\left( \h_a(s^k) \h_b(s^\ell)\right) &=& \h_a(s^{-\ell m - k} ) \h_b(s^{\ell}) (s, s)^{ (-\ell^2 m + k \ell) \Bs(\av, \bv) } \\ \label{wa:fla2} \si_b \, \left( \h_a(s^k) \h_b(s^\ell)\right) &=& \h_a(s^k) \h_b(s^{-kn-\ell}) (s, s)^{k \ell \Bs(\av, \bv)}. \ee Using these formulas and the classification given in \S \ref{s:swap}, \S \ref{s:fix}, the Lemma follows from a direct computation. For example, suppose we want to show (1): then note that (a) there are three cases for $w$ to consider; and (b) the result can equally be verified for $\s_w$ replaced by $\s_w^{-1}$: \begin{enumerate} \item In the $A_1 \times A_1$ case, we have $m=n=0$ and $w=s_b.$ Then $\si_b \h_a(s) = \h_a(s)$ immediately from the definition.
\item In the $B_2$ case, we may assume without loss of generality that $m=-1, n=-2$ and  $w = s_b s_a s_b$. Then \be{b2:1} \s_w^{-1} \h_a(s) = \h_a(s) (s, s)^{6 \B(\av, \bv) } \ee using (\ref{wa:fla1}-\ref{wa:fla2}). But as $(s, s)= \pm1$, the last term in this expression is equal to $1.$
\item In the $G_2$ case, we assume without loss of generality $m=-1, n=-3$ and $w= s_b(s_a s_b)^2.$ Then \be{g2:1} \s^{-1}_w \h_a(s) = \h_a(s) (s, s)^{18 \B(\av, \bv) } = \h_a(s) \ee using (\ref{wa:fla1}-\ref{wa:fla2}). 
\end{enumerate} Part (2) of the Lemma now follows from Lemma \ref{s-hom}, pt. (1), since if $w a = -a$ then $s_a w = a.$ We suppress the proof of (3) as it follows along the same lines as (1) using the classification from Lemma \ref{s:fix}.
\end{proof}

\subsection{The cover $\widetilde{N}$ of the group $N$} \label{s:ncover} 
\newcommand{\hnz}{\tN_{\zee}}
\newcommand{\hhz}{\wt{\mf{H}}}
\renewcommand{\r}{\mf{r}}
\renewcommand{\t}{\mf{t}}

\newcommand{\rtnz}[1]{\mathbf{\tN}_{\zee} \, {#1}}

\spoint Recall the group $N_{\zee}$ introduced in \S \ref{s:nz} and its presentation given in Proposition \ref{Nzee:pres}. If we drop the constraint $\rnz{3}$ we obtain the following group, $\tN_{\zee}.$ 

\begin{de} Let $\tN_{\zee}$ be the group with generators $\r_a, (a \in \Pi)$ and relations as below. Let $\t_a:= \r_a^2.$ \begin{enumerate}[$\mathbf{\tN_{\zee}} \,1$] \item The elements  $\r_a \ (a \in \Pi)$ satisfy the braid relations (\ref{bd:rel}). 
\item For $a,\ b \in \Pi$ we have \be{t:r} \t_a^{-1} \,  \, \r_b \,   \t_a = \begin{cases} \r_b & \text{ if } \la b, \av \ra \text{ even; } \\  \r_b^{-1} & \text{ if } \la b, \av \ra \text{ odd. } \end{cases} \ee 
 
 \end{enumerate} 
 
\end{de} 

\begin{nlem} \label{Nz:ker} The kernel of the natural map $\tN_{\zee} \rr N_{\zee}$ is the abelian group  $\la \t_a^2 \mid a \in \Pi \ra.$ \end{nlem} 
\begin{proof} It is immediate from $\rtnz{2}$ that $\t_a^2$ is central in $\tN_{\zee}.$ The result now follows from the explicit presentation for $N_{\zee}$ which we gave in Corollary \ref{Nzee:pres}. \end{proof} 

\spoint  Using the results from the previous section, we can now verify the following.

\begin{nprop} \label{i-hom} The map which sends $\r_a \mapsto \s_a$ for each $a \in \Pi$ induces a homomorphism  $i: \hnz \rightarrow \Aut(\tH).$     \end{nprop} 
\begin{proof} From Proposition \ref{bd:sa}, the $\s_a$ satisfy the braid relations, and so $\rtnz{1}$ is satisfied. Note that we can replace $\r_b$ with $\r_b^{-1}$ everywhere in \eqref{t:r} to obtain an equivalent relation-- this will be the actual relation we verify. As $i(\t_a) = \Ad_{\h_{a}(-1)}$ we need to verify the following two facts: \begin{enumerate} \item Assume $\la b , \av \ra$ is even, so that $\si_b(\h_a(-1)) = \h_a(-1)$. Now let $\gamma \in \Pv_e,$ and $s \in F^*$ and compute \be{i:1} \si_b i(\t_a) (\h_{\gamma}(s)) &=& \si_b \left( \h_a(-1) \h_{\gamma}(s) \h_a(-1)^{-1} \right) \\
&=& \h_a(-1) \h_{\gamma}(s) \h_b(s^{ - \la b, \gamv \ra}) \h_a(-1)^{-1}. \ee This is seen to be equal to $i(\t_a) \si_b (\h_{\gamma}(s)) = \Ad_{\h_a(-1)} ( \si_b(\h_{\gamma}(s))),$ verifying the first case of $\rtnz{2}$.

\item Assume $\la b , \av \ra$ is odd, so that $\si_b(\h_a(-1)) = \h_a(-1) \h_b(-1).$ The corresponding relation $\rtnz{2}$ can equivalently be written as $\r_b^{-1} \t_a = \t_a \t_b \r_b^{-1},$ and so let us verify this holds under the map $i$. The left hand side applied to $\h_{\gamma}(s)$ with $\gamma \in \Pv_e, s \in F^*$ is just \be{i:2} \si_b \left(  \h_a(-1) \h_{\gamma}(s) \h_a(-1)^{-1} \right) &=& \h_a(-1) \h_b(-1) \h_{\gamma}(s) \h_{b}(s^{ - \la b, \gamv \ra}) \si_b (\h_a(-1))^{-1} \\
&=& \h_a(-1) \h_b(-1) \h_{\gamma}(s) \h_{b}(s^{ - \la b, \gamv \ra}) \h_b(-1)^{-1} \h_a(-1)^{-1}. \ee This is immediately seen to equal $i(\t_a) i(\t_b) \si_b (\h_{\gamma}(s)).$
\end{enumerate} 
\end{proof}

\spoint Using Proposition \ref{i-hom} we have a morphism $\hnz \rr \Aut(\tH)$, and we then form the semi-direct product $ \hnz \rtimes \tH.$ 
Define the subgroup $\hnz[2] \subset \hnz$ generated by $\t_a:=\r_a^2,\ (a \in \Pi).$ Recall the subgroup $\tH_{\zee} \subset \tH$ in Remark \ref{s:pres:hz}. From Remark \ref{s:pres:hz}, $i(\hnz[2])$ preserves $\tH_{\zee}$, and in fact the corresponding map  $i: \Aut(\hnz[2]) \rr \Aut(\tH_{\zee})$ has image in the subgroup of internal automorphisms $\tH_{\zee} \subset \Aut(\tH_{\zee})$ . In this way, we obtain a homomorphism \be{i-hom-2} j: \hnz[2] \rr \tH_{\zee}, \, \t_a \mapsto \h_a(-1) \, (a \in \Pi), \ee and using it we define the set \be{J} J:= \{ (\t^{-1}, j(\t)), \t \in \hnz[2] \} \subset \hnz \rtimes \tH, \ee encapsulating the ``redundancy'' between $\hnz$ and $\tH.$ 

\begin{nlem} \label{J:normal} The set $J \subset \hnz \rtimes \tH$ is a normal subgroup.  \end{nlem}
\begin{proof} First we verify that $J$ is a subgroup. As $J$ contains the identity this amounts to showing that it is closed under multiplication. For $a, b \in \Pi$  we compute  \be{J:sbgp} \t_a ^{-1} j(\t_a) \,  \t_b^{-1}  j(\t_b)  = \t_a^{-1} \t_b^{-1}  \t_b( \h_a(-1) ) \h_b(-1). \ee By (\ref{sa:2:hz}) and $\rh{2}$ this becomes  $$ \t_a^{-1} \, \t_b^{-1} \, \h_a(-1) \h_b(-1) (-1, -1)^{ \Bs(\av, \bv) \la b, \av \ra} = \t_a^{-1} \t_b^{-1} \, (-1, -1)^{\B(\av, \bv)} \h_b(-1) \h_a(-1) (-1, -1)^{ \Bs(\av, \bv) \la b, \av \ra}.$$ Now $\Bs(\av, \bv) + \Bs(\av, \bv) \la b, \av \ra = \Bs(\av, \bv) ( \la b, \av \ra + 1) $ is always even: either $\la b, \av \ra$ is odd, and so $1 + \la b, \av \ra$ is even, or $\la b, \av \ra$ is even and then so is $\Bs(\av, \bv)$ by (\ref{bq:db}). Thus the previous displayed equation is just $( (\t_b \t_a)^{-1}, j(\t_b \t_a))$.

Next, we verify that $J$ is closed under conjugation by $\h_{\gam}(t)$ for $\gamma \in \Pv_e,$ $t \in F^*$ using Lemma \ref{th:sq}: \be{J:nor-1} \h_{\gamma}(t) \t_a^{-1} \h_a(-1)  \h_{\gamma}(t)^{-1} &=& \t_a^{-1} \t_a( h_{\gamma}(t) )  \h_a(-1) h_{\gamma}(t)^{-1} \\
&=& \t_a^{-1} \, \h_a(-1) \h_{\gamma}(t)  \h_a(-1)^{-1} \h_a(-1) \h_{\gamma}(t)^{-1} \\
&=&  \t_a^{-1} \h_a(-1). \ee   

Finally, we show $J$ is closed under conjugation by $\si_{\gamma}$ for $\gamma \in \Pi$: First note that \be{J:nor-2} \si_{\gamma} \t_a^{-1}  \h_a(-1) \s_{\gamma} &=& \si_{\gamma} \t_a^{-1} \s_{\gamma} \si_{\gamma}(\h_a(-1)) \s_{\gamma} \\ \label{J:nor-3}
&=& \si_{\gamma} \, (\si_{\gamma} \, \t_a)^{-1} \h_a(-1) \h_{\gamma}(-1^{\la \gamma, \av \ra}). \ee We have two cases. If $\la \gamma, \av \ra$ is even, this is computed using $\rtnz{2}$ as $\si_{\gamma} (\t_a \si_{\gamma})^{-1} \h_a(-1) = \t_a^{-1} \h_a(-1).$ On the other hand, if $\la \gamma, \av \ra$ is odd, then (\ref{J:nor-3}) becomes, using again $\rtnz{2},$  \be{Jnor-4} \si_{\gamma} (\t_a \s_{\gamma})^{-1} \, \h_a(-1) \h_{\gamma}(-1) = \si_{\gamma} \si_{\gamma} \t_a^{-1} \h_a(-1)  \h_{\gamma}(-1) = \t^{-1}_{\gamma} \t^{-1}_a  \h_a(-1) \h_{\gamma}(-1), \ee concluding the proof. \end{proof}

\spoint \label{Ntilde:const} The map $\varphi': \hnz \ltimes \tH \rr N$ which sends $\h_{\gamma}(s) \mapsto h_{\gamma}(s)$ for $\gamma \in \Pv_e, s \in F^*$ and $\r_a \mapsto \dw_a, $ for $ a \in \Pi$ is a homomorphism, and moreover $J$ lies in the kernel of this map, so we have a diagram \be{tN:cover} \xymatrix{ \hnz \rtimes \tH \ar[r]^{\varphi'} \ar[d]^{\pi}  & N \\ \tN \ar[ur]_{\varphi} & } \ee where $\pi$ is the natural quotient map with kernel $J$.
\begin{nprop} The kernel of $\varphi$ is $A$, i.e. there is an exact sequence $1 \rr A \rr \tN \stackrel{\varphi}{\rr} N \rr 1.$ \end{nprop}
\begin{proof} If $x \in \ker(\varphi)$, choose $(\r, \h)  \in \hnz \ltimes \tH$ which lies over $x$ so that we also have $\varphi'(\r \, \h)=1.$ We would like to argue that there exists $a \in A$ such that $(\r, a \h) \in J,$ i.e. $j(\r^{-1}) \in A \, \h.$ As $\varphi'(\h, \r)=1$ we have $\varphi'(\h) = \varphi'(\r)^{-1}.$ The left hand side of this expression lies in $H$ and the right hand side in $N_{\zee},$ so both sides lie in $N_{\zee} \cap H = H_{\zee}.$ In particular, as $\varphi'(\h) \in H_{\zee}$ we conclude by the definition of $\varphi$ and by Lemma \ref{Hto:lem} that there exists $\zeta \in A$ such that we can write (with respect to some fixed ordering on $\Pi$) $\zeta \h= \prod_{a \in \Pi} \h_{a}(s_a)$ for some $\zeta \in A$ with each $s_a \in \{\pm 1 \}.$ Letting $x= \varphi'(\r^{-1}) \in N_{\zee}$, we thus have $x = \prod_{a \in \Pi} \, t_a^{\epsilon_a}$ where $\epsilon_a$ is $0$ or $1$ depending on whether $s_a$ is $1$ or $-1.$ From Lemma \ref{Nz:ker}, it follows that $\r^{-1} = \mf{z} \, \prod_{a \in \Pi} \, \t_a^{\epsilon_a}$ where $\mf{z}$ is some product of elements from $\{ \t_a^2 \mid a \in \Pi \}.$ As $j(\mf{z}) \in A$, \be{ker:N}  j(\r^{-1}) = j(\mf{z}) \prod_{a \in \Pi} \h_a(s_a) \in A \, \h. \ee \end{proof}

\label{tN:not}  We shall adopt the following notation for the remainder of the paper. For each $a \in \Pi$, the image in $\tN$ of $\r_a \in \hnz$ will be denoted by $\w_a.$ Also, for each $a \in \Pv_e$ and $s \in F^*,$ the image in $N$ of $\h_a(s)\in \tH$ under the map $\varphi$ will continue to be denoted by $\h_a(s)$. 


\subsection{Constructing the cover} \label{sec:cover1}

\newcommand{\nt}{\widetilde{n}}
\newcommand{\hh}{\widetilde{h}}
\newcommand{\nti}{\widetilde{\nu}}
\newcommand{\os}{\mathbf{1}_S}

In this section, we construct a group $E$ of operators acting on the fiber product $S:= G \times_N \tN$ following \cite[p.40-47]{mat} We study in this section various relations among the elements of $E$, and then explain in Corollary \ref{p:ker} why $E$ can be regarded as a central extension of $G$ by $A.$

\spoint In the previous section, we have constructed a cover $\varphi:\tN \rr N,$ and in \S \ref{sec:bru} a map $\nu: G \rr N$. Using these two ingredients, we can define  \be{S:def} S:= G \times_N \tN = \{ (g, \nt) \in G \times \wt{N} \mid  \nu(g) = \varphi(\nt) \} \ee where $p$ and $\nti$ are the natural projections making the following diagram of sets Cartesian  \be{S:car} \xymatrix{ S \ar[d]_p \ar[r]^{\nti} & \wt{N} \ar[d]^{\varphi} \\ G \ar[r]^{\nu} & N } \ee Let us denote by $\os:= (1, 1) \in S$ from now on. Keeping the notation from the end of \S \ref{tN:not}, we consider the following operators on $S$ (in each of the below $(g, \nt) \in G \times \tN$ is assumed to lie in $S$)

\begin{enumerate}
\item Let $\h \in \tH$ and define $\lam(\hh) (g, \nt) = (\varphi(\hh) g, \hh \nt).$ 
\item Let $u \in U$ and define $\lam(u) (g, \nt) = ( u g, \nt).$
\item Let $a \in \Pi.$ Recall the dichotomy of Lemma \ref{bru:two}, and define \be{lam:a} \lam_a (g, \nt) = \begin{cases} (\dw_a g, \w_a \nt) & \text { if } \nu(\dw_a g) = \dw_a \nu(g); \\ (\dw_a g, \h_a(s) \nt) & \text{ if } \nu( \dw_a g) = h_a(s) \nu(g),\ s \in F^*.  \end{cases}  \ee 
\end{enumerate} 

\begin{nrem} If we introduce the notation $\left[ \nu(w_a g) \nu(g)^{-1} \right]^{\sim}$ for either $ \w_a$ or $\h_a(s)$ depending on whether $\nu(w_a g) \nu(g)^{-1}$ is either $w_a$ or $h_a(s)$ for some $s \in F^*$, then $\lam_a$ can also be written more compactly as \be{lam:a:new} \lam_a(g, \tn) = (w_a g, \left[ \nu(w_a g) \nu(g)^{-1} \right]^{\sim} \nt). \ee \end{nrem}

It is clear that $\lambda(\hh^{-1})$ and $\lambda(u^{-1})$ are inverses to $\lambda(\hh)$ and $\lambda(u)$ respectively. It is also easy to see that $\lambda_a$ must be a bijection for each $a \in \Pi$ (see also Proposition \ref{lam-prop}, (5) which shows that it is equal to $\lam_a \lam(\h_a(-1)^{-1})$). One can also check the following formula works \be{lam-inv} \lam_a^{-1} (g, \nt) = \begin{cases} ( \dw_a^{-1} g, \w_a^{-1} \nt)  & \text{ if } \nu(\dw_a^{-1} g) = \dw_a^{-1} \nu(g); \\ \left( \dw_a^{-1} g, \h_a(s^{-1})^{-1} \nt \right) & \text{ if } \nu(\dw_a^{-1} g) = h_a(s) \nu(g). \end{cases} \ee 

\begin{de} Let $E \subset \Aut(S)$ denote the subgroup generated by \be{E:def} \begin{array}{lcr} \lambda_a (a \in \Pi), &  \lambda(\hh)\ (\hh \in \tH), & \text{and } \lambda(u)\ (u \in U). \end{array} \ee \end{de}

\begin{nlem} \label{E:trans} The group $E$ acts transitively on $S$. \end{nlem}
\begin{proof} If $(g, \nt) \in S,$ by the Bruhat decomposition we write $g=u n u'$ for $u, u' \in U$ and where $n \in N$ is equal to $\varphi(\tn)$. By construction, we can write $\nt=\hh \widetilde{w}$ where $\hh \in \tH$ and $\widetilde{w}$ may be written as $\w = \w_{b_1} \cdots \w_{b_n}$ with each $b_i \in \Pi.$ Writing $\dw$ to be the same product with each $\w_{b_i}$ replaced by $\dw_{b_i}$ we see by definition of the operators described above that $\lam_{b_1} \cdots \lam_{b_n} \lam(u') \, \os = (\dw \, u' , \w).$ Applying $\lam(u) \lam(\hh)$ to this expression, we obtain $(g, \nt)$.  \end{proof}

\spoint The following is a summary of the main properties of the operators in $E$ that we shall need (cf. \cite[Lemma 6.7]{mat} for (1)-(6), and \S7 in \emph{op. cit} for (7)).

\begin{nprop} \label{lam-prop} The following relations hold in the group $E.$

\begin{enumerate}
\item The map $\hh \mapsto \lambda(\hh)$ is an injective homomorphism $\lambda: \wt{H} \rr E$ and $\lam(A) \subset E$ is central. 
\item The map $u \mapsto \lambda(u)$  is an injective homomorphism $\lambda: U \rr E.$ 
\item For $\hh \in \tH$ and $u \in U$ we have $\lam(\hh) \lambda(u) \lambda(\hh)^{-1} = \lambda (\varphi(\hh) u \varphi(\hh)^{-1})$.
\item For $u \in U^a$ (cf. \eqref{Ua:dec}) we have $\lam_a  \lam(u) \lam_a^{-1} = \lam( \dw_a u \dw_a^{-1})$ and also $\lam_a^{-1} \lam(u) \lam_a = \lam (\dw_a^{-1} u \dw_a ).$
\item We have $\lambda_a^2 = \lambda(\h_a(-1)).$ 

\item We have $\lam_a^{-1} \lam(\hh) \lam_a = \lam( \w_a^{-1} \, \hh \, \w_a).$ 
\item The elements $\lambda_a (a \in \Pi)$ satisfy the braid relations. \end{enumerate}

\end{nprop}

The remainder of this section is devoted to the proof of this result. Parts (1)-(6) follow by direct calculations whereas (7) first requires a rank 2 reduction (to finite dimensional root systems) and then an analysis as in Matsumoto \cite[\S 7]{mat}.  Fix some element $(g, \nt) \in S$ for the remainder of the argument.

\tpoint{Proof of Proposition \ref{lam-prop}, parts (1)-(2)} Clearly the maps $\lambda: \tH \rr E$ and $\lambda: U \rr E$ are homomorphisms. To verify the injectivity of the first claim, suppose that $\hh \in \wt{H}$ acts as the identity on $S$, so that $\lambda(\hh).\os = (\varphi(\hh), \hh)= \os$, which implies that $\hh=1.$ Similarly one verifies the injectivity of $\lambda: U \rr E.$  

As $\lambda$ is a homomorphism and since $A \subset \tH$ is central, it follows that $\lambda(A)$ commutes with $\lam(\tH).$ Since  \be{A:cen.3} \begin{array}{lcr} \lam(u) \lam(x) (g, \nt) = \lam(u) ( g, x \nt) = (ug, x \nt) & \text{ and } &  \lam(x) \lam(u) (g, \nt) = \lam(x) ( ug, \nt) = (ug, x \nt) \end{array} \ee we see that the elements $\lam(A)$ commute with $\lam(U)$. Finally, we compute 
$$ \lam(x) \lam_a (g, \nt) = (\dw_a g, x \, [ \nu(\dw_a g) \nu(g)^{-1}]^{\sim} \, \nt )  \text{ and  }  \lam_a \lam(x) (g, \nt) = \lam_a (g, x \nt) = (\dw_a g, [\nu(\dw_a g) \nu( g)^{-1}]^\sim x \nt) $$
As $x \in A$ is central in $\wt{N},$ the commutativity of $\lam_a$ with $\lam(A)$ follows. 

\tpoint{Proof of Proposition \ref{lam-prop}, pt (3)} Let $\hh \in \wt{H}$ and $u \in U.$ As we noted above, $\lam(\hh)^{-1} = \lam(\hh^{-1})$, so \be{l:u:com} \lam(\hh) \lam(u) \lam(\hh)^{-1} (g, \nt) &=& \lam(\hh) \lam(u) (\varphi(\hh^{-1})g, \hh^{-1} \nt) = \lam(\hh) ( u \varphi(\hh^{-1})g, \hh^{-1} \nt) \\ &=& (\varphi(\hh) u \varphi(\hh^{-1})g, \hh \, \hh^{-1} \nt ) = (\varphi(\hh) u \varphi(\hh^{-1})g, \nt ) \\  &=& \lambda( \varphi(\hh) u \varphi(\hh^{-1}) (g, \nt). \ee  

\tpoint{Proof of Proposition \ref{lam-prop}, pt (4)} Again, this is a direct computation-- let us just verify the first statement, the second being similar. By definition, we have   \be{lamp:1}  \lam_a \lam(u) (g, \nt) &= &(\dw_a u g, [\nu(\dw_a u g) \nu(u g)^{-1} ]^{\sim} \, \nt)  \\    \lam( \dw_a \, u \, \dw_a^{-1}) \lam_a  (g, \nt) &=&  ( \dw_a \, u \, \dw_a^{-1} \dw_a g, [\nu(\dw_a g) \nu(g)^{-1}]^{\sim} \nt)   .\ee As $u \in U^a$ we have $\dw_a u \dw_a^{-1} \in U$ and so $\nu( \dw_a u g) \nu(u g)^{-1}  = \nu(\dw_a g) \nu(g)^{-1},$ and the desired result follows.

\tpoint{Proof of Proposition \ref{lam-prop}, pt (5)} By definition we can easily compute the following, \be{lama:sq} \lam_a^2(g, \nt) = \left( \dw_a^2 g, [ \nu(\dw_a^2 g) \nu(\dw_a g)^{-1} ]^{\sim} \, [\nu(\dw_a g) \nu(g)^{-1}  ]^{\sim} \nt \right).\ee If $\nu(\dw_a g) \nu(g)^{-1}  = \dw_a$ then as $\dw_a^2 = h_a(-1)$ we must also have $\nu(\dw_a^2 g) \nu(\dw_a g)^{-1}  = \dw_a$ and so the above expression is just $\lam(\h_a(-1))(g, \nt)$ as desired. On the other hand, if $\nu(\dw_a g) = h_a(s) \nu(g)\, s \in F^*$ then \be{} \nu(\dw_a^2 g) = h_a(-1) \nu(g) = h_a(-1) h_a(s)^{-1} \nu(\dw_a g) = h_a(- s^{-1}) \nu(w_a g).\ee Thus the right hand side of (\ref{lama:sq}) becomes $(h_a(-1)g, \h_a(-s^{-1}) \h_a(s) \nt)$. Using $\rh{1}$ and the fact that $(-s^{-1}, s)=(-s, s)^{-1}=1$ by Lemma \ref{stein} (iii), it follows that $\h_a(-s^{-1}) \h_a(s) = \h_a(-1)$.

\tpoint{Proof of Proposition \ref{lam-prop}, pt (6)}  Let us now assume without loss of generality that $\hh = \hh_b(s)$ for some $b \in \Pv_e, s \in F^*$. The general case will follow from the fact that $\lambda: \tH \rr E$ is a homomorphism.  We have two cases to consider, either $\nu(\dw_a g) = h_a(t) \nu(g), \, t \in F^*$ or $\nu(\dw_a g) = \dw_a \nu(g)$. We leave the latter, simpler case to the reader and focus on the former. Under this assumption, a direct calculation shows that \be{} \nu(\dw_a^{-1} h_b(s) \dw_a g) = h_a(t s^{\la a, \bv \ra})^{-1} \nu( h_b(s) \dw_a g). \ee 
Using the definition of \eqref{lam-inv}, we then have  \be{} \lam_a^{-1} \lam(\h_b(s)) \lam_a (g, \nt) = ( \dw_a^{-1} \, h_b(s) \,  \dw_a g,  \h_a(t s^{ \la a, \bv \ra})^{-1} \h_b(s) \h_a(t)  \nt  ), \ee and so we need only verify that \be{}   \h_a(t s^{ \la a, \bv \ra})^{-1} \h_b(s) \h_a(t) = \w_a^{-1} \h_b(s) \w_a = \h_b(s) \h_a(s^{- \la a, \bv \ra}). \ee Observe that $\h_a(x)^{-1} = \h_a(x^{-1}) (x, x)^{\Qs(\av)}$ for any $x \in F^*.$ So if we write $x= s^{ \la a, \bv \ra}$ then the left hand side of the above expression is equal to \be{} \h_a(t^{-1} x^{-1} ) (tx, tx)^{\Q(\av)} \h_b(s) \h_a(t) &=& (tx, tx)^{\Q(\av)} (t^{-1} x^{-1}, s)^{\Bs(\av, \bv)} \h_b(s) \h_a(t^{-1} x^{-1}) \h_a(t) \\
&=&  (t x , t x)^{\Q(\av)} (t^{-1}x^{-1}, s)^{\Bs(\av, \bv)} (t^{-1} x^{-1}, t)^{\Qs(\av)} \h_b(s) \h_a(x^{-1}).  \ee So we are reduced to showing that the central part of the above expression is equal to 1: \be{} (t x, t x)^{\Q(\av)} (t^{-1} x^{-1}, t)^{\Qs(\av)}  (tx, s)^{\Bs(\av, \bv)} = (tx, x)^{\Q(\av)} (tx, s)^{\Bs(\av, \bv)} = (tx, s)^{2 \Bs(\av, \bv)} =1.  \ee

\tpoint{Preliminaries for a rank two reduction}  Let $J \subset I$ be any subset, and let us introduce the following parabolic versions of our constructions. Set $\Pi_J:= \{ a_i \mid i \in J \}$ and define $G_J \subset G$ to be the subgroup generated by $x_a(s)$ and $x_{-a}(s)$ with $a \in \Pi_J, s \in F.$ We also set $W_J:= \la s_a \mid a \in \Pi_J \ra$ and choose a set of representatives $W^J$ for the cosets $W_J \setminus W$ which satisfy the condition \be{kost-rep} \ell(w_J w_1) = \ell(w_J) + \ell(w_1) \text{ for } w_J \in W_J, \, w_1 \in W^J. \ee  Writing $P_J:= B W_J B,$ defining  $A_J = \{ h \in H \mid h^{a_i} =1, i \in J \},$ and setting $U^J \subset U$ to be the subgroup spanned by root groups $U_a$ with positive (real) roots $a = \sum_i n_i a_i$ such that $n_i > 0$ for $i \in J$, we have decompositions \be{par:bru} G = \bigsqcup_{w \in W^J} P_J \dw B \text{ and } P_J = G_J A_J U^J. \ee Next we define $N_J = G_J \cap N$ and $\tN_J:= \varphi^{-1}(N_J)$ so that $S_J:= p^{-1}(G_J)$ fits into the following  \be{S:car} \xymatrix{ S_J \ar[d]_{p_J} \ar[r]^{\nti_J} & \wt{N}_J \ar[d]^{\varphi_J} \\ G_J \ar[r]^{\nu_J} & N_J } \ee where the maps in this diagram have the natural meaning. Define the subgroup $E_J \subset \Aut(S_J)$ generated by \be{E_J} \lambda(\hh) \  \hh \in \tH_J, \ \   \lambda(u) \  u \in U_J \ \ \text{ and } \lambda_a \  a \in \Pi_J.  \ee These elements also clearly define operators in $E$, and we note the following.

\begin{nlem} \label{J:red} The natural map $E_J \rr E$ is injective, i.e. if $e_J \in E_J$ acts trivially on $S_J$, then it acts trivially on $E$. \end{nlem} 

\begin{proof} Let $(g, \nt) \in S$ and use (\ref{par:bru}) to write $g= m_J g_1$ where $m_J \in G_J$ and $g_1 \in U^J \dw B$ with $w \in W^J$. Let us also choose a decomposition\footnote{This is not unique, but this poses no problem for the argument below.} $\nt = \nt_J \nt_1$ where $\nt_J \in \tN_J$ satisfies $\varphi(\nt_J) = \nu_J(m_J)$ and $\varphi(\nt_1) = \nu(g_1).$ Using these choices, we define the map \be{tauJ} \tau_J: S_J \rr S, \, (g_J, \tn_J) \mapsto (g_Jg_1, \tn_J \tn_1). \ee Now one can verify that $e_J \tau_J  = \tau_J e_J$ for $e_J \in E_J.$ In fact, checking this easily reduces to checking it for $e_J=\lam_a, \, a \in \Pi_J$. However, using (\ref{nu:bru}) and (\ref{kost-rep}) we have \be{nu:factor}  \nu( g_J g_1) = \nu(g_J) \nu(g_1) \text{ for any } g_J \in G_J, \ee and from this we easily see that $\lam_a \tau_J = \tau_J \lam_a$ for $a \in \Pi_J$.  Finally, we now compute \be{estab} e_J (g, \tn) = e_J (m_J g_1, \tn_J \tn_1) = \tau_J (e_J(m_J, \tn_J) ) = \tau_J ( m_J, \tn_J) = (m_J g_1, \tn_J \tn_1 ) = (g, \tn), \ee where the assumption on $e_J$ is used in the third equality.  \end{proof}

\tpoint{Simple transivity of $E_J$ and the Braid relations }  For any $J \subset I$, the action of $E_J$ on $S_J$ is transitive by an argument analogous to Lemma \ref{E:trans}. The remainder of this section will be devoted to proving the following stronger statement. 

\begin{nprop} \label{stran:2} The action of $E_J$ on $S_J$ is simply transitive when $J = \{ a , b \} \subset I$. \end{nprop}  The proof will be as follows: in \S \ref{s:E-r} we construct another family of operators on $S_J$ denoted by $E_J^*$ and acting ``on the right.''  This action is again easily seen to be transitive and moreover we show that the action of $E_J$ and $E_J^{\ast}$ commute. But if a set carries two transitive, commutative actions, then both of these actions must be simply transitive.  Before starting the proof, let us show how the braid relations for $\lam_a$ follow.

\begin{proof}[Proof of Proposition \ref{lam-prop}, pt. 7] Let $a= a_i, b= a_j$ for some $i, j \in I$. Let $h:= h_{ij}$ and suppose we wish to deduce $(\lam_a \lam_b)^h = (\lam_b \lam_a)^h.$ Apply both sides of this purported equality to $(1, 1) \in S_J$ to obtain $( (\dw_a \dw_b)^h, (\tw_a \tw_b)^h)$ and $( (\dw_b \dw_a)^h, (\tw_b \tw_a)^h)$ respectively. As the elements $\{ w_a, w_b \}$ and $\{ \tw_a, \tw_b \}$ both satisfy the braid relations in $G_J$ and $\tN_J$ (or $G, \tN$), we have an equality $(\lam_a \lam_b)^h = (\lam_b \lam_a)^h$ in $E_J$ by the simple transitivity of the $E_J$-action on $S_J$. But then using Lemma \ref{J:red}, we can deduce a similar equality as elements of $E.$ \end{proof}

\tpoint{Construction of $E_J^*$} \label{s:E-r}   Let $\iota: S_J \rr S_J$ be the involution which sends $(g, \nt) \mapsto (g^{-1}, \nt^{-1}),$ and let $E_J^*:= \iota E_J \iota$. Then $E_J^*$ is generated by elements  \be{Er:gen} \begin{array}{lcr} \rho(\hh) := \iota \lam(\hh) \iota \text{ for } \h \in H_J, & \rho(u) := \iota \lam(u) \iota \text{ for } u \in U_J, & \rho_a := \iota \lam_a \iota \text{ for } a \in \Pi_J. \end{array} \ee A version of Proposition \ref{lam-prop} can then be proven for $E_J^*$ and we shall use these results from now on (and just cite the corresponding result for the $\lambda$-operators). It will also be useful to note the following more explicit description of the operators making up $E_J^*$, \be{rho-exp} \rho(\hh) (g, \nt) &=& (g \hh^{-1}, \nt \, \hh^{-1}) \text{ for } \hh \in H_J \\
\rho(u)(g, \nt) &=& (g u^{-1}, \nt) \text{ for } u \in U_J \\  \label{rhoa:def}
\rho_a(g, \nt) &=&  \begin{cases} (g \dw_a^{-1},  \nt \, \w_a^{-1}) & \text { if } \nu( \dw_a \, g^{-1}) = \dw_a \nu(g^{-1}); \\ ( g \dw_a^{-1},  \nt \, \h_a(t)^{-1} ) & \text{ if } \nu( \dw_a  g^{-1}) =  h_a(t) \nu(g^{-1}).   \end{cases} \ee 

\begin{nrem} Alternatively, we can write the definition of $\rho_a$ as follows. First we define \be{rhoa:alt} \left[ \nu(g)^{-1} \nu( g \dw_a^{-1}) \right]^{\sim}= \begin{cases}  \w_a^{-1} & \text{ if } \nu(g)^{-1} \nu( g \dw_a^{-1}) = \dw_a^{-1};  \\ \h_a(t)^{-1} & \text{ if } \nu(g)^{-1} \nu( g \dw_a^{-1}) = h_a(t^{-1}).  \end{cases} \ee Then using the fact that $\nu(g)^{-1} = \nu(g^{-1})$ we can write, \be{rhoa:alt} \rho_a(g, \tn) = \left( g \dw_a^{-1}, \tn \,\left[ \nu( g)^{-1} \nu(g \dw_a^{-1}) \right]^{\sim} \right). \ee  \end{nrem}

\tpoint{Commutativity of $E_J$ and $E_J^{\ast}$, preliminary reductions }   In the remainder of this section we will show that the action of $E_J$ and $E_J^*$ commute on $S_J$ where $J= \{ a, b \}.$ As $J$ is fixed, we shall omit it from our notations from now on. Note that we allow the case where $a=b.$ We begin with the following simple result which shows that the commutativity of $E$ and $E^*$ can be reduced to checking \be{left:1} \lam_a \rho_b = \rho_b \lam_a \text{ for } a, b \in \Pi .\ee 

\begin{nlem} \label{comm:simp}  Let $\h \in \tH, u \in U$ and $s \in S$. \begin{enumerate} 
\item The operators $\lambda(\h)$ and $\lambda(u)$ commute with the elements in $E^*$. Similarly the operators $\rho(\h)$ and $\rho(u)$ commute with the elements of $E$. 
\item  If $\lam_a \rho_b s  = \rho_b \lam_a s$ then $\lam_a \rho_b (\rho_b s) = \rho_b \lam_a (\rho_b s)$ and also $\lam_a \rho_b (\lam_a s) = \rho_b \lam_a (\lam_a s).$ \end{enumerate}
\end{nlem}
\begin{proof} Let us just verify the first claim of (1), the proof of the second being similar. By their definition, the elements $\lambda(\wt{H})$ clearly commute with $\rho(U)$ and $\rho(\wt{H}).$  Now, let $a \in \Pi$, $\hh \in \wt{H}$, let us examine the relation between $\lam(\hh) \rho_a$ and $\rho_a \lam(\hh)$ by applying both operators to an element $(g, \nt) \in S$. Unwinding the definitions, we find that both these operators lead to the same result since (write $h:= \varphi(\h)$ below)
\be{comm:easy:h} \nu(g)^{-1} \nu(g w_a^{-1})  =  \nu(h g)^{-1} \nu(h g w_a^{-1}). \ee Similarly we can show that  $\rho_a \lambda(u) = \lambda(u) \rho_a$ since $\nu(g)^{-1} \nu(g w_a^{-1})  = \nu(u g)^{-1} \nu(u g w_a^{-1}) .$  

For the second statement, we just apply Proposition \ref{lam-prop}, (5) and the first part.
 \end{proof}

\tpoint{A further reduction} It suffices to check (\ref{left:1}) on a certain restricted class of elements in $S$ as the next result shows. 

\begin{nlem} \label{rk2:red} Let $a, b \in \Pi.$ If $ \lambda_a \rho_b s = \rho_b \lambda_a s $ for all elements in $S$ of the form \be{rk2:form} s = \lam(u_a) \rho(v_b) (\dw, \w ) = (u_a \,  \dw \, v_b, \w) \ee where $w \in W$ $u_a \in U_a$ and $v_b \in U_b,$ then $\lambda_a \rho_b = \rho_b \lambda_a.$ \end{nlem}
\begin{proof} From the Bruhat decomposition for $G,$ every element of $S$ can be written in the form $\lambda(u) \lambda(\hh) \rho(v) [\dw, \w]$ for some $u, v \in U$, $\hh \in \tH$, $w \in W.$ Now write $u = u^a u_a$ and $v = v^b v_b$ according to decompositions $U= U^a \rtimes U_a$ and $V = V^b \rtimes V_b$ (cf. \eqref{U:dec}), one has $\lambda(u) = \lambda(u^a) \lambda(u_a)$ and $\lambda(v) = \lambda(v^b) \lambda(v_b)$, and so \be{red:1} \lam_a \rho_b \lambda(u) \lambda(\hh) \rho(v) [\dw, \w] = \lam_a \rho_b \lambda(u^a) \lambda(u_a)  \lambda(\hh) \rho(v^b) \rho(v_b) [\dw, \w]. \ee Setting $\tu^a:= w_a u^a w_a^{-1}$, $\ttv_b:= w_b u^b w_b^{-1},$ and $\hh':= \w_a \hh \w_a^{-1}$ we can use Lemmas \ref{lam-prop} and \ref{comm:simp} to deduce the last expression is equal to \be{red:2} \lam(\tu^a) \lam_a \rho_b \lam(u_a)  \lam(\hh) \rho(v^b) \rho(v_b) [ \dw, \nt] &=& \lam(\tu_a) \lam_a \rho_b \lam(\hh) \lam(u'_a)  \rho(v^b) \rho(v_b) \,  [ \dw, \w] \\
&=& \lam(\tu_a) \lam(\hh') \lam_a \rho_b \lam(u'_a) \rho(v^b) \rho(v_b) [ \dw, \w] \\
&=& \label{red:4}\lam(\tu_a) \lam(\hh') \rho(\ttv^b) \underbrace{ \lam_a \rho_b \lam(u'_a)  \rho(v_b) [ \dw, \w] }.
\ee Applying the hypothesis of the Proposition to the underbraced terms transforms this last expression to 
\be{red:3} \lam(\tu_a) \lam(\hh') \rho(\ttv^b) \,  \lam_b \rho_a \lam(u'_a)  \rho(v_b) [ \dw, \w]. \ee Running the steps in (\ref{red:2} - \ref{red:4}) in reverse now, we obtain the desired result. \end{proof}

Let $(g, \w)$ be an element as in Lemma \ref{rk2:form}, so that we have $g = u_a \, \dw \, v_b.$  Writing (cf. \eqref{lam:a:new}, \eqref{rhoa:alt}) \be{la:rb} \lam_a \rho_b (g, \tw) &=& \left( \dw_a g \dw_b^{-1}, \left[ \nu(\dw_a \, g \dw_b^{-1}) \nu(g \dw_b^{-1})^{-1} \right]^{\sim} \, \tw \,  \left[ \nu(g)^{-1} \nu(g \dw_b^{-1}) \right]^{\sim} \right)\\ \label{rb:la} \rho_b \lam_a (g, \tw) &=& \left( \dw_a g \dw_b^{-1}, \left[ \nu(\dw_a \, g ) \nu(g )^{-1} \right]^{\sim} \, \tw \,  \left[ \nu(\dw_a \, g)^{-1} \nu(\dw_a g \dw_b^{-1}) \right]^{\sim} \right), \ee we are reduced to verifying the equality of the above two expressions. There are now three cases to consider.

\tpoint{Case 1: when $wb \neq \pm a$}  \label{case1:wb,a} In this case, the equality of (\ref{la:rb}) and (\ref{rb:la}) follows immediately from,

\begin{nclaim} If $w b \neq \pm a,$ then  
 \be{c1:2} \nu(g)^{-1} \nu(g \dw_b^{-1} ) &=& \nu (\dw_a g)^{-1} \nu(\dw_a g \dw_b^{-1} ) \\ 
 \nu(\dw_a g \dw_b^{-1}) \nu(\dw_a g  )^{-1} &=& \nu(\dw_a g ) \nu(g)^{-1} .
 \ee \end{nclaim}

 \begin{proof}[Proof of Claim] Let us verify the first statement (\ref{c1:2}) as the proof of the second is similar. Turning to the left hand side (\ref{c1:2}), there are two cases to consider:
 
 \begin{enumerate}[a.] \item If $w b > 0$, then we claim that the left hand side of (\ref{c1:2}) is always equal to $w_b^{-1}$. Indeed, since $w b \neq a$ using (\ref{Ua:dec})  there exists $u^a \in U^a$ such that we may write \be{} g= u_a \dw v_b = u^a u_a \dw \ee and thus  $g \dw_b^{-1} = u^a u_a \dw \dw_b^{-1}.$ Hence, $\nu( g \w_b^{-1}) = \nu(g) \dw_b^{-1}.$  Turning to the right hand side of (\ref{c1:2}), since $\dw_a$ normalizes $U^a$ we can write $\dw_a g = \dw_a u^a_1 \, \dw_a u_a \dw$ and $\dw_a g \dw_b^{-1} = u_1^a \, \dw_a u_a \dw \dw_b^{-1}$ for some $u^a_1 \in U^a.$  Thus we obtain \be{ab:na-1} \nu( \dw_a g ) = \nu( \dw_a u_a \dw ) \text{ and } \nu(\dw_a g \dw_b^{-1}) = \nu(\dw_a u_a \dw \dw_b^{-1}). \ee Next we note that since $w b \neq a$ we have \be{ab:na-2} w^{-1} a < 0 \text{ if and only if } (w w_b^{-1})^{-1}(a) < 0, \ee from which it follows that $ \nu( \dw_a g )  \nu( \dw_a g \dw_b^{-1}) = \dw_b^{-1}.$
 
 \item If $w b < 0$ and $u_b=1$ then $g = u_a \dw$ and the argument as in the previous case applies. So we assume $w b < 0$ and $u_b \neq 1.$ Then it follows that the left hand side of (\ref{c1:2}) is equal to $h_b(s)$ for some $s \neq 0$. Similar reasoning as in the previous case allows us to conclude that the right hand side of (\ref{c1:2}) is also equal to $h_b(s).$  \end{enumerate} \end{proof}
 
\newcommand{\nut}{\tilde{\nu}} 
 
\tpoint{Case 2: when $w b = a$} \label{lp:case2} As $\dw u_b \dw^{-1} \in U_a,$ we may as well assume $g= u_a \dw,$ and so $\nu(g)^{-1} \nu(g \dw_b^{-1}) = \dw_b^{-1}.$ Now, if $u_a=1$ the desired equality of (\ref{la:rb}) and (\ref{rb:la}) is obvious, so let us assume $u_a=x_a(t),$ $t \in F^*.$ We may then compute using (\ref{rk1:l}) that $\nu(\dw_a g \dw_b^{-1}) \nu(g \dw_b^{-1})^{-1}= h_a(-t^{-1})$, and so the $\tN$ component from (\ref{la:rb}) is equal to $\h_a(-t^{-1})\,  \tw \, \w_b^{-1}.$ 

To compare with (\ref{rb:la}) we need to compute $\nu(\dw_a g) \nu(g)^{-1}$ and $\nu(\dw_a g )^{-1} \nu(\dw_a g \dw_b^{-1})$. The former is $\dw_a$ since $w^{-1}(a) >0$. As for the latter, noting $\nu(\dw_a g) = \dw_a \dw$ and using $w^{-1} x_a(s) w = x_b(s)$ \footnote{This follows from a case-by-case using the possibilities enumerated in Lemmas \ref{lem:eltsfixingroot} and \ref{lem:weylatob} and sign rules listed in Remark \ref{signs}. Note that we are always invoking these rules for finite-type root systems.} with (\ref{rk1:r}),  \be{wb:a-1} \nu(\dw_a g \dw_b^{-1}) = \nu (\dw_a x_a(t) \dw \dw_b^{-1}) = \nu(\dw_a  \dw x_b(t) \dw_b^{-1} ) = w_a \dw h_b(t). \ee  
So $\nu(\dw_a g )^{-1} \nu(\dw_a g \dw_b^{-1}) = h_b(t)$ and the $\tN$ component of (\ref{rb:la}) is equal (cf \ref{rhoa:alt}) to $\w_a \tw \h_b(t^{-1})^{-1},$ and we are left to show that \be{} \w_a \, \tw \, \h_b(t^{-1})^{-1} = \h_a(-t^{-1})\,  \tw \, \w_b^{-1}. \ee Using Lemma \ref{w:h}, the right hand side is transformed to $\tw \, \h_b(-t^{-1}) \w_b^{-1}$ and so it suffices to check $\w_a \, \tw \, \h_b(t^{-1})^{-1} \tw_b = \tw \, \h_b(-t^{-1}) .$ We compute using Lemma \ref{s-hom}(1), the fact that $\w_a \tw = \tw \w_b$ \footnote{Again, this can be verified again on a case-by-case basis using the possibilities enumerated in Lemmas \ref{lem:eltsfixingroot} and \ref{lem:weylatob}}, and (\ref{ha:inv}): \be{} \tw_a \, \tw \, \h_b(t^{-1})^{-1} \tw_b &=& \tw_a \, \tw \, \tw_b \, \h_b(t)^{-1} = \tw \, \tw_b \tw_b \h_b(t)^{-1} \\
&=& \tw \, \h_b(-1) \, \h_b(t^{-1}) (t, t)^{\Qs(\bv)} = \tw \h_b(-t^{-1}) (-1, t)^{\Q(\bv)} (t, t)^{\Qs(\bv)}. \ee Using Lemma \ref{stein}, we find $(-1, t)^{\Q(\bv)} (t, t)^{\Qs(\bv)}=1$ and so the proof is concluded.
 
 \tpoint{Case 3: $wb=-a$} \label{lp:case3} Let us assume $g = u_a \dw v_b.$ Using Lemma \ref{comm:simp}, pt 2 we may assume both $u_a, v_b \neq 1$ for otherwise we can reduce easily to the previous case of $wb =a.$ 

 \begin{nclaim} Assume $u_a= x_a(s)$ and $v_b= x_b(t)$ with $st(1-st) \neq 0.$ Then \be{wb:-a} \nu ( \dw_a g \dw_b^{-1} ) = \nu(\dw_a g ) h_b(t - s^{-1})  =  h_a(- (s - t^{-1})^{-1}) \nu(g \dw_b^{-1} ). \ee \end{nclaim} 
\begin{proof}  From (\ref{rk1:l}) we have $\nu(w_a g) = h_a(-s^{-1}) \nu(g)$, and using the rank one equality (\ref{rk1})
\be{wb:-a,1} x_{a}(s) = x_{-a}(s^{-1}) h_a(s) \dw_a x_{-a}(s^{-1}), \ee we find that there exists $u \in U$ such that  \be{wb:-a,2} \dw_a g \dw_b^{-1}  = \dw_a x_{-a}(s^{-1}) h_a(s) \dw_a x_a(s^{-1})  \, \dw \, x_b(t) \, \dw_b^{-1}  =  u \,  h_a(-s^{-1}) x_a(s^{-1}) \dw x_b(t) \, \dw_b^{-1} . \ee Using the fact that $\dw^{-1} x_a(y) \dw = x_b(-y)$ for $y \in F$ (same reasoning as before (\ref{wb:a-1}) ) the above  becomes  \be{wb:-a,3} u h_a(-s^{-1}) x_a(s^{-1}) \dw x_b(t-s^{-1}) w_b \in U \, h_a(-s^{-1}) \dw h_b(t- s^{-1}) U, \ee where we have now also used (\ref{rk1:r}). This proves the first equality of the claim. The second is proven similarly. 
\end{proof}

Assume that $st \neq 0$ but $1=st$. Then $t - s^{-1}=0$ and so from (\ref{wb:-a,3}) we find that $w_a g w_b \in U H u_a w w_b.$ Using the previous remark as well as Lemma \ref{comm:simp}, we can again reduce to the previous case of $wb=a$. So, without loss of generality we may assume $st(1-st) \neq 0$ and hence apply the claim. 

Using the claim, to compare (\ref{la:rb}) and (\ref{rb:la}) we are reduced to showing the equality, 
\be{b=-a:1}  \h_a( - (s- t^{-1})^{-1}) \,  \tw \, \h_b(t^{-1})^{-1} = \h_a(-s^{-1}) \, \tw \, \h_b( ( t- s^{-1})^{-1})^{-1}.  \ee Using Lemma \ref{w:h}, this amounts to showing 
\be{b=-a:2}   \h_a( - (s-t^{-1})^{-1}) \h_a(t)^{-1} = \h_a(-s^{-1}) \h_a(t - s^{-1})^{-1}. \ee This follows from the following result on Steinberg symbols: for $s, t \in F^*$ such that $st(1-st) \neq 0$ then \be{1:x} (- (s-t^{-1})^{-1}, t^{-1}) (t, t) = (-s^{-1}, (t - s^{-1})^{-1}) ( t - s^{-1}, t - s^{-1}). \ee By using the bimultiplicativity (and also Lemma \ref{stein} (ii)), this amounts to showing \be{1:x-2} (1- st, t) = ( 1-st, t - s^{-1}). \ee In other words, we need to show that \be{1:x-3} (1 - st, t^{-1}) (1 - st, t - s^{-1} ) = (1 - st, 1 - (st)^{-1}) = 1. \ee The last statement follows since if $x \neq 1$ we have $(1-x,x) =1$ and if also $x \neq 0$ we obtain  \be{1:x-4} (1-x , 1 - x^{-1})= (1 - x, x)(1-x, 1- x^{-1}) =(1-x, x-1) =1, \ee where we have used Lemma \ref{stein} (iii) in the last step.

\subsection{Some further properties of the cover $\tG$}

To sum up, we have now constructed a group $E$ which satisfies the properties listed in Proposition \ref{lam-prop}. Moreover, there exists a map $\varphi:E \rr G$ which sends $e \in E \mapsto p (e.\os)$ in the notation of \eqref{S:car}. Our aim in this section is to show that this map makes $E$ into a central extension of $G$ with kernel $A$ and satisfying the Tits axioms \rd{1}-\rd{5} described in \S \ref{s:axioms}.

\spoint As noted in Proposition \ref{lam-prop}, the map $\lam:\tH \rr E, \hh \mapsto \lam(\hh)$ is injective. From Proposition \ref{lam-prop} and the presentation of $\tN_{\zee}$ given in \S \ref{s:ncover} we find that the map $\lam: \tN_{\zee} \rr E$ sending $\r_a \mapsto \lam_a$ for $a \in \Pi$ is a homomorphism, and moreover Proposition \ref{lam-prop}, (6) shows that we obtain an map $\tN_{\zee} \rtimes \tH \rr E$ which we continue to denote by $\lambda$. Now,  Proposition \ref{lam-prop}, (5) together with the definition of $J$ in \eqref{J} shows that this descends to a map $\tN \rr E$ which we continue to denote by $\lambda.$ The map $\lam: \tN \rr E$ is injective. Indeed  $\lam(\tn) \os= (\varphi(\tn), \tn)$ and we have already argued that $\varphi(\tn)=1$ implies $\tn \in A$, but $\tn=1$ from Proposition \ref{lam-prop} (1). Denoting by $\lam(\tH)$ and $\lam(\tN)$ the image in $E$ of the group $\tH$ or $\tN$, we have the following

\begin{nlem} \label{lam:W} The map $\lam: \tN \rr E$ induces an isomorphism $\tN / \tH \rr \lam(\tN) / \lam(\tH).$ Moreover, both groups are isomorphic to $W$, the Weyl group of $G$.
 \end{nlem}\begin{proof} The first statement is immediate from the injectivity of $\lambda$ on $\tN$ (and $\tH$). The fact that $\tN / \tH$ is isomorphic to $W$ follows from  Lemma \ref{Nz:ker} and the explicit construction of $\tN$.  \end{proof}
 
For each $w \in W$ with reduced decomposition $w = s_{b_1} \cdots s_{b_n}$ with $b_i \in \Pi$, we set \be{lam:w} \lam_w:= \lam_{b_1} \cdots \lam_{b_n}, \ee which is well defined by Proposition \ref{lam-prop} (7) . Writing $\lam(U):= \{ \lam(u) \mid u \in U \}$ we next note that \be{E:bru} E = \bigcup_{w \in W} \lambda(U) \lambda(\tH) \, \lambda_w \,  \lambda(U). \ee The proof of this is standard: one verifies the right hand side of (\ref{E:bru}) is closed under multiplication by the generators of $E$ using the relations verified in Proposition \ref{lam-prop}. We refer to \cite{mat} for the details (there are no further Kac-Moody complications). Using (\ref{E:bru}) we have the following two results.

\begin{ncor} \label{p:ker} The kernel of the map $p:E \rr G$ is $\lam(A)$ (which is isomorphic to $A$). \end{ncor} 
\begin{proof} Let $e \in E$, which we may write using (\ref{E:bru}) as $e = \lam(u_1) \lam(\hh) \lam_w \lam(u_2)$ for $u_1, u_2 \in U,$ $w \in W,$ and $\hh \in \tH.$ One easily checks that $e.\os = (u_1 \varphi(\hh) w u_2, \hh \tw)$. So if $e \in \ker(p)$ we must have $w=1$, i.e. $e \in \lam(U) \lam(\hh)$. From here, it follows easily that in fact $e \in \lam(A)$. \end{proof}

\begin{ncor} \label{E:st} The group $E$ acts simply transitively on $S.$ \end{ncor} 
\begin{proof} Recall that $E$ acts transitively on $S$ by Lemma \ref{E:trans}, so it suffices to show that if $e \in E$ and $e. \os = \os$, then $e=1$. This however follows from the decomposition \eqref{E:bru}, which again reduces us to checking the claim for operators of the form $\lam(\h) \lam(u)$ with $\h \in \tH, u \in U$. \end{proof} 

\tpoint{On the unipotent elements} From the remarks at the end of \S \ref{s:stein-functor}, there exists for each $w \in W$ and $a \in R_{re}$ a sign $\eta_{w, a}= \pm 1$ such that \be{eta:w} \dw x_a (\eta_{w, a} \, s) \dw^{-1} = x_{w a}(s). \ee  We can use this (together with the fact that the orbit of the simple roots is the set of real roots) to define unipotent elements in $E$ corresponding to all real roots (i.e. to negative real roots as well).

\begin{nlem} \label{well-def} Let $a \in R_{re}$ and $w \in W, i \in I$ such that $w(a_i) = a$. Then the operator \be{xi:a} \xi_a(s):= \lam_w \lam(x_{a_i}(\eta_{w, a_i} s)) \lam_w^{-1} \ee depends only on $a$ and $s$ (and not on the choice of $w$ and $i$). \end{nlem} 
\begin{proof} Suppose $v \in W, j \in \Is$ are such that $v (a_j) = a$. We would like to show that \be{un-well} \lam_w \lam(x_{a_i}(\eta_{w, a_i} s)) \lam_w^{-1} = \lam_v \lam(x_{a_j}(\eta_{v, a_j} s)) \lam_v^{-1}. \ee Using Corollary \ref{E:st}, it suffices to verify that \be{un-well:2} \lam_v^{-1} \lam_w \lam(x_{a_i}(\eta_{w, a_i} s)) \lam_w^{-1} \lam_v \os = (x_{a_j}(\eta_{v, a_j}(s), 1). \ee As  $(w^{-1} v)^{-1}(a_i) > 0$,  the left hand side is computed the from definitions to be \be{un-well:3} \lam_v^{-1} \lam_w ( x_{a_i}(\eta_{w, a_i} s) \dw^{-1} \dot{v}, \tw^{-1} \wt{v}) = ( \dot{v}^{-1} \dw x_{a_i}(\eta_{w, a_i} s)  \dw^{-1} \dot{v}, \wt{v}^{-1} \tw \tw^{-1} \wt{v}), \ee which is easily seen to be $(x_{a_j} (\eta_{v, a_j} s), 1).$  \end{proof}  

\begin{nprop} \label{Um:split} The map $\xi: U \rr E$ , $x_{-a}(s) \mapsto \xi_{-a}(s)$ for $a \in R_{re, +}, s \in F$ is a homomorphism (i.e. the cover $E$ splits over $U^-$). \end{nprop} 
\begin{proof} The only relations in $U^-$ are the ones (\ref{kmg:sg_rel1}) with $a, b \in R_{re, -}$ forming a prenilpotent pair (cf. \S \ref{prenilp}). But form the comments in \S \ref{prenilp}, the relations (\ref{kmg:sg_rel1}) holds in the group $U_{-, w}:= \dw U_w \dw^{-1}$ for some $w \in W$. Defining $\tU_{-, w}$ to be the subgroup generated by $\{ \xi_{\beta}(s) \mid \, s \in F \}$ with $\beta \in w R(w^{-1})$ one has $p( \tU_{-, w}) = U_{-, w}$ and it suffices to show that $p$ is injective when restricted to $\tU_{-, w}$. However, this follows easily as $p$ is injective when restricted to $\lam_w^{-1} \tU_{-, w} \lam_w \subset \lam(U).$ \end{proof} 

For each $a \in R_{re}$, let $\tU_a \subset E$ consist of the element $\{ \xi_a(s) \mid s \in F \}.$ Set now: \be{B:def}\begin{array}{lcccr} \tB:= \lam(U) \rtimes \lam(\tH), &  \tB^-:= \xi(U^-) \rtimes \lam(\tH), &  \tB_a:= \tU_a \rtimes \lam(\tH), &  \text{ and } \tG_a:= \la \tB_a, \tB_{-a} \ra \text{ for } a \in R_{re}  \end{array} \ee  

\begin{ncor} \label{titsax:met} The family $(E, (\tB_a)_{a \in R_{re}})$ satisfies the axioms \rd{1} - \rd{5}. \end{ncor}

We suppress the proofs, as they are straightforward using the simple transitivity and the corresponding properties of $G$. 

\section{Unramified Whittaker functions on Metaplectic Covering groups}

In this section, we specialize the construction of the previous one to define $n$-fold metaplectic covers of $\bG(\mc{K}),$ where $\mc{K}$ is a non-archimedean local field. The notation in \S \ref{p-adic} for local fields will be fixed throughout here.  We establish some basic structural properties of metaplectic covers in \S \ref{s:struc}, develop the notions of unramified Whittaker and Iwahori-Whittaker functions in \S \ref{s:whit}, and finally present our generalization of the Casselman-Shalika formula for unramified Whittaker functions in \S \ref{s:CS}. 

Throughout this section we impose the condition $q \equiv 1 \mod 2n$ where $q$ is the size of the residue field. See \S \ref{s:HS} and \S \ref{p-spe} for comments on the implications of this assumption. We use it to ensure the splitting of the integral subgroup of the torus. 

\subsection{Structure of covers over a local field} \label{s:struc}

\tpoint{Construction of the group $\tG$} Let $(I, \cdot, \D)$ be a simply-connected root datum and let $\bG$ be the associated Tits functor. We refer to $G:= \bG(\mc{K})$ as the ``$p$-adic'' Kac-Moody group (instead of the perhaps more correct ``Kac-Moody group over a non-archimedean local field'').  Fix a metaplectic structure $(\Qs, n)$ on $\mf{D}$ and then construct the metaplectic root datum as in \S \ref{s:met-rts}.  We write $\mf{D} = (\Lv, \{ \av_i \}_{i \in \Is} , \Lambda, \{ a_i \}_{i \in \Is} )$ and $\wt{\D}= (\tLv, \{ \tav_i \}, \tLam, \{ \ta_i \}).$ Let $(\cdot, \cdot)_n: \mc{K}^* \times \mc{K}^* \rr \bmu_n$ be the $n$-th order Hilbert symbol (cf. \S \ref{s:HS}), which is also a bilinear Steinberg symbol, and use it to construct a central extension as in \S \ref{sec:cover1}. We denote the extension by $\tG$ (as opposed to $E$) from now on. Recall also  that the structure theory of $\tG$ (e.g. the Bruhat/Birkhoff decompositions) are made possible via Proposition \ref{titsax:met}. Furthermore, our extension $\tG$ splits over $U$ and $U^-$ with a given splitting (cf. \ref{Um:split}). We shall just write $U,\ U^-$ for the corresponding subgroup of $\tG$ here. The Weyl group for $\tG$, i.e. $\tN / \tH$ will be denoted by $W$ (in agreement with Lemma \ref{lam:W}), and for each $w \in W$ we let $\tw$ be  element in $\tG$ constructed as a product of the elements $\lam_{a} \, (a \in \Pi)$ as in (\ref{lam:w}). In place of $\lam(\h)$ with $\h \in \tH$ we shall just write $\h$ here. We keep the notation for unipotent elements from (\ref{xi:a}), i.e. they are denoted as $\xi_a(s)$ with $a \in R_{re}$ and $s \in F$.

\tpoint{On $\tH$ and its abelian subgroups} We maintain the notation from \S \ref{more-sc-not} on simply connected root datum and fix both a basis $\av_i \, (i \in \Is_e)$ for $\Pv_e$ as well as an ordering on $\Is_e$ from now on. From Lemma \ref{H:straight}, every $t \in H$ may be written uniquely as $t= \zeta \, \prod_{i \in \Is_e} \h_{a_i}(s_i)$ with $s_i \in K^*$ and $\zeta \in \bmu_n$ . For $\lv \in \Lv,$ written in terms of the basis $\av_i (i \in \Is_e)$ as $\lv = \sum_{i \in \Is_e} c_i \av_i$, we define (always with respect to a fixed order on $\Is_e$) the elements in $\tH$ \be{p:lv} \pi^{\lv}:= \prod_{i \in \Is_e} \h_{a_i}(\pi^{c_i}). \ee We are interested now in certain abelian subgroups of $\tH$, and first define $\tH_{\O}$ to be the subgroup generated by the elements $h_{a_i}(s)$ with $s \in \O^*,$ $i \in \Is_e.$ Under the assumption that $q \equiv 1 \mod 2n$, it follows from the relations $\rh{1}$, $\rh{2},$ and \eqref{q:2} that $\tH_{\O}$ is in fact an abelian group.  We can also define the group $\mathtt{T}:= C_{\tH}(\tH_{\O})$ which is the centralizer of $\tH_{\O}$ in $\tH$. The same proof as in \cite[Lemma 5.3]{macn:ps} shows that $\mathtt{T}$ is a maximal abelian subgroup of $\tH$. Also note that we have a natural isomorphism \be{tH:Lv} \tH_{\O} \setminus \tH \stackrel{\sim}{\rr} \bmu \times \Lv  .\ee 

\tpoint{Iwahori and ``maximal compact'' subgroups} Let $K \subset G$ be the subgroup generated by elements $x_a(s)$ with $a \in R_{re},$ $s \in \O.$ We extend $\varpi: \O \rr \kappa$ to a map also denoted as $\varpi:K \rr \bG(\kappa)$ where $\kappa$ is the residue field of $\mc{K}$. The preimage in $G$ of $B(\kappa)$ and $B^-(\kappa)$ under this map will be denoted by $I$ and $I^-$ respectively. These are the \emph{Iwahori} subgroups of $G$. Note also that for each $w \in W$, the elements $\dw$ lie in $K$  (cf. \eqref{dw}). One then has the following (disjoint) decompositions of Iwahori-Matsumoto type \be{iwa-mat:K} K = \bigsqcup_{w \in W} I \, \dw \, I = \bigsqcup_{w \in W}  I^- \, \dw \, I^- = \bigsqcup_{w \in W} I^- \, \dw \, I^-, \ee  which follow from the arguments as in \cite[Proposition 2.4]{iwa:mat} combined with (\ref{G:w}). 

Let now $\tI \subset \tG$ be the subgroup generated by the following elements: \be{Ip:gen}\begin{array}{lcr}  \xi_a(s), s \in \O, a \in R_{+, re}, & \xi_{-a}(t), \, a \in R_{re, +}, \, t \in \pi \O & \text{ and } \tH_{\O}. \end{array} \ee  Writing\footnote{Note that $\tI_-$ is not the group $\tI^-$ introduced below and which denotes the Iwahori defined with respect to the opposite Borel. As we shall not use $\tI_-$ in the sequel, we hope this notation does not cause any confusion.} $\tI_+:= \tI \cap U$ and $\tI_-:= \tI \cap U^-$ one can show as in \cite[Theorem 2.5]{iwa:mat} that  \be{iwa:dec} \tI = \tI_+ \, \tH_{\O} \, \tI_-.  \ee We shall suppress the argument, but just note that it involves two parts: a reduction to rank $1$ (this can be accomplished using Proposition \ref{lam-prop}), and a direct rank 1 computation argument.

\renewcommand{\ti}{\widetilde{i}}

\begin{nlem} \label{iwa-split} The cover $p$ splits over $I$, i.e. $p|_{\tI}: \tI \rr I$ is an isomorphism.  \end{nlem} 
\begin{proof} Let $x \in \tI$ and use (\ref{iwa:dec}) we write $x = \ti_+ \h_{\O} \ti_- \in \tI$ with $\ti_+ \in \tI_{+}, \ti_- \in \tI_-$ and $\h_{\O} \in \tH_{\O}.$ Write $i_+, i_-$, and $h_{\O}$ for the images of these elements in $G$ under $p.$ By definition (recall the setup of \S \ref{sec:cover1}) we have $x.\os= (i_+ h_{\O} i_-, \h_{\O} \nt)$ with $\ti^-.\os= (i_-, \nt)$ we must have $i_-=h_{\O}=i_+=1$. If $p(x)=1$ then $i_+h_{\O}i_- =1$ which forces $i_+=i_-=h_{\O}=1$. As we have already seen that $p$ splits over both $U^-$ (and hence $\tI_-$) and $\tH_{\O}$, we must have $\tn=1$ and $\h_{\O}=1.$ Thus $x.\os=\os$ and by the simple transitivity $x=1.$ 
\end{proof}  

Let $\tK \subset \tG$ be the subgroup generated by elements $\xi_a(s)$ with $s \in \O,$ $a \in R_{re}$. Note that for each $w \in W$ we have $\tw \in K:$ indeed, this follows from the simple transitivity (cf. Corollary \ref{E:st}) and the fact that for $a \in \Pi,$ we have  \be{twa:K} \xi_{a}(-1) \xi_{-a}(1) \xi_{a}(-1) \os = (\dw_a, \tw_a). \ee The last equation is verified directly from the definitions. Now one can show (again as in \cite{iwa:mat}) \be{K:I} \tK = \bigsqcup_{w \in W} \tI \, \tw \, \tI. \ee Moreover, $p(\tI \, \tw \, \tI) =I \, \dw \, I$. Thus if $x \in \tK$ is such that $p(x)=1$ then $x \in \tI$ and, from the previous paragraph, in fact $x=1$. Thus $p: \tK \rr K$ is an isomorphism. Similarly we can define the Iwahori subgroup $\tI^-$ and decompositions similar to (\ref{iwa-mat:K}) follow immediately. 

\tpoint{Iwasawa decomposition} The Iwasawa decomposition for $G$ is well-known (cf. \cite[\S 3.2]{bkp} for the affine case; the statement is similar for general Kac-Moody groups).  As for $\tG,$ we have the following. 

\begin{nprop} \label{iwa:dec} Every $g \in \tG$ has an Iwasawa decomposition $g = k a u$ with $k \in \tK, \, u \in \tU$, and $a \in \tH$ as well an opposite Iwasawa decomposition $g = k a' u^-$ with $k \in \tK, u \in \tU^-$ and $a' \in \tH.$ Moreover the classes of $a$ and $a'$ in $\tH_{\O} \setminus \tH $ are unique in any such decomposition. \end{nprop}

The existence of these decompositions follows immediately from the Bruhat decompositions (\ref{G:w}) together with the classical argument of Steinberg \cite[Ch. 8]{stein:chev} which relies on the Tits axioms (i.e. the ordinary $BN$-pair axioms) for $(\tB, \tN)$ and $(\tB^-, \tN)$  as well as a rank one result of the following form: for each $a \in \Pi$, there exists a set $Y_a \subset \tK$ such that $\tB \setminus \tB \, \w_a \, \tB = \tB \setminus \tB Y_a.$ A similar approach establishes the Iwasawa decomposition with respect to the opposite Borel, and uniqueness in both cases follows along the same lines as in the classical case (note that we have $\tK \cap \tH = \tH_{\O}$). 

We now wish to establish some further notation for these decompositions: for $g \in \tG$ written (non-uniquely) as $g = k a u$, we denote the class of $a$ in $\tH_{\O} \setminus \tH$ by $\iw_{\tA}(g).$ Using (\ref{tH:Lv}), we set \be{log:cen:g} \begin{array}{lcr} \ln (g):= \ln(\iw_{\tA}(g)) \in \Lv & \text{ and } & \z(g) :=  \z(\iw_{\tA}(g)) \in \bmu \end{array} \ee so that image of $\iw_{\tA}(g)$ in $\Lv \times \bmu_n$ is $(\ln(g), \z(g)).$  

\tpoint{Some finiteness results} We would next like to state the finiteness results which are necessary for the formulation of the Whittaker function.

\renewcommand{\th}{\widetilde{h}}

\begin{nthm} \label{thm:fin} The following sets are finite \begin{enumerate}[(1)] \item $K \setminus K h U \cap K h' U^-$ where $h, h' \in H$
\item $\tK \setminus \tK \, \th \, \tU \cap K \, \th' \, \tU^-$ where $\th, \th' \in \tH$ \end{enumerate} \end{nthm}

The first result was proven for (untwisted) affine Kac-Moody groups in \cite[Theorem 1.9]{bgkp} and more generally by H\'ebert in \cite[Theorem 5.6]{heb}. The second finiteness follows immediately from the first.

\subsection{Whittaker and Iwahori-Whittaker functions} \label{s:whit}

Most of this section is not absolutely necessary to understand the Casselman-Shalika formula presented in the next one. We include it to cast our results in a more familiar representation theoretic framework, though one needs Conjecture \ref{whit-dim-conj} to make the full link. In any case, we do not use this Conjecture in this work and the reader may prefer directly to skip to \S \ref{avg:op} where the actual definition of the Whittaker function used in this paper is defined.

\spoint \label{prin:ch} Recall that in \S \ref{s:GS} we have chosen an additive character $\psi: \mc{K} \rr \C^*$ with conductor $\O$ to define the Gauss sums in \S \ref{s:GS}. We extend $\psi$ to a \emph{principal} character $U^- \rr \C^*$ as follows: first for each $a \in \Pi$ we can use $\psi$ to define a map from $\tU_{-a} \rr \C^*, \, \xi_{-a}(s) \mapsto \psi(s)$; then, we take the product of these to define a map $\prod_{a \in \Pi} \tU_{-a} \rr \C^*$; and finally, we use the natural isomorphism $\tU^- / [\tU^-, \tU^-] \rr \prod_{a \in \Pi} \tU_{-a}$ to define the map $\psi$ on $U^-$.

\newcommand{\M}{M^{\epsilon}}
\spoint  \label{prin-ser} Throughout this section, we fix an embedding $\epsilon: \bmu \rr \C^*$. A function $f: \tG \rr X$ where $X$ is a $\C$-vector space will be called \emph{$\epsilon$-genuine} if \be{ep-gen} f (\zeta g) = \epsilon(\zeta) f(g) \text{ where } \zeta \in \bmu, g \in \tG. \ee Let $\M(\tG)$ be the vector space of $\epsilon$-genuine functions $f:\tG \rr \C$ such that \be{M:def} f( g u a_{\O}  ) = f(g) \text{ for } a_{\O} \in \tH_{\O}, u \in U, g \in \tG. \ee  The space $\M(\tG)$ carries an action of $\tG$ by left translation denoted $g.f(x) = f(g^{-1}x)$ with $g,\ x \in \tG$ and an action on the right by $\C[\tLv]$ written as follows: for $\mv \in \tLv$ and $f \in \M(\tG)$ 
\be{act:f} (e^{\mv} \circ f) (g)  = q^{  - \la \mv, \rho \ra} f(  g \pi^{\mv}). \ee The same expression will also be used when $\mv \in \Lv$ as well. One uses here the fact that $\mathtt{T}$ is an abelian subgroup of $\tH$. Note that these two actions commute. Also, for a subgroup $V \subset \tG$ we write $\M(\tG, V)$ for the subspace of functions which are left $V$-invariant.

\newcommand{\Ie}{\mathtt{I}^{\epsilon}}

\spoint Consider now the vector space $\Ie$ of $\epsilon$-genuine functions $F: \tG \rr \C[\Lv]$ satisfying the condition  
\be{prin:ind} F ( \zeta g  a_{\O} \pi^{\mv}  u )  = q^{ \la  \mv, \rho \ra}  \epsilon(\zeta) \, e^{ - \mv} \, F(g) \text { for } \zeta \in \bmu_n, a_{\O} \in \tH_{\O}, \mv \in \tLv, u \in U,  g \in \tG,  \ee where the action of $e^{-\mv}$ on $\C[\Lv]$ is by usual multiplication in the group algebra.   The group $\tG$ acts on $\Ie$ by left translation again. Define the action $\circ$ of $\C[\tLv]$ on $\Ie$ as \be{act:I} \left( e^{\xv} \circ F \right) (g) := q^{- \la \xv, \rho \ra} F(  g \pi^{\xv}) \text{ for } F \in \Ie, \xv \in \tLv. \ee Again the same formula will be used for all $\xv \in \Lv,$ and for $f \in \M(\tG),$ we define an element in $\Ie$ as 
\be{map:real} f \mapsto F_{f}(g):= \sum_{\mv \in \Lv} f( g \pi^{\mv}) e^{\mv} q^{ - \la \mv ,\rho \ra} = \sum_{\mv \in \Lv} (e^{\mv} \circ f)(g) [e^{\mv} ]. \ee This is an isomorphism of $(\tG, \C[\tLv])$-modules which we denote by $\Phi: \M \rr \Ie, \, f \mapsto F_f$.

\newcommand{\one}{\mathbf{v}}

\spoint Consider the case when $V=\tK$, i.e. $\M(\tG, \tK).$ By the Iwasawa decomposition (Proposition \ref{iwa:dec}), every $\epsilon$-genuine function $f: \tG \rr \C$ can be written in the form \be{f:K} f = \sum_{\lv \in \Lv} \, c_{\lv} \one_{\tK, \lv} \text{ with } c_{\lv} \in \C \ee and where we adopt the notation (cf. (\ref{log:cen:g}) ) that for $g \in \tG$  \be{one:K} \one_{\tK, \lv}(g) = \begin{cases} \epsilon(\z(g)) & \text{ if } \ln(g) = \lv ; \\ 0 & \text{ otherwise. } \end{cases} \ee For a function as in (\ref{f:K}), we define its support $\Supp(f):= \{ \lv \in \Lv \mid c_{\lv} \neq 0 \}$ and we let  $\M_{\leq}(\tG, \tK)$ be the space of all $f \in \M(\tG, \tK)$ such that $\Supp(f)$ satisfies a condition as in Definition \ref{loo-space}. Denote by $\unp \subset \Ie$ the image under the map (\ref{map:real}) of $\M_{\leq}(\tG, \tK)$. If $\lv=0$ we shall usually just write $\one_{\tK}$ in place of $\one_{\tK, 0}$ and call this the \emph{spherical vector}. Finally let us record here  \be{trans:sph} e^{\xv} \circ \one_{\tK, \lv} = q^{ -\la \rho, \xv \ra} \one_{\tK, \lv - \xv} \text{ for } \xv \in \tLv. \ee 

\spoint Consider now the case when $V= \tI$ or $V=\tI^-$, i.e. $\M(\tG, \tI)$ or $\M(\tG, \tI^-)$. Let us consider the group $\aw:= W \ltimes \Lv$ whose elements are written as pairs $(w, \lv)$. For such an $x \in \aw$ we define the element $\wt{x}:= \tw \, \pi^{\lv} $ which lies in $\tG.$\footnote{This is the ``affine'' Weyl group of \cite{bkp}, which plays a role analogous to the usual affine Weyl group in the theory of $p$-adic groups.} A combination of the Iwasawa and Iwahori-Matsumoto decompositions (cf. \ref{iwa-mat:K}) shows that for each $g \in \tG$, there exist unique elements $\zeta \in \bmu_n$ and $x \in \aw$ such that $g \in \tI \, \zeta \, \wt{x} \, U$. For $y \in \aw,$ we then define \be{def:vw}\ve_{\tI, y}(g) &=& \begin{cases} 0 & \text{ if } y \neq x; \\ \epsilon(\zeta) & \text{ if } y=x. \end{cases} \ee In a similar way we can define the functions $\ve_{\tI^-,y}$ with respect to the opposite Iwahori subgroup $\tI^-$. One then notes that every function on $\M(\tG, I^{\pm})$ can be written as a (possibly infinite) linear combination of the functions $\tve_x^{\pm}.$ Also, note the equality of functions \be{sph:dec} \one_{\tK} = \sum_{w \in W} \ve_{\tI, w} = \sum_{w \in W} \ve_{\tI^-,w}, \ee and that if $\xv \in \tLv$ then we have $e^{\xv} \circ \ve_{\tI, x} = q^{- \la \rho, \xv \ra}  \ve_{\tI, x- \xv}$ and similarly for the $\ve_{\tI^-, x}$. 

\newcommand{\wfun}{\mc{W}}

\tpoint{Whittaker functionals} \label{s:whit-fun} 
Let $\psi$ be a principal character on $U^-$ as in \S \ref{prin:ch}. We define a $V$-Whittaker \emph{functional} to be a map of vector spaces $L: \M(\tG, V) \rr \C$ satisfying \be{whit-fun} L(n^-.f) = L(f) \psi(n^-) \text{ for } n^- \in U^-,\ f \in \M(\tG, V). \ee Let $\whi_V$ be the vector-space of all such functionals. If $V$ is the identity subgroup, we often drop it from our notation and just write $\whi$ in this case. If $V'  \subset V$ then we have a natural map $\whi_V \rr \whi_{V'}.$   Letting $\C_{\psi}$ be the one-dimensional module on which $U^-$ acts via the character $\psi$, we see that $\whi_V= \Hom_{U^-}(\M(\tG, V), \C_{\psi})$. 

We shall prefer to work with a formal or generic version in the sequel, so let $\C_{\psi}[\Lv]$ to be the module on which $U^-$ acts via the character $\psi$, i.e. $n^-. e^{\mv} = \psi(n^-) e^{\mv}$ for $n^- \in U^-$ and $\mv \in \Lv.$ Given $L \in \whi_V,$ we then define an element $\mathtt{L}$ in $\whit_V:=\Hom_{U^-}(\Ie, \C_{\psi}[\Lv])$ by the formula (cf. \ref{map:real}) \be{bbL} \mathtt{L} (F_f) = \sum_{\mv \in \Lv} L(e^{\mv} f) [e^{\mv}]. \ee

\tpoint{Whittaker functions} Associated to $L \in \whi_V$ and $f \in \M(\tG, V)$ we obtain a corresponding unramified \emph{Whittaker function} as $\wf_{L, f}(g):= L(g. f)$ for $g \in \tG.$ Note that we have \be{whit-fun-trans} \wf_{L, f} (v \, g \, n^-) = \psi(n^-) \, L (f) (g) \text{ for }  g \in \tG, n^- \in U^-,\ v \in V. \ee This gives the following result (using the same argument as in the finite dimensional case).
\begin{nlem} \label{whit-vanish} Let $f \in \M(\tG, \tK)$ and $L\in \whi_{\tK}$ . Then $\wf_{L, f}(g)$ is determined by its values on the elements $g=\pi^{\lv}$ with $\lv \in \Lv_+$  using the relation (\ref{whit-fun-trans}), i.e. $\wf_{L, f}(\pi^{\mv})=0$ unless $\mv \in \Lv_+.$ \end{nlem} 

 Let $L \in \whi$ (which then induces elements $L \in \whi_V$ for any $V$) and  $f \in \M(\tG, V).$ Then we have the formal version of the above construction $\wfun_{L,f}: \tG \rr \C[\Lv]$:  
\be{gen-fun} \wfun_{L, f} (g) := \mathtt{L}(F_{g.f}) = \sum_{\mv \in \Lv} L( e^{\mv}  (g.f) ) [e^{\mv} ].\ee It is again right $(U^-, \psi)$-invariant and left $V$-invariant. Also note there exists an action of $\C[\tLv]$  on $\whi$:  \be{act-on-fun} (e^{\mv} \circ L) ( f ) = L ( e^{\mv} \circ f) \text{ for } \mv \in \tLv, \, f \in \M(\tG). \ee We shall write $L_{\mv}:= e^{\mv} \circ L.$ Actually, for any $\mv \in \Lv$ (not necessarily in $\tLv$) we can still define a new Whittaker functional by the formula (\ref{act-on-fun}). In this notation, one finds \be{L:prin} \wfun_{L, f}(g) = \sum_{\mv \in \Lv} L_{\mv}(g.f) \, [e^{\mv}]. \ee

\begin{de} For $L \in \whi$ we define the \emph{$L$-spherical Whittaker function}  and \emph{$L$-Iwahori-Whittaker} function as \be{Lsph-fun}\begin{array}{lcr}  \wfun_{L, \one_{\tK}}: \tG \rr \C[\Lv]\, & \text{ and } &  \wfun_{L, \tve^-_w}:  \tG \rr \C[\Lv]. \end{array} \ee \end{de} \noindent Note that we have the equality of functions $\tG \rr \C[\Lv]$ \be{Iwa-Whit} \wfun_{L, \one_{\tK}} (g) = \sum_{w \in W} \wfun_{L, \tve^-_w}(g). \ee

\tpoint{Averaging operators} \label{avg:op} Now we want to construct the actual object which we compute a Casselman-Shalika formula for in this paper. We refer to \cite{pat:whit} for more details.  If $\Gamma$ is any group and $X$ is a right $\Gamma$-set and $Y$ is a left $\Gamma$-set, we define \be{fiber} X \times_{\Gamma} Y =  X \times Y  / \sim \ee where $\sim$ is the equivalence relation generated by $(x\gamma, \gamma^{-1} y) \sim (x, y)$ for $x \in X, y \in Y,$ and $\gamma \in \Gamma.$ For each $w \in W, \lv \in \Lv$ multiplication induces a natural map \be{m_w} m_{w, \lv}: \bmu_n U \, \tw \, \tI^- \times_{\tI^-} \tI^- \pi^{\lv} U^- \rr \tG. \ee If $\lv \in \Lv_+$ and $\mv \in \Lv,$ then for each $x \in m_{w, \lv}^{-1}(\pi^{\mv}),$ there are natural projections  \be{proj-fib} \begin{array}{lcr} \nn: m_{w, \lv}^{-1}(\pi^{\mv}) \rr (U^- \cap \tK) \setminus U^- & \text{ and } &  \z_w: m_{w, \lv}^{-1}(\pi^{\mv}) \rr \bmu_n \end{array} \ee defined as follows. Let $x \in m_{w, \lv}^{-1}(\pi^{\mv})$ have representative $(a, b)$ with $a \in \bmu_n U \tw \tI^-$ and $b \in \tI^- \pi^{\lv} U^-$ some chosen representatives. Writing $b = i \pi^{\lv} u^-$ with $i \in \tI^-,$ $u^- \in U^-$ one can verify that since $\lv \in \Lv_+$ the class of $u^- \in U^-_{\O} \setminus U^-$ is well-defined (cf. \cite[\S 4.2]{pat:whit}), and will be denoted as $\nn(x).$  Writing $a= \zeta u \tw i$ with $\zeta \in \bmu_n,$ $u \in U^-,$ $i \in \tI^-,$ the element $\z_w(x):= \zeta$ is well-defined and independent of the representatives $(a, b)$ taken for $x.$ Using these maps, we define for each $\mv \in \Lv$ and $\lv \in \Lv_+$ the averaging operator\footnote{Note that we should probably index the operators below by the inverse of $\ve_{\tI^-, w}$ and $\tve_{\tK}$.}  \be{Avg:mu} \Av_{\mv}(\ve_{\tI^-, w})(\pi^{\lv}):= \sum_{x \in m_{w, \lv}^{-1}(\pi^{\mv}) } \psi(\nn(x)) \iota(\z_w(x)). \ee Using the multiplication map $m_{\one_{\tK, \lv}}: \bmu_n U \tK \times_{\tK} \tK \pi^{\lv} U^- \rr \tG,$  we set \be{Avg:K} \Av_{\mv}(\one_{\tK})(\pi^{\lv}):= \sum_{x \in m_{\tK, \lv}^{-1}(\pi^{\mv}) } \psi(\nn(x)) \iota(\z(x)) \text{ for } \lv \in \Lv_+, \mv \in \Lv \ee and where $\nn(x) \in U^-_{\O} \setminus U^-$ and $\z(x) \in \bmu_n$ are defined in a similar manner to the above. The above constructions (\ref{Avg:K}) and (\ref{Avg:mu}) are in fact well-defined (i.e. the fibers of $m_{w, \lv}$ and $m_{\tK, \lv}$ are finite) by virtue of Theorem \ref{thm:fin}(2) as well as the decomposition \cite[Lemma 4.4]{pat:whit} (the proof given there is written for loop groups, but works in general).  Setting \be{avg-gen-funs} \begin{array}{lcr} \W(\pi^{\lv}) = \sum_{\mv \in \Lv} \Av_{\mv}(\one_{\tK})(\pi^{\mv}) & \text{ and } & \W_{w, \lv} = \sum_{\mv \in \Lv} \Av_{\mv}(\tve_w)(\pi^{\lv}) \text{ for } \lv \in \Lv_+, w \in W, \end{array} \ee one can verify as in \emph{loc. cit} that \be{W:w} \W(\pi^{\lv}) = \sum_{w \in W} \W_{w, \lv} \ee and in fact both sides are well defined in $\C_{\leq}[\Lv].$  In what follows we shall refer to the left hand side as \emph{the} Whittaker function, and the summands in the right hand side as \emph{the} Iwahori-Whittaker functions.   Note that in defining these functions, we have not explicitly mentioned any Whittaker \emph{functionals}. So to explain the connection with the constructions from the previous paragraph, we propose.
\begin{nconj} \label{whit-dim-conj} In terms of the action (\ref{act-on-fun}), we have the following. \begin{enumerate} \item The space of Whittaker functionals $\whi$ is a free module of rank $\# \,  \Lv / \tLv$ over $\C[\tLv].$  \item  \label{whit-dim-conj:1} If $L \in \whi_{\tK}$ then there exists $\mv \in \Lv$ such that  for all $\lv \in \Lv_+, $\be{prec-conj} \wfun_{L_{\mv}, \one_{\tK}}(\pi^{\lv}) = \mc{W}(\pi^{\lv}). \ee  \end{enumerate} \end{nconj}

\subsection{The Casselman-Shalika formula} \label{s:CS}

\tpoint{Recursions for Iwahori-Whittaker functions and metaplectic Demazure-Lusztig Operators}  Recall that for each $w \in W$ and $\lv \in \Lv$ we have defined an element $\Tm_w(e^{\lv}) \in \C^{\fin}_{v, \gf}[\Lv].$ In other words for each $\mv \in \Lv$ there exists a polynomials $\Upsilon^{\lv}_{w, \mv}(v, \gf_i) \in \C[v, v^{-1}, \gf_i]$\footnote{In fact, it lies in $\C[v, \gf_0, \ldots, \gf_{n-1}].$ by Remark \ref{T:polv}, which is why we drop the $v^{-1}$ from our notation} such that we can write \be{T:Up}  \Tm_w(e^{\lv}) = \sum_{\mv \in \Lv} \Upsilon^{\lv}_{w, \mv}(v, \gf_i) \, e^{\mv}. \ee Then the same proof which gives \cite[Corollary 5.4]{pat:pus} shows

\begin{nprop} \label{DL-whit} For any $w \in W$ and $\lv \in \Lv_+$ the Iwahori-Whittaker function $\W_{w, \lv}$ (cf. \ref{avg-gen-funs}) is the $p$-adic specialization (cf. \S \ref{p-spe}) of  $v^{   \la \lv, \rho \ra} \Tm_w(e^{\lv}).$ Informally, we write this as \be{Tw:spec} 
\W_{w, \lv} = q^{ - \la \lv, \rho \ra} \sum_{\mv \in \Lv} \Upsilon^{\lv}_{w, \mv}(q^{-1}, \g_i) e^{\mv}.
 \ee \end{nprop} For an interpretation of the above statement in the language of Whittaker functionals and a link to the work of Kazhdan-Patterson \cite{ka:pa}, we refer to \cite[\S 7.5]{pat:pus}.


\tpoint{Casselman-Shalika formula} Finally we state and prove the Casselman-Shalika formula.

\begin{nthm} \label{main-cs} For each $\lv \in \Lv_+,$ $\W(\pi^{\lv})$ is the $p$-adic specialization (cf. \S \ref{p-spe}) of the expression \be{main-cs-1} 
v^{ \la \lv, \rho \ra} \, \mf{m} \,    \wt{\Delta} \,  \sum_{w \in W} (-1)^{\ell(w)} \left( \prod_{a \in R(w)}  e^{-\tav} \right) w \star e^{\lv}. \ee
 \end{nthm}

\begin{proof}  Recalling the polynomials $\Upsilon^{\lv}_{w, \mv}(v, \gf_i)$ from (\ref{T:Up}) and defining \be{Up:all} \Upsilon^{\lv}_{\mv}(v, \gf_i) := \sum_{w \in W} \Upsilon^{\lv}_{w, \mv}(v, \gf_i) \ee it follows form Corollary \ref{fin-comb} that each $\Upsilon^{\lv}_{\mv}(v, \gf_i) \in \C[v, v^{-1}, \gf_i]$\footnote{Again, it can be written as a polynomial in $v, \gf_0, \ldots, \gf_{n-1}$.} and also that \be{Psym:Up} [e^{\mv}] \sum_{w \in W} \Tm_w(e^{\lv}) = \Upsilon^{\lv}_{\mv}(v, \gf_i). \ee Thus again invoking Corollary \ref{fin-comb}, the expression (\ref{main-cs-1}) has a well-defined $p$-adic specialization. The Theorem follows by applying Proposition \ref{DL-whit} and (\ref{W:w}).
\end{proof}


\end{document}